\documentclass[10pt]{article}
\usepackage[paperwidth=225mm, paperheight=310mm]{geometry}
\usepackage{graphicx}
\usepackage{amssymb}
\usepackage{mathtools}
\usepackage{amsmath}
\usepackage{hyperref}
\usepackage{enumerate}
\usepackage{amsthm}
\usepackage{amsmath}
\usepackage{float}
\usepackage{color}   
\usepackage{lineno}

\makeatletter
\renewenvironment{proof}[1][\proofname]{%
   \par\pushQED{\qed}\normalfont%
   \topsep6\p@\@plus6\p@\relax
   \trivlist\item[\hskip\labelsep\bfseries#1\@addpunct{.}]%
   \ignorespaces
}{%
   \popQED\endtrivlist\@endpefalse
}
\makeatother

\newtheorem{theorem}{Theorem}
\newtheorem{proposition}[theorem]{Proposition}
 \numberwithin{theorem}{section}
 
\newtheorem{lemma}[theorem]{Lemma}

\newtheorem{remark}[theorem]{Remark}

\newtheorem{conjecture}[theorem]{Conjecture}
\numberwithin{equation}{section}

\newcommand{\la}{\langle}
\newcommand{\ra}{\rangle}
\renewcommand{\P}{\mathbb{P}}
\newcommand{\E}{\mathbb{E}}
\newcommand{\R}{\mathbb{R}}

\newcommand{\Z}{\mathbb{Z}}
\newcommand{\N}{\mathbb{N}}

\newcommand{\qN}{\mathbb{N}}

\newcommand{\cA}{\mathcal A}

\newcommand{\cR}{\mathcal{R}}
\newcommand{\cN}{\mathcal{N}}

\newcommand{\cH}{\mathcal{H}}

\newcommand{\cI}{\mathcal{I}}

\newcommand{\cL}{\mathcal{L}}

\newcommand{\cF}{\mathcal F}

\newcommand{\eps}{\varepsilon}

 \newcommand{\nn}{\nonumber}
 \newcommand{\no}{\noindent}

\makeatletter
\def\keywords{\xdef\@thefnmark{}\@footnotetext}
\makeatother

\begin{document}


\author{
Jieliang Hong
}
\title{Rescaled SIR epidemic processes converge to super-Brownian motion in four or more dimensions}

\date{{\small  {\it  Department of Mathematics, Southern University of Science and Technology,\\
 Shenzhen, China\\
 E-mail:  {\tt hongjl@sustech.edu.cn} 
   }
  }
  }

\maketitle
\begin{abstract}
In dimensions $d\geq 4$, by choosing a suitable scaling parameter, we show that the rescaled spatial SIR epidemic process converges to a super-Brownian motion with drift, thus complementing the previous results by Lalley \cite{L09} and Lalley-Zheng \cite{LZ10} on the convergence of SIR epidemics in $d\leq 3$.  The scaling parameters we choose also agree with the corresponding asymptotics for the critical probability $p_c$ of the range-$R$ bond percolation on $\Z^d$ as $R\to \infty$.
\end{abstract}

\tableofcontents

\section{Introduction}

The celebrated works of Durrett-Perkins \cite{DP99} and Mueller-Tribe \cite{MT95} prove that the rescaled long range contact processes converge to a super-Brownian motion with drift in $d\geq 2$; to a super-Brownian motion with killing of the density in $d=1$. In the contact process setting, all individuals are either susceptible or infected, meaning that the infections will not give immunity upon recovery. By introducing a natural third type of individuals, recovered, to this model, we get the well-known susceptible/infected/recovered epidemic process (SIR), in which recovered individuals are immune to future infections. Lalley \cite{L09} and Lalley-Zheng \cite{LZ10} established the convergence of rescaled SIR epidemic processes for $d \leq 3$, but there are no such convergence results in higher dimensions $d\geq 4$. The results in the paper will then close the hole.

Define $\Z_u^d=\Z^d/u=\{x/u: x\in \Z^d\}$ for any $u>0$. Our SIR epidemic process takes place on $\Z^d_R$ where $R\in \N$ represents the infection range. Call $x,y\in \Z^d_R$ neighbors if $0<\| x-y\| _\infty\leq 1$ where $\| \cdot \| _\infty$ denotes the $l^\infty$ norm on $\R^d$. The state of the epidemic process at time $t$ can be characterized by a function $\psi_t: \Z_R^d\to \{-1,0,1\}$. For each vertex $x\in \Z^d_R$, we say $x$ at time $t$ is infected if $\psi_t(x)=1$; susceptible if $\psi_t(x)=0$; recovered if $\psi_t(x)=-1$.  Define the continuous time SIR epidemic process $(\psi_t, t\geq 0)$ as follows:\\
 (a) Each particle at infected sites dies at rate $1$ and gives birth to one new particle at rate $\beta$.\\
 (b) When a birth event occurs at some site $x$, the new particle is sent to a site $y$ chosen randomly from $\cN(x):=\{y\in \Z_R^d: 0<\| x-y\| _\infty\leq 1\}$, the set of neighbors of $x$. \\
 (c) If $y$ is susceptible, the birth is valid, the new particle establishes there and $y$ becomes infected. If $y$ is recovered or infected, such a birth will be suppressed.\\
  (d) When a death occurs at $x$, the site $x$ becomes recovered and refrains from establishing new particles.
 
One may think of the particles as the virus spreading over individuals positioned on $\Z_R^d$. The infected individual at $x$ recovers at rate $1$ (the virus is removed and the particle dies) while he transmits the disease (gives birth to a new particle)  with rate $\beta$ to his neighbors in $ \cN(x)$. If $x$ attempts to infect his neighbor $y$ who is susceptible, then $y$ will become infected and a new particle (the virus) establishes itself there; otherwise the attempt of infection is invalid. Once the individual is recovered, he will be immune to future infections. We note the multi-occupancy of particles is not allowed, i.e. there is at most one particle at each site. In fact, if the assumption on the long-lasting immunity upon recovery is removed, one obtains immediately the contact process model considered in \cite{DP99}.  Our $d\geq 4$ setting for the SIR epidemic here corresponds to their $d\geq 2$ setting for the contact process in \cite{DP99}. These dimension settings are essentially related to the fact that a super-Brownian motion has no fixed-time density iff $d\geq 2$ and no occupation density, or local time, iff $d\geq 4$ (see, e.g., Chp. III of \cite{Per02}). 

Notice that by ignoring rules $(c)$ and $(d)$, the SIR epidemic process is identical to a branching random walk (BRW) where each particle dies at rate $1$ and gives birth to one offspring to its neighboring positions at rate $\beta$. To obtain a nontrivial scaling limit, we set $\beta=1+\theta/N$ where $N>0$ satisfies 
\begin{align}\label{9e1.1}
R^d=
\begin{cases}
N,&d\geq 5\\
N\log N, &d=4\\
N^{3-d/2}, &d\leq 3.
\end{cases}
\end{align}
To justify such a choice of $N$, by following the contact process arguments in \cite{DP99}, we consider a typical particle in the branching random walk at time $N$ and count the number, denoted by $\chi_N$, of its neighboring sites that have been visited before time $N$. The expected value of $\chi_N$ is approximately 
\begin{align*}
 \E(\chi_N)\approx C\int_1^N \int_0^{t} (t+s)^{-d/2}ds dt=
\begin{cases}
C,&d\geq 5\\
C\log N, &d=4\\
CN^{2-d/2}, &d\leq 3.
\end{cases}
\end{align*}
Here we trace backward with time $t$ when the nearby relative branches off the family tree of the typical particle and $0< s< t$ is the extra time the relative takes before it reaches the neighboring site, the probability of which is at most $C(t+s)^{-d/2}$. To ensure that the epidemic survives, we require that the fraction of the available sites for birth in each neighborhood is at least $1-\theta/N$, so we need $\E(\chi_N)/R^d=O(1/N)$, giving
$N$ as in \eqref{9e1.1}.

{\bf Connection with the range-$R$ bond percolation}. The above asymptotics for $N$ as in \eqref{9e1.1} also appear in the corresponding asymptotics for the critical probability $p_c(R)$ of the range-$R$ bond percolation on $\Z^d$ as $R\to \infty$: In higher dimensions $d>6$, this asymptotics had been confirmed in Van der Hofstad and Sakai \cite{HS05} by using the lace expansion. When $d\leq 6$, the corresponding bounds for $p_c(R)$ were proved recently in \cite{FP16}, \cite{Hong21} and \cite{Hong23} where the authors used the connection between the bond percolation and the discrete-time SIR epidemics to study the critical probability $p_c(R)$. Such a connection is another motivation for us to work on the convergence of SIR epidemics in this paper.

Now that the scaling parameter is determined, we apply the usual Brownian scaling to scale time by $N$, space by $\sqrt{N}$, and define the rescaled SIR epidemic process $(\psi_t, t\geq 0)$ as follows:\\
 (a) Each particle at infected sites dies at rate $N$ and gives birth to one new particle at rate $N+\theta$.\\
 (b) When a birth occurs at $x$, the new particle is sent to a site $y$ chosen randomly from $y\in \Z^d_{N^{1/2} R}$ with $0<\| y-x\| _\infty\leq N^{-1/2}$. \\
 (c) If $y$ is susceptible, the birth is valid, the new particle establishes there and $y$ becomes infected. If $y$ is recovered or infected, the birth will be suppressed.\\
  (d) When a death occurs at $x$, the site $x$ becomes recovered and refrains from establishing new particles.

To describe the convergence result, we introduce some notation. Let $M_F(\R^d)$ be the space of finite measures on $(\R^d,\mathfrak{B}(\R^d))$ equipped with the topology of weak convergence of measures. Write $$\mu(\phi)=\int \phi(x) \mu(dx)$$ for any measure $\mu$ and any integrable function $\phi$. Set $\Omega_X=D([0,\infty), M_F(\R^d))$ to be the Skorohod space of c\`adl\`ag paths on $M_F(\R^d)$ with the Skorohod topology, the space where the convergence will occur. For each $n\geq 1$, let $C_b^n(\R^d)$ denote the space of bounded continuous functions whose partial derivatives of order less than $n+1$ are also bounded and continuous. 
A super-Brownian motion $X=(X_t,t\geq 0)$ is a continuous $M_F(\R^d)$-valued strong Markov process defined on some complete filtered probability space  $(\Omega, \cF, \cF_t, \P)$. We say $X$ starts at $X_0\in M_F(\R^d)$ with branching rate $\gamma_0>0$, diffusion coefficient $\sigma_0^2>0$ and drift $\theta_0\in \R$ if it satisfies the following martingale problem:
\begin{align}\label{9e10.25}
(MP)^{\gamma_0, \sigma_0^2, \theta_0}_{X_0}: &\text{ For any }  \phi \in C_b^2(\R^d),\nn\\
&M_t(\phi)=X_t(\phi)-X_0(\phi)-\int_0^t X_s\Big(\frac{\sigma_0^2}{2}  \Delta\phi\Big) ds-\theta_0\int_0^t X_s(\phi) ds\nn\\
&\text{ is an $(\cF_t)$-martingale with } \langle M(\phi)\rangle_t=\int_0^t X_s(\gamma_0 \phi^2)ds.
\end{align}
The above martingale problem uniquely characterizes the law of super-Brownian motion on $\Omega_{X, C}=C([0,\infty), M_F(\R^d))$, the space of continuous  $M_F(\R^d)$-valued paths furnished with the compact-open topology.

Now we are ready to present our main theorem. For simplicity, we assume that there are no recovered sites at the beginning of the epidemic. Let 
 \begin{align} \label{9e10.13}
\xi_t=\psi_t^{-1}(1):=\{x\in \Z^d_{N^{1/2} R}: \psi_t(x)=1\}
\end{align}
 denote the set of infected sites at time $t$. Then the set of recovered sites is given by  
 \begin{align*} 
 \psi_t^{-1}(-1):=\{x\in \Z^d_{N^{1/2} R}: \psi_t(x)=-1\}=\Big(\bigcup_{s\leq t} \xi_s\Big)-\xi_t.
\end{align*}
One may conclude from the above that the infected sets $\{\xi_t, t\geq 0\}$  uniquely determine the SIR epidemic process $(\psi_t, t\geq 0)$. By assigning mass $1/N$ to each infected individual, we define the measure-valued process by 
\begin{align}\label{9ec10.25}
X_t^N:=\frac{1}{N}\sum_{x\in\xi_t} \delta_x.
\end{align}
 It follows that $(X_t^N, t\geq 0)\in \Omega_X$.  When $d\geq 5$, we let $Y_1,Y_2, \cdots$ be i.i.d. random variables on $\R^d$ so that $Y_1=0$ or $Y_1$ is uniform on $[-1,1]^d$, each with probability $1/2$. Set $V_n=Y_1+\cdots+Y_n$ for $n\geq 1$ and $V_0=0$. In addition, we let $W_0$ be uniform on $[-1,1]^d$, independent of $\{Y_1, Y_2,\cdots\}$. Define
\begin{align}\label{9ec10.64}
b_d= 2^{-d} \sum_{l=0}^\infty  \P(V_{l+1}\in  [-1,1]^d \backslash \{0\}) +2^{-d} \sum_{l=0}^\infty \sum_{j=0}^\infty  \frac{1}{2^{l+j+1}} \frac{(l+j)!}{l!j!} \sum_{m=0}^j  \P(W_{0}+V_{l+m}\in [-1,1]^d).
\end{align}
 In $d=4$, we let $b_4={9}/(2\pi^2)$.
\begin{theorem}\label{9t0}
Let $d\geq 4$. If $X_0^N\to X_0$ in $M_F(\R^d)$ as $N\to \infty$ where $X_0$ has no point masses, then the rescaled SIR epidemic processes $(X_t^N, t\geq 0)$ defined in \eqref{9ec10.25} converge weakly to  $(X_t, t\geq 0)$ on $\Omega_X$ as $N\to \infty$ where $X$ is a super-Brownian motion satisfying $(MP)^{\gamma_0, \sigma_0^2, \theta_0}_{X_0}$ with branching rate $\gamma_0=2$, diffusion coefficient $\sigma_0^2=1/3$ and drift $\theta_0=\theta-b_d$.
\end{theorem}
\begin{remark}
A careful reader may have noticed that the measure-valued epidemic process $\{X_t^N\}_{t\geq 0}$ is not Markovian--the suppressed births onto the recovered sites depend on the history of the process. However, its weak limit, the super-Brownian motion $X$, is a Markov process. This may not make sense at first thinking; we will show later in the proof that the suppressed births due to collisions between distant relatives can be ignored, making $\{X_t^N\}_{t\geq 0}$ ``almost'' Markovian for $N$ large. Therefore when passing to the limit, we get a Markov process. We refer the reader to Section \ref{9s4} for more details. 
\end{remark}
Turning to the lower dimensions $d\leq 3$, the distant relatives may no longer be ignored as the occupation density exists, for which the scaling limit of the (discrete-time) spatial SIR epidemics has been studied in Lalley \cite{L09} and Lalley-Zheng \cite{LZ10}. Instead of the finer lattice, their epidemic model takes place on the integer lattice $\Z^d$ where a large population of size $N$ is located at each site,  the so-called ``Village'' model. It is shown that after suitable scaling, which agrees with our argument in \eqref{9e1.1}, the SIR epidemic process will converge to a super-Brownian motion with killing of the local time. More specifically, the limiting process $X$ is the  solution to the martingale problem
\begin{align}\label{9ec4.38}
X_t(\phi)=X_0(\phi)+M_t(\phi)+\int_0^t X_s\Big(\frac{\Delta}{6}\phi+\theta \phi\Big) ds-\int_0^t X_s(L_s \phi) ds, \forall \phi\in C_b^2(\R^d),
\end{align}
where $X$ is a continuous $M_F(\R^d)$-valued process, $L_s$ is the local time of $X$, i.e. the density function of the occupation measure $\int_0^s X_u(\cdot) du$, and $M(\phi)$ is a continuous martingale with $\langle M(\phi)\rangle_t=\int_0^t X_s(2\phi^2)ds$. We refer the reader to Theorem 2.2 of Lalley-Perkins-Zheng \cite{LPZ14} for the existence and uniqueness of the above martingale problem. Although the settings are slightly different, it has been conjectured in Frei-Perkins \cite{FP16} that the same conclusion holds for the discrete-time long-range SIR epidemics (see Conjecture 1.2 therein). We also conjecture that the same conclusion holds under our continuous-time setting.
\begin{conjecture}\label{c1.2}
Let $d\leq 3$. If $X_0^N$ converges to some appropriate compactly supported $X_0$ in $M_F(\R^d)$ as $N\to \infty$, then the rescaled SIR epidemic processes $(X_t^N, t\geq 0)$ defined in \eqref{9ec10.25} converge  weakly to  $(X_t, t\geq 0)$ on $\Omega_X$ as $N\to \infty$ where $X$ is as in \eqref{9ec4.38}.
\end{conjecture}

We briefly describe the idea for the proof of Theorem \ref{9t0} with $d\geq 4$, which was inspired by the contact process proof in \cite{DP99} along with some necessary adjustments: Define a sequence of approximating processes $X_t^{k, N}$ that gives the upper and lower bounds for the rescaled SIR epidemic process $X_t^N$. Recall from \eqref{9e10.13} that $\xi_t$ is the set of infected sites of the SIR epidemic process. First, we define $\xi_t^0$ to be a branching random walk by ignoring the collision rules $(c),(d)$, where the multi-occupancy of particles is allowed, and we may view $\xi_t^0$ as a set in which repetitions of elements are allowed. For $k\geq 1$,   let $\xi_t^k$ be the branching random walk $\xi_t^0$ that avoids births onto sites that have been visited by $\xi_t^{k-1}$. Since $\xi_t^0$ is an overestimate of the set of infected sites $\xi_t$, we get $\xi_t^1$ is an underestimate of $\xi_t$ as the set of forbidden sites is larger for $\xi_t^1$. On the other hand, $\xi_t^2$ will be an overestimate since it only removes particles that are born onto the smaller set visited by $\xi_t^1$. We may continue this procedure to define $\xi_t^k$ for $k\geq 3$ but we only need $k=0,1,2$ for our proof.  
 From these approximating branching random walks $\xi_t^k$ with $k=0,1,2$, we may define their rescaled versions by
 \begin{align*} 
X_t^{k,N}:=\frac{1}{N}\sum_{x\in \xi_t^k} \delta_x.
\end{align*}
 Write $\mu\geq \nu$ for measures $\mu, \nu$ if $\mu(A) \geq \nu(A)$ for any Borel set $A\subset \R^d$. 
 By construction we have 
  \begin{align}\label{9e1.7}
  X_t^{1,N}\leq X_t^N\leq X_t^{2,N}\leq X_t^{0,N}.
  \end{align}
 \begin{proposition}\label{9p1.1}
Let $d\geq 4$. Under the hypothesis of Theorem \ref{9t0},  we have for all $T>0$,
\begin{align*}
\E\Big(\sup_{t\leq T} \Big\vert X_t^{2,N}(1)-X_t^{1,N}(1)\Big\vert \Big)\to 0 \text{ as } N\to \infty.
\end{align*}
 \end{proposition}
Together with \eqref{9e1.7}, the above proposition implies that it suffices to find the limit of $X^1$ or $X^2$, so Theorem \ref{9t0} is immediate from the following result.
 \begin{proposition}\label{9p1.2}
 Let $d\geq 4$. Under the hypothesis of Theorem \ref{9t0},  we have $\{X_t^{1,N}\}_{t\geq 0}$ converges weakly to  $X$ on $\Omega_X$ as $N\to \infty$ where $X$ is a super-Brownian motion with branching rate $2$, diffusion coefficient $1/3$ and drift $\theta-b_d$.
 \end{proposition}

\section{Proof of the main result}\label{9s2}

Let $\theta\in \R$ and consider $N\in \qN$ such that $N>2\vert \theta\vert $.  For simplicity, we may set $\theta\geq 0$ as all the calculations below work similarly for $\theta<0$. Define the rescaled fine lattice where all our processes live on by
\begin{align}\label{9e1.3}
\cL_N=
\begin{cases}
N^{-1/2} \cdot N^{-1/d} \cdot \Z^d,&d\geq 5\\
N^{-1/2}\cdot  N^{-1/4}(\log N)^{-1/4} \cdot \Z^4, &d=4.
\end{cases}
\end{align}
The first scaling factor $N^{-1/2}$ is from the usual space-time scaling while the second scaling factor is from $\Z_R^d$ and the relation between $N$ and $R$ as in \eqref{9e1.1}.
We use the same labeling system from \cite{DP99} to study our SIR epidemic process. Define
\begin{align}\label{9e1.16}
\cI=\bigcup_{n=0}^\infty \qN \times \{0,1\}^n=\{(\beta_0, \beta_1, \cdots, \beta_n): \beta_0\in \qN, \beta_i \in \{0,1\}, 1\leq i\leq n\},
\end{align}
where $\beta_0$ labels the ancestor of a particle $\beta$. If $\beta=(\beta_0, \beta_1, \cdots, \beta_n)$ for some $n\geq 0$, then we let $\vert \beta\vert =n$ be the generation of $\beta$ and write
$\beta\vert i=(\beta_0, \cdots, \beta_i)$ for $0\leq i\leq n$. Let $\pi \beta=(\beta_0, \beta_1, \cdots, \beta_{n-1})$ be the parent of $\beta$ and set $\pi \beta=\emptyset$ if $\vert \beta\vert =0$. For $j=0$ or $1$,   denote by $\beta \vee j=(\beta_0, \beta_1, \cdots, \beta_n, j)$ the offspring of $\beta$. If $\gamma_0=\beta_0$,  use $\beta\wedge \gamma$ to denote the most recent common ancestor of $\beta$ and $\gamma$, that is, if we let $k_{max}=\max\{0\leq k\leq |\beta|\wedge |\gamma|: \beta|k=\gamma|k\}$, then $\gamma\wedge \beta=\beta|k_{max}=\gamma|k_{max}$. Set $\beta\wedge \gamma=\emptyset$ if $\gamma_0\neq \beta_0$. Write $ \beta\geq \gamma$ if $\beta$ is an offspring of $\gamma$ and $\beta>\gamma$ if it is strict.

The set of initially infected particles is given by $\{x_i: 1\leq i\leq M_N\} \subseteq \cL_N$. We further assume that $X_0^{N}=\frac{1}{N}\sum_{i=1}^{M_N} \delta_{x_i}$ converges to some $X_0\neq 0$ in $M_F(\R^d)$. For any $i>M_N$,   set $x_i$ to be the cemetery state $\Delta$. {\bf Unless otherwise noted, we only consider particles $\beta$ with $1\leq \beta_0\leq M_N$ below.} Let $K_0^N$ be a finite subset of $\cL_N$ disjoint from $\{x_i\}$, denoting the set of initial recovered sites at time $0$.

On some complete probability space $(\Omega, \cF, \P)$, we define the following independent collections of random variables:
\begin{itemize} 
\item $\{t_\beta: \beta \in \cI\}$ are i.i.d. with distribution $Exp(2N+\theta)$.
\item $\{\delta_\beta: \beta \in \cI\}$  are i.i.d. with  $\P(\delta_\beta=-1)= \frac{N}{2N+\theta}$, and $\P(\delta_\beta=1)= \frac{N+\theta}{2N+\theta}$.
\item $\{e_\beta: \beta \in \cI\}$  are i.i.d. with distribution  $\P(e_\beta=0)= \P(e_\beta=1)=\frac{1}{2}$.
\item $\{W^\beta: \beta \in \cI\}$ are i.i.d. that are uniform on $\cN_N=\{y\in \cL_N: 0<\| y\| _\infty\leq N^{-1/2}\}$.
\end{itemize}
Here we use $t_\beta$ to measure the time until a birth or death event occurs ($2N+\theta$ is the total rate). A death event occurs if $\delta_\beta=-1$ and a birth event if $\delta_\beta=1$.  When a birth event occurs, the new particle is labeled by $\beta \vee e_\beta$ and is displaced from its parent $\beta$ by an amount of $W^\beta$. Relabel the parent particle $\beta$ by $\beta \vee (1-e_\beta)$. We use $e_\beta$ to record the change of the family line of $\beta$. The lifetime of a particle $\beta$ is given by
\begin{align}\label{9ea6.95}
T_\beta=\sum_{m=0}^{\vert \beta\vert } t_{\beta\vert m}.
\end{align}
By convention, we let $T_\emptyset=-\infty$. 
The lifetime is not the death time of the particle--it might already be dead before $T_\beta$ if $\delta_{\beta\vert m}=-1$ for some $m<\vert \beta\vert $. We further define the death time of $\beta$ by
\begin{align}\label{e6.95}
\zeta_{\beta}^0=T_\beta \wedge \inf\{T_{\beta\vert m}: \delta_{\beta\vert m}=-1, m< \vert \beta\vert \},
\end{align}
where $\inf \emptyset =\infty$. Hence we get $\zeta_{\beta}^0=T_\beta$ if no death occurs along the family line of $\beta$.

Recall we use $\beta \vee e_\beta$ to label the new particle displaced from its parent $\beta$. The family line of $\beta$ is changing iff $e_{\beta\vert m}=\beta_{m+1}$, so the position of the family line of $\beta$ is given by 
\begin{align}\label{ea3.22}
B_t^\beta=
\begin{cases}
x_{\beta_0}+\sum_{m=0}^{\vert \beta\vert -1} W^{\beta\vert m} 1(e_{\beta\vert m}=\beta_{m+1}, T_{\beta\vert m}\leq t), &\text{ if } t<\zeta_\beta^0\\
\Delta, &\text{ if } t\geq \zeta_\beta^0.
\end{cases}
\end{align}
The current location, $B^\beta$,  of $\beta$ is given by
\begin{align*}
B^\beta=B_{T_\beta^{-}}^\beta=B_{T_{\pi \beta}}^\beta,
\end{align*}
where $T_\beta^{-}$ is the time infinitesimal close to and before $T_\beta$.  Define the $\sigma$-field by
\begin{align*}
\cF_t=\sigma\{1(T_\beta\leq t)(T_\beta, \delta_\beta, e_\beta, W_\beta): \beta\in \cI\}\vee \{\P-\text{ null sets}\}.
\end{align*} 
Write $\beta \sim t$ if $T_{\pi \beta}\leq t<T_\beta$, meaning the particle $\beta$ might be alive at time $t$. The particles that are dead will be identified by setting their location to be $\Delta$ and letting $\phi (\Delta)=0$ for any function $\phi$. Define our first measure-valued process, the branching random walk $X_t^0$, by
\begin{align}\label{9e9.01}
X_t^0(\phi)=X_t^{0,N}(\phi)=\frac{1}{N}\sum_{\beta \sim t} \phi(B_t^\beta).
\end{align}
In particular, when $t=0$, by assumption we get 
\begin{align*}
X_0^{0}=X_0^{0,N}=\frac{1}{N}\sum_{i=1}^{M_N} \delta_{x_i}\to X_0\in M_F(\R^d) \text{ as } N\to \infty,
\end{align*}
thus giving $X_0^0(1)$ is bounded for all $N$ (most of the time we will suppress the dependence of $X_t^{0, N}$ on $N$ and simply write $X_t^0$, but when necessary we will write $X_t^{0, N}$). The following is an elementary result on the first moment of the branching random walk from Lemma 2.9 of \cite{DP99}.

\begin{lemma}\label{9l2.0}
If $\phi: \R^d\to \R$ is a bounded Borel measurable function, then 
\begin{align*}
\E(X_t^0(\phi))=e^{\theta t} \int E_x^N(\phi(B_t^N)) X_0^0(dx),
\end{align*}
where $B_t^N$ is a continuous time random walk starting from $x$ under the law $P_x^N$ that takes a step uniformly over $\cN_N$ at rate $N+\theta$.
\end{lemma}

Now we turn to the SIR epidemic process. Denote by $\text{Supp}(\mu)$ the closed support of a measure $\mu$. Recall $K_0^N \subset \cL_N$ is the set of recovered sites at time $0$. If $t\to \mu_t$ is a cadlag measure-valued path, we define
\begin{align}\label{9e10.93}
\cR_t^\mu:=\bigcup_{s\leq t} \text{Supp}(\mu_s) \quad \text{ and }  \quad \overline{\cR}_t^\mu:=\cR_t^\mu\cup K_0^N.
\end{align}
Assume $\{K_0^N\}_{N\geq 1}$ satisfies
\begin{align}\label{9e0.93}
\lim_{N\to \infty} \frac{1}{N} \E\Big( \sum_{\beta} 1(T_\beta\leq t, B^\beta\neq \Delta)  1(B^\beta+W^\beta\in K_0^N)\Big)=0.
\end{align}
The above ensures that the suppressed births onto the initial recovered sites can be ignored, which of course holds if $K_0^N=\emptyset$ for all $N\geq 1$. We will show below that the above is also necessary to obtain the convergence result as in Theorem \ref{9t0}.
Define
\begin{align}\label{9ea3.01}
\zeta_\beta(\mu)=\zeta_\beta^0\wedge \inf \Big\{T_{\beta\vert m}: m<\vert \beta\vert , e_{\beta\vert m}=\beta_{m+1}, B_{T_{_{\beta\vert m}}}^\beta\in \overline{\cR}_{T_{\beta\vert m}^{-}}^\mu\Big\},
\end{align}
to characterize the first time the family line of $\beta$ hits a site that has already been occupied by $\mu$. Define our rescaled SIR epidemic by
\begin{align}\label{9ec3.01}
X_t(\phi)=X_t^N(\phi)=\frac{1}{N} \sum_{\beta \sim t} \phi(B_t^\beta) 1(\zeta_\beta(X)>t).
\end{align}
In the definition of $\zeta_\beta(X)$,  we use $\overline{\cR}_t^X$ to denote the set of the recovered and infected sites for each time $t$, which are exactly the forbidden locations to give births. Similar to (2.6) of \cite{DP99}, the existence and uniqueness of $X_t$ are trivial if the initial mass is finite. Following the discussion below Conjecture \ref{c1.2}, we define
  \begin{align}\label{9e0.03}
X_t^n(\phi)=X_t^{n,N}(\phi)=\frac{1}{N} \sum_{\beta \sim t} \phi(B_t^\beta) 1(\zeta_\beta^n>t)    \text{ where }  \zeta_\beta^n=\zeta_\beta({X^{n-1}})\text{ for } n=1, 2.
\end{align}
Recall $X_t^0$ from \eqref{9e9.01}. One may easily see that $X_t^0\geq X_t$ for any $t$ (recall $\mu\geq \nu$ for measures $\mu,\nu$ on $\R^d$ if $\mu(\phi)\geq \nu(\phi)$ for all functions $\phi\geq 0$). Thus it is immediate that $\overline{\cR}_t^X\subseteq \overline{\cR}_t^{X^0}$. It follows that $\zeta_\beta^1=\zeta_\beta({X^{0}})\leq \zeta_\beta(X)$ and hence $X_t^1\leq X_t$. Similar reasoning gives $X_t^2\geq X_t$. One may conclude
  \begin{align}\label{9e10.03}
X_t^1\leq X_t \leq X_t^2 \leq X_t^0, \quad \forall t\geq 0.
\end{align}
Let $n=1$ or $2$.
 Consider a particle $\beta$ born at time $T_{\pi \beta}$. In order that such a particle, $\beta$, is alive in $X^n$, we require $\zeta_\beta^n>T_{\pi\beta}$. At the end of its lifetime $T_\beta$,  there is either a death event or a birth event happening to $\beta$: on the event of a death, i.e. $\delta_\beta=-1$, the particle $\beta$ is lost; on the event of a birth, i.e. $\delta_\beta=1$, the particle at $B^\beta$ is relabeled and a new particle will be sent to its neighboring site with displacement $W^\beta$ as long as the site $B^\beta+W^\beta$ has never been occupied before by $X^{n-1}$. Collecting the information, we may write
\begin{align*}
&X_{T_\beta}^n(\phi)-X_{T_\beta^-}^n(\phi)\\
&= \frac{1}{N} 1_{\{\zeta_\beta^n>T_{\pi \beta}\}} \Big\{ -\phi(B^\beta)1_{\{\delta_\beta=-1\}} +  1_{\{\delta_\beta=1\}} \phi(B^\beta+W^\beta) 1(B^\beta+W^\beta\notin \overline{\cR}^{X^{n-1}}_{T_\beta^-})\Big\}\\
&=\frac{1}{N} 1_{\{\zeta_\beta^n>T_{\pi \beta}\}}\Bigg\{  \phi(B^\beta)\delta_\beta +1_{\{\delta_\beta=1\}} \times \\
&\quad \quad\quad\quad \quad\quad\Big[ \big(\phi(B^\beta+W^\beta) -\phi(B^\beta)\big)-\phi(B^\beta+W^\beta) 1(B^\beta+W^\beta\in \overline{\cR}^{X^{n-1}}_{T_\beta^-}) \Big]\Bigg\},
\end{align*}
where the second equality uses $\delta_\beta\in \{1,-1\}$. Summing the above over $\beta$ with $T_\beta\leq t$ and introducing
\begin{align}\label{9e0.94}
a_\beta^n(t)\equiv 1(T_\beta\leq t, \zeta_\beta^n>T_{\pi \beta})=1(T_\beta\leq t, \zeta_\beta^n\geq T_{\beta}),
\end{align}
 we get
\begin{align*}
X_{t}^n(\phi)=X_{0}^n(\phi)&+\frac{1}{N} \sum_{\beta} a_\beta^n(t)  \phi(B^\beta)\delta_\beta +\frac{1}{N} \sum_{\beta} a_\beta^n(t) 1_{\{\delta_\beta=1\}}\Big(\phi(B^\beta+W^\beta) -\phi(B^\beta)\Big) \\
&-\frac{1}{N} \sum_{\beta} a_\beta^n(t) 1_{\{\delta_\beta=1\}} \phi(B^\beta+W^\beta) 1(B^\beta+W^\beta\in \overline{\cR}^{X^{n-1}}_{T_\beta^-}).
 \end{align*}
Centering $\delta_\beta$ and $1_{\{\delta_\beta=1\}}$ by their expectations, we arrive at 
\begin{align}\label{9e2.1}
X_{t}^n(\phi)=X_{0}^n(\phi)&+\frac{1}{N} \sum_{\beta} a_\beta^n(t)  \phi(B^\beta) (\delta_\beta-\frac{\theta}{2N+\theta}) \\
&+\frac{\theta}{2N+\theta} \frac{1}{N} \sum_{\beta} a_\beta^n(t)  \phi(B^\beta)\nn\\
&+\frac{N+\theta}{2N+\theta} \frac{1}{N} \sum_{\beta} a_\beta^n(t)  \Big(\phi(B^\beta+W^\beta) -\phi(B^\beta)\Big)\nn \\
&-\frac{N+\theta}{2N+\theta} \frac{1}{N} \sum_{\beta} a_\beta^n(t)  \phi(B^\beta+W^\beta) 1(B^\beta+W^\beta\in \overline{\cR}^{X^{n-1}}_{T_\beta^-})\nn\\
 &+\frac{1}{N} \sum_{\beta} a_\beta^n(t)\Big(1_{\{\delta_\beta=1\}}-\frac{N+\theta}{2N+\theta}\Big)\times\nn\\
&\quad\quad\quad\Big(\phi(B^\beta+W^\beta)1(B^\beta+W^\beta\notin \overline{\cR}^{X^{n-1}}_{T_\beta^-}) -\phi(B^\beta) \Big).\nn
\end{align}
Rewrite the above by giving a notation for each term on the right-hand side to get
\begin{align}\label{9ea3.67}
X_{t}^n(\phi)&=X_{0}^n(\phi)+M_t^n(\phi)+D_t^{n,1}(\phi)+D_{t}^{n,2}(\phi)-K_{t}^{n}(\phi)+E^{t}_{n}(\phi),
\end{align}
where $M_t^n(\phi)$ represents the martingale term, $D_t^{n,1}(\phi)$ the drift term, $D_t^{n,2}(\phi)$ the diffusion term, $K_{t}^{n}(\phi)$ the collision term and $E_t^{n}(\phi)$ the error term. Set $\eps_N=\theta/(2N+\theta)$. Throughout the rest of the paper, we always pick 
\begin{align*}
\phi \in C_b^3(\R^d) \text{ and } \| \phi\| _\infty=\sup_{x\in \R^d} \vert \phi(x)\vert .
\end{align*}
\begin{lemma}\label{9l2.1}
Let $n=1$ or $2$. We have $M_t^n(\phi)$ is an $\cF_t$-martingale with 
\begin{align*}
\la M^n(\phi)\ra_t=\Big(2+\frac{\theta}{N}\Big)(1-\eps_N^2) \int_0^t X_r^n(\phi^2) dr.
\end{align*}
Moreover, 
\begin{align*}
\la M^2(\phi)-M^1(\phi)\ra_t=\Big(2+\frac{\theta}{N}\Big)(1-\eps_N^2) \int_0^t (X_r^2(\phi^2)-X_r^1(\phi^2)) dr.
\end{align*}
\end{lemma}

\begin{lemma}\label{9l2.2}
Let $n=1$ or $2$.  For all $t>0$, 
\begin{align*}
\lim_{N\to \infty} \E\Big(\sup_{s\leq t} \Big\vert D_s^{n,1}(\phi)-\theta \int_0^s X_r^n(\phi) dr\Big\vert \Big)=0.
\end{align*}
\end{lemma}

\begin{lemma}\label{9l2.3}
Let $n=1$ or $2$. For all $t>0$, 
\begin{align*}
\lim_{N\to \infty} \E\Big(\sup_{s\leq t} \Big\vert D_s^{n,2}(\phi)-  \int_0^s X_r^n(\frac{\Delta}{6}\phi) dr\Big\vert \Big)=0.
\end{align*}
\end{lemma}

\begin{lemma}\label{9l2.4}
Let $n=1$ or $2$. For all $t>0$, 
\begin{align*}
\lim_{N\to \infty} \E\Big(\sup_{s\leq t} \vert E_s^{n}(\phi)\vert \Big)=0.
\end{align*}
\end{lemma}

\no ${\bf Convention\ on\ Constants.}$ Constants whose value is unimportant and may change from line to line are denoted $C$. All these constants may depend on the dimension $d$, the drift $\theta$, the time $t$, and the test function $\phi$. All these parameters, $d,\theta, t, \phi$, will be fixed before picking $C$. Our constants $C$ will never depend on $N$ or the initial condition $X_0^0$.

\begin{lemma}\label{9l2.5}
For any $t>0$, there is some constant $C>0$ such that for $N$ large,
\begin{align*}
\E\Big(\frac{1}{N} \sum_{\beta} a_{\beta}^0(t) 1(B^\beta+W^\beta \in \overline{\cR}_{T_\beta^-}^{X^0})\Big)\leq C (X_0^0(1)+X_0^0(1)^2).
\end{align*}
\end{lemma}

\begin{lemma}\label{9l2.6}
Let $n=1$ or $2$. For all $0<t<\infty$, 
\begin{align*}
\lim_{N\to \infty} \E\Big(\sup_{s\leq t} \Big\vert K_s^{n}(\phi)-b_d \int_0^s X_r^1(\phi) dr\Big\vert \Big)=0.
\end{align*}
\end{lemma}

The proofs of Lemmas \ref{9l2.1}-\ref{9l2.4} follow similarly to those in \cite{DP99}, where the two authors claim them to be ``The four easy convergences''. The only difference lies in the definition of $\zeta_\beta^n$, which, fortunately, does not play an important role in the proof, so we omit the details (for interested readers, we have included a version of the proofs in Appendix \ref{9a1}). The proof of Lemma \ref{9l2.5} is given in Section \ref{9s3}. Proving Lemma \ref{9l2.6} will be the main part of the paper, which is done in Section \ref{9s5}.

Let $\phi\equiv 1$. Apply Lemmas \ref{9l2.1}-\ref{9l2.6} in \eqref{9ea3.67} and collect all the error terms to get 
\begin{align*} 
&X_t^2(1)-X_t^1(1)= M_t^2(1)-M_t^1(1)+\theta \int_0^{t} (X_s^2(1)-X_s^1(1)) ds+ \tilde{E}_t,\nn\\
&\text{ where $M_t^1(1)$ and $M_t^2(1)$ are  as in Lemma \ref{9l2.1} and } \lim_{N\to \infty} \E(\sup_{s\leq t} \vert \tilde{E}_s\vert )=0.
\end{align*}
Having established the above, the proof of Proposition \ref{9p1.1} is quite straightforward and follows similarly to that of Proposition 1 in \cite{DP99}. We omit the details. 

Next, to prove Proposition \ref{9p1.2},  again we apply Lemmas \ref{9l2.1}-\ref{9l2.6} in \eqref{9ea3.67} and collect all the error terms to see that for any $\phi\in C_b^3(\R^d)$,
 \begin{align}\label{9ea7.80}
&X_t^1(\phi)=X_0^1(\phi)+M_t^1(\phi)+(\theta-b_d)\int_0^{t} X_s^1(\phi) ds+\int_0^t X_s^1(\frac{\Delta}{6}\phi) ds+\hat{E}_t^1(\phi),\nn\\
&\text{ where } M_t^1(\phi) \text{ is as in Lemma \ref{9l2.1} and } \lim_{N\to \infty} \E(\sup_{s\leq t} \vert \hat{E}_s^1(\phi)\vert )=0.
\end{align}
 Hence we arrive at the same conclusion as in (2.18) of \cite{DP99}. Using their proof for Proposition 2 in \cite{DP99}, we may conclude that Proposition \ref{9p1.2} holds.\\

It remains to prove Lemma \ref{9l2.5} and Lemma \ref{9l2.6}.

\section{Upper bounds for the collision term}\label{9s3}
In this section, we will give the proof of Lemma \ref{9l2.5}. To explain the section name, we recall  from \eqref{9e2.1}   the collision term $K_t^n(\phi)$ given by
\begin{align}\label{9ea5.69}
K_{t}^n(\phi)=&\frac{N+\theta}{2N+\theta} \frac{1}{N} \sum_{\beta} a_\beta^n(t)  \phi(B^\beta+W^\beta) 1(B^\beta+W^\beta\in \overline{\cR}^{X^{n-1}}_{T_\beta^-})\nn\\
\leq &\| \phi\| _\infty \frac{1}{N} \sum_{\beta} a_\beta^0(t)  1(B^\beta+W^\beta\in \overline{\cR}^{X^{0}}_{T_\beta^-}).
\end{align}
Therefore Lemma \ref{9l2.5} indeed gives an upper bound for $\E(K_t^n(\phi))$. 

Recall from \eqref{9e0.94} to see $ a_\beta^0(t)=1{(T_\beta\leq t, B^\beta\neq \Delta)}$  where we have replaced $\zeta_{\beta}^0>T_{\pi \beta}$ by $B^\beta\neq \Delta$ as $B^\beta=B^\beta_{T_{\pi\beta}}$. Recall \eqref{9e10.93} to see $\overline{\cR}^{X^{0}}_{T_\beta^-}:= {\cR}^{X^{0}}_{T_\beta^-}\cup K_0^N$. The assumption on $K_0^N$ from \eqref{9e0.93} implies
\begin{align*} 
 \| \phi\| _\infty \frac{1}{N}\E\Big( \sum_{\beta} 1{(T_\beta\leq t, B^\beta\neq \Delta)} 1(B^\beta+W^\beta\in K_0^N) \Big)\to 0 \text{ as } N\to \infty.
\end{align*} 
Define
\begin{align}\label{9e9.02}
J(t):=& \frac{1}{N}   \sum_{\beta}1{(T_\beta\leq t, B^\beta \neq \Delta)}   1(B^\beta+W^\beta\in \cR^{X^{0}}_{T_\beta^-}).
\end{align}
 It suffices to bound $\E(J(t))$ for the proof of Lemma  \ref{9l2.5}.

 For $n=0,1,2$, by \eqref{9e9.01} and \eqref{9e0.03} we have
  \begin{align*}
\text{Supp}(X_t^n)=\{B^\gamma_t: \gamma \sim t, \zeta_\gamma^n>t\}=\{B^\gamma:  T_{\pi \gamma}\leq t <T_\gamma, \zeta_\gamma^n>t\}.
\end{align*}
Using the definition of $\zeta_\gamma^n$, we claim that the above can be rewritten as
  \begin{align*}
\text{Supp}(X_t^n)=\{B^\gamma:  T_{\pi \gamma}\leq t <T_\gamma, \zeta_\gamma^n>T_{\pi \gamma}\}.
\end{align*}
 To see this, we note if $\zeta_\gamma^n>T_{\pi \gamma}$, then the particle $\gamma$ is alive in $X^n$. By the definition of $\zeta_\gamma^n$ from \eqref{9ea3.01}, we get $\zeta_\gamma^n=\zeta_{\gamma}^0$ and so $\zeta_{\gamma}^0>T_{\pi \gamma}$, in which case the definition of $\zeta_{\gamma}^0$ from \eqref{e6.95} implies $\zeta_{\gamma}^0=T_\gamma>t$, thus giving $\zeta_\gamma^n=\zeta_{\gamma}^0>t$. The other direction is immediate.   It follows that
   \begin{align}\label{9e3.1}
\cR_t^{X^n}=&\bigcup_{s\leq t}\{B^\gamma:  T_{\pi \gamma}\leq s <T_\gamma,\quad \zeta_\gamma^n>T_{\pi \gamma}\}\nn\\
=&\{B^\gamma:  T_{\pi \gamma}\leq t, \quad \zeta_\gamma^n>T_{\pi \gamma}\}.
\end{align}
Let $t=T_\beta^-$ and $n=0$ to get
\begin{align*}
\cR^{X^{0}}_{T_\beta^-}=\{B^\gamma:  T_{\pi \gamma}\leq T_\beta^-, \zeta_\gamma^0>T_{\pi \gamma}\}=\{B^\gamma: T_{\pi \gamma}<T_\beta, B^\gamma\neq \Delta\}.
\end{align*}
Apply the above in \eqref{9e9.02} to see 
\begin{align}\label{9e9.03}
\E(J(t)) \leq &\frac{1}{N} \sum_{\beta,\gamma} \E\Big( 1(T_\beta\leq t,T_{\pi \gamma}<T_\beta)  1(B^\beta+W^\beta=B^\gamma \neq \Delta)\Big).
\end{align}
For each particle $\alpha \in \cI$, we denote by $\cH_\alpha$ the $\sigma$-field of all the events in the family line of $\alpha$ strictly before $T_\alpha$, plus the value of $t_\alpha$, which is given by
\begin{align*}
\cH_\alpha=\sigma\{t_{\alpha\vert m}, \delta_{\alpha\vert m}, e_{\alpha\vert m}, W^{\alpha\vert m}, m<\vert \alpha\vert \} \vee \sigma(t_\alpha).
\end{align*}
Then $T_\beta\in \cH_\beta$ and $W^\beta$ is independent of $\cH_\beta$. We may also conclude from $T_{\pi \gamma}<T_\beta$ that  $\gamma$ is not a strict descendant of $\beta$, thus giving $W^\beta$ is independent of $\cH_\gamma$. Recall $W^\beta$ is a uniform random variable on $\cN_N=\{y\in \cL_N: 0<\| y\| _\infty\leq N^{-1/2}\}$. Define  
\begin{align}\label{9e9.04}
\psi(N)=\vert \cN_N\vert =
\begin{dcases}
 (2[N^{1/d}]+1)^d-1 \sim 2^d N,  &\text{ in } d\geq 5,\\
 (2[(N\log N)^{1/4}]+1)^4-1\sim 2^4 N\log N,  &\text{ in } d=4,
\end{dcases}
\end{align}
 where $f(N)\sim g(N)$ if $f(N)/g(N) \to 1$ as $N\to \infty$. 
Calculate the expectation in \eqref{9e9.03} by conditioning on $\cH_\beta \vee \cH_\gamma$ to see
\begin{align}\label{9e2.23}
\E(J(t)) \leq &\frac{1}{N}\frac{1}{\psi(N)} \sum_{\beta,\gamma} \E\Big(1(T_\beta\leq t,T_{\pi \gamma}<T_\beta )  1(B^\beta-B^\gamma \in \cN_N)\Big).
\end{align}
We shall note $B^\beta \neq \Delta$ and $B^\gamma \neq \Delta$ are implicitly included in the event $\{B^\beta-B^\gamma \in \cN_N\}$.

\begin{lemma}\label{9l3.2}
Let $\Psi: [0,\infty)\times \Omega\to \R$ be a bounded and $\cF_t$-predictable function and let $\beta\in \cI$. Then
\begin{align*}
\Psi(T_\beta,\omega) 1(T_\beta\leq t)-(2N+\theta)\int_0^t 1(T_{\pi\beta}<r\leq T_\beta) \Psi(r,\omega)dr
\end{align*}
is an $\cF_t$-martingale whose predictable quadratic variation is given by
\begin{align}\label{a2.23}
(2N+\theta) \int_0^t 1(T_{\pi\beta}<r\leq T_\beta) \Psi^2 (r,\omega) dr.
\end{align}
\end{lemma}
\begin{proof}
The martingale proof follows from Lemma 3.2 of \cite{DP99} while the predictable quadratic variation \eqref{a2.23} can be easily derived from the stochastic integral concerning the compensated Poisson process therein. Also, the restriction on $\vert \beta\vert >0$ in Lemma 3.2 of \cite{DP99} can be removed.
\end{proof}
Define
\begin{align}\label{9e2.24}
\text{nbr}_{\beta,\gamma}(r)=1(T_{\pi \beta}<r\leq T_\beta, T_{\pi \gamma}<r, B^\beta-B^\gamma \in \cN_N).
\end{align}
Apply Lemma \ref{9l3.2} in \eqref{9e2.23} to get
\begin{align}\label{9e2.31}
\E(J(t))\leq &\frac{1}{N\psi(N)} \sum_{\beta,\gamma} (2N+\theta) \E\Big(\int_0^t   \text{nbr}_{\beta,\gamma}(r) dr\Big)\nn\\
\leq&\frac{C}{\psi(N)} \int_0^t  \E\Big( \sum_{\beta,\gamma}\text{nbr}_{\beta,\gamma}(r)\Big) dr.
\end{align}
It suffices to get bounds for $\E(\sum_{\beta,\gamma} \text{nbr}_{\beta,\gamma}(r))$. Define
\begin{align*}
\psi_0(N)={\psi(N)}/{N} \text{ and } \ I(t)=1+\int_0^t (1+s)^{1-d/2} ds, \forall t\geq 0.
\end{align*}
It is easy to check by \eqref{9e9.04} that there exist constants $C>c>0$ so that
\begin{align}\label{a2.31}
c\psi_0 (N)\leq I(N)\leq C\psi_0(N), \quad \forall N\text{ large}.
\end{align}
Recall we only consider particles $\beta,\gamma$ with $1\leq \beta_0, \gamma_0\leq M_N=NX_0^0(1)$.
\begin{lemma}\label{9l4.1}
There is some constant $C>0$ so that for any $0<r<t$,
\begin{align*}
\text{(i)} \quad &\E\Big(\sum_{\beta,\gamma: \gamma_0\neq \beta_0} \text{nbr}_{\beta,\gamma}(r)\Big)\leq C(NX_0^0(1))^2   \frac{1}{(1+(2N+2\theta)r)^{d/2-1}}.\\
\text{(ii)}  \quad &\E\Big(\sum_{\beta,\gamma: \gamma_0= \beta_0} \text{nbr}_{\beta,\gamma}(r)\Big)\leq CNX_0^0(1)  I((2N+2\theta)r).
\end{align*}
\end{lemma}
Assuming the above lemma, we may finish the proof of Lemma \ref{9l2.5}.
\begin{proof}[Proof of Lemma \ref{9l2.5} assuming Lemma \ref{9l4.1}]
Apply Lemma \ref{9l4.1} in \eqref{9e2.31} to get
\begin{align*}
\E(J(t)) \leq & \frac{C}{\psi(N)}  (NX_0^0(1))^2 \int_0^t \frac{1}{(1+(2N+2\theta)r)^{d/2-1}} dr\\
&+\frac{C}{\psi(N)}  NX_0^0(1)  \int_0^t  I((2N+2\theta)r)dr\\
\leq &  \frac{C}{\psi_0(N)}  (X_0^0(1))^2 I((2N+2\theta)t) +\frac{C}{\psi_0(N)}  X_0^0(1)  I((2N+2\theta)t) \\
\leq &  C(X_0^0(1))^2 +CX_0^0(1),
\end{align*}
where the last inequality is by \eqref{a2.31}. The proof is complete.
\end{proof}

It remains to prove Lemma \ref{9l4.1}. We first give some elementary estimates. Let $Y_0^N,Y_1^N, \cdots$ be i.i.d. random variables on $\R^d$ so that $Y_0^N=0$ or $Y_0^N$ is uniform on  $N^{1/2} \cN_N$, each with probability $1/2$. Set $V_n^N=Y_1^N+\cdots+Y_n^N$ for $n\geq 1$ so that $\{V_n^N\}$ is a random walk that with probability $1/2$ stays put, and with probability $1/2$ takes a jump uniformly on $N^{1/2} \cN_N$. The following result is from Lemma 4.2 of \cite{DP99}.
\begin{lemma}\label{9l4.2}
 There is some constant $C>0$ independent of $N$ so that
\begin{align*}
\P(V_m^N\in x+[-1,1]^d)\leq C  (1+m)^{-d/2}, \quad \forall m\geq 0, x\in \R^d.
\end{align*}
\end{lemma}

We also need estimates for the probability concerning the Poisson and Gamma random variables. For any $\lambda\geq 0$ and $m\geq 0$, define
\begin{align}\label{e4.2}
\pi(\lambda, m)=e^{-\lambda} \frac{\lambda^m}{m!}, \quad \text{ and } \quad \Pi(\lambda, m)=\sum_{k=m}^\infty   \pi(\lambda, k).
\end{align}
Let $\Pi(\lambda, m)=1$ if $m<0$. Define $\pi(\lambda, m)$=0 if $\lambda< 0$.
Let $\xi_1,\xi_2,\cdots$ be i.i.d. random variables with distribution $Exp(1)$. Define $\Gamma_m=\xi_1+\cdots+\xi_m$ for $m\geq 1$. Then $\Gamma_m$ has a $\Gamma(m,1)$ distribution.  Set by convention $\Gamma_m=0$ if $m\leq 0$.
 Recall $T_\beta$ from \eqref{9ea6.95}. By the scaling property of gamma distribution, we have
 \begin{align}\label{9ea2.00}
\P(T_\beta\leq t)=\P(\Gamma_{\vert \beta\vert +1}\leq (2N+\theta)t)=\Pi((2N+\theta)t, \vert \beta\vert +1),
 \end{align}
 where the last equality comes from that $\Gamma_m \leq r$ iff at least $m$ events occur before time $r$ in a Poisson process with rate $1$. Similarly,  
 \begin{align}\label{9ea6.84}
\P(T_{\pi \beta}\leq t<T_\beta)=\P(\Gamma_{\vert \beta\vert }\leq (2N+\theta)t<\Gamma_{\vert \beta\vert +1})=\pi((2N+\theta)t, \vert \beta\vert ).
 \end{align}

\begin{lemma}\label{9l4.3}
Let $\lambda\geq 0$.\\
(i) For any $p>0$, there is some constant $C=C(p)>0$ so that
\begin{align*}
\sum_{m=0}^\infty \pi(\lambda, m) (1+m)^{-p}\leq C  (1+\lambda)^{-p}.
\end{align*}
(ii) There is some constant $A>0$ so that for any $p\geq 0$ and $0< r\leq t$,
\begin{align}\label{9ea6.78}
 \E\Big(\sum_{m>ANr} (\frac{2N+2\theta}{2N+\theta})^m 1(\Gamma_{m}<(2N+\theta)r) m^p\Big)\leq C_p 2^{-Nr},
\end{align}
for some constant $C_p>0$.\\

 \no (iii) There is some constant $C>0$ so that for any $0< r\leq t$,
\begin{align*}
\sum_{m=0}^\infty \pi((2N+\theta)r, m) I(m) \leq  CI(N).
\end{align*}
\end{lemma}
\begin{proof}
The proof of (i) is immediate from Lemma 4.3 of \cite{DP99}. The inequality in (ii) is a revised version of Lemma 10.3 of \cite{DP99}, which also follows easily from standard large deviation estimates (see, e.g., Theorem 5.4 of \cite{MU05}). For the proof of $(iii)$, we note the sum of $m<ANr$ gives at most $I(ANr)\leq CI(N)$ while for $m>ANr$, we use $I(m)\leq m$ and
\begin{align*}
\pi((2N+\theta)r, m)\leq \Pi((2N+\theta)r, m)=\P(\Gamma_{m}<(2N+\theta)r)
\end{align*}
to see that the sum of $m>ANr$ is at most $C_1 2^{-Nr}$ by (ii). So the inequality holds.
\end{proof}

\begin{proof}[Proof of Lemma \ref{9l4.1}(i)]
Since we only consider particles $\beta,\gamma$ with $1\leq \beta_0, \gamma_0\leq NX_0^0(1)$, we may let $\beta_0=i$ and $\gamma_0=j$ for some $1\leq i\neq j\leq NX_0^0(1)$. Set $0< r<t$. Recall $\text{nbr}_{\beta,\gamma}(r)$ from \eqref{9e2.24} to get
\begin{align*}
\sum_{\beta: \beta_0=i}\sum_{\gamma: \gamma_0=j} \E(\text{nbr}_{\beta,\gamma}(r))\leq &\sum_{\beta: \beta_0=i}\sum_{\gamma: \gamma_0=j} \P(T_{\pi \beta}<r< T_\beta, T_{\pi \gamma}<r, B^\beta\neq \Delta, B^\gamma \neq \Delta)\\
&\times \P(N^{1/2}(x_j-x_i)+V_{\vert \beta\vert +\vert \gamma\vert }^N \in [-1,1]^d),
\end{align*}
where the last probability follows from \eqref{ea3.22} which implies that conditioning on $\{B^\beta, B^\gamma \neq \Delta\}$,
\begin{align}\label{eb3.22}
B^\beta-B^\gamma=x_{i}-x_j+\sum_{s=0}^{\vert \beta\vert -1} W^{\beta\vert s} 1(e_{\beta\vert s}=\beta_{s+1})-\sum_{t=0}^{\vert \gamma\vert -1} W^{\gamma\vert t} 1(e_{\gamma\vert t}=\gamma_{t+1}).
\end{align}
 We note $B^\beta\neq \Delta$ is equivalent to $\delta_{\beta\vert m}=1$ for all $0\leq m\leq \vert \beta\vert -1$. Similarly for $B^\gamma \neq \Delta$. These two events are both independent of $T_{\pi \beta}, T_\beta, T_{\pi \gamma}$. By considering $\vert \beta\vert =l$ and $\vert \gamma\vert =m$ for $m\geq 0$ and $l\geq 0$, we get
\begin{align}\label{9e2.22}
\sum_{\beta: \beta_0=i}\sum_{\gamma: \gamma_0=j}& \E(\text{nbr}_{\beta,\gamma}(r)) 
\leq  \sum_{l=0}^\infty \sum_{m=0}^\infty (\frac{N+\theta}{2N+\theta})^{l+m} \cdot 2^l \cdot 2^m \cdot \pi((2N+\theta) r, l)\nn\\
&\times  \Pi((2N+\theta) r, m)  \cdot \P(N^{1/2}(x_j-x_i)+V_{l+m}^N \in [-1,1]^d), 
\end{align}
where the first factor on the right-hand side gives the probability that $B^\beta\neq \Delta$ and $B^\gamma \neq \Delta$, i.e. no death along both family lines. The $2^l$ and $2^m$ count the number of $\beta, \gamma$. The fourth and fifth terms are the probability of $T_{\pi \beta}<r< T_\beta$ and $T_{\pi \gamma}<r$ (see \eqref{9ea2.00} and \eqref{9ea6.84}). Use Lemma \ref{9l4.2} to bound the last term on the right-hand side to conclude
\begin{align}\label{9e2.21}
\sum_{\beta: \beta_0=i}\sum_{\gamma: \gamma_0=j} \E(\text{nbr}_{\beta,\gamma}(r))\leq   &\sum_{l=0}^\infty \sum_{m=0}^\infty (\frac{2N+2\theta}{2N+\theta})^{l+m}  \pi((2N+\theta) r, l)\nn\\
&\times  \Pi((2N+\theta) r, m) \cdot \frac{C}{(1+l+m)^{d/2}}.
\end{align}
The following lemma will be used frequently below.
\begin{lemma}\label{9l4.0}
There is some constant $C>0$ such that for any $0< r\leq t$ and $l\geq 0$, we have
\begin{align}\label{e4.21}
& \sum_{m=0}^\infty (\frac{2N+2\theta}{2N+\theta})^{m}   \Pi((2N+\theta) r, m) \frac{1}{(1+l+m)^{d/2}}\leq  C   \frac{1}{(1+l)^{d/2-1}}.
\end{align}
\end{lemma}
\begin{proof}
Let $A>0$ be as in Lemma \ref{9l4.3} (ii) so that
\begin{align*}
  &\sum_{m>ANr} (\frac{2N+2\theta}{2N+\theta})^{m}   \Pi((2N+\theta) r, m) \frac{1}{(1+l+m)^{d/2}}\\
  &\leq  \frac{1}{(1+l)^{d/2}} \sum_{m>ANr} (\frac{2N+2\theta}{2N+\theta})^{m}   \P(\Gamma_{m}<(2N+\theta)r)  \\ 
  &\leq   \frac{1}{(1+l)^{d/2}}\times C_0   2^{-Nr}  \leq C\frac{1}{(1+l)^{d/2-1}},
\end{align*}
where in the first inequality we write $\Pi((2N+\theta)r, m)$ as $\P(\Gamma_{m}< (2N+\theta)r)$ by \eqref{9ea2.00}. The second inequality is by Lemma \ref{9l4.3} (ii). Next, the sum for $m\leq ANr$ is at most 
\begin{align*}
&\sum_{m=0}^{ANr}   e^{A\theta r}   \frac{1}{(1+l+m)^{d/2}}\leq C \frac{1}{(1+l)^{d/2-1}}.
\end{align*}
 We conclude \eqref{e4.21} holds.
\end{proof}
The above lemma implies that \eqref{9e2.21} can be bounded above by
\begin{align*}
&  \sum_{l=0}^\infty   (\frac{2N+2\theta}{2N+\theta})^{l}  \pi((2N+\theta) r, l) \cdot C  \frac{1}{(1+l)^{d/2-1}}\nn\\
&= Ce^{\theta r} \sum_{l=0}^\infty      \pi((2N+2\theta) r, l)   \frac{1}{(1+l)^{d/2-1}}\leq C   \frac{1}{(1+(2N+2\theta) r)^{d/2-1}},
\end{align*}
where the last inequality is by Lemma \ref{9l4.3}(i). The proof of Lemma \ref{9l4.1}(i) is complete by noticing that there are at most $(NX_0^0(1))^2$ such pairs of $i$ and $j$.
\end{proof}

\begin{proof}[Proof of Lemma \ref{9l4.1}(ii)]
Now we proceed to the case when $\beta_0=\gamma_0=i$ for some $1\leq i\leq NX_0^0(1)$. By translation invariance, we get
\begin{align*}
&\E\Big(\sum_{\beta,\gamma: \gamma_0= \beta_0} \text{nbr}_{\beta,\gamma}(r)\Big)=  NX_0^0(1) \E\Big(\sum_{\beta,\gamma\geq 1} \text{nbr}_{\beta,\gamma}(r)\Big),
\end{align*}
In the above, we have replaced $\beta_0=\gamma_0=1$ by $\beta,\gamma\geq 1$ where $1\in \cI$ is the first individual in generation $0$. It suffices to prove for all $0<r<t$,
\begin{align}\label{9e9.51}
I_0:= \E\Big(\sum_{\beta,\gamma\geq 1} \text{nbr}_{\beta,\gamma}(r)\Big)\leq C I((2N+2\theta) r).
\end{align}
Recall $\text{nbr}_{\beta,\gamma}(r)$ from \eqref{9e2.24}. By letting $\alpha=\beta\wedge \gamma$, we may write the sum of $\beta,\gamma$ as   
\begin{align}\label{9e6.11}
I_0\leq \sum_{k=0}^\infty \sum_{\alpha\geq 1, \vert \alpha\vert =k}  \sum_{l=0}^\infty \sum_{m=0}^\infty\sum_{\substack{\beta\geq \alpha,\\ \vert \beta\vert =k+l}}\sum_{\substack{ \gamma\geq \alpha,\\ \vert \gamma\vert =k+m}} \E\Big(\text{nbr}_{\beta,\gamma}(r) \Big).
\end{align} 
Let $\alpha, \beta, \gamma$ be as in the summation and then condition on $\cH_{\alpha}$ to get
\begin{align}\label{9e2.62}
\E(\text{nbr}_{\beta,\gamma}(r)\vert \cH_{\alpha})\leq & 1{\{T_{\alpha}<r, B^{\alpha}\neq \Delta\}}\cdot    (\frac{N+\theta}{2N+\theta})^{l+m}\nn\\
&\times \P(T_{\pi \beta}-T_{\alpha}<r-T_{\alpha}\leq T_\beta-T_{\alpha}\vert \cH_{\alpha})\nn\\
&\times \P(T_{\pi \gamma}-T_{\alpha}<r-T_{\alpha}\vert \cH_{\alpha})\nn\\
&\times  \P(B^\beta-B^\gamma \in \cN_N \vert \cH_{\alpha}),
\end{align}
where the second term is the upper bound for the probability that no death events occur on the family line after $\beta,\gamma$ split. Apply  \eqref{9ea2.00} and \eqref{9ea6.84} in \eqref{9e2.62} to see
\begin{align}\label{e2.63}
\E(\text{nbr}_{\beta,\gamma}(r)\vert \cH_{\alpha})\leq & 1{\{T_{\alpha}<r, B^{\alpha}\neq \Delta\}} \cdot   (\frac{N+\theta}{2N+\theta})^{l+m} \pi((2N+\theta) (r-T_\alpha), l-1) \nn\\
&\times   \Pi((2N+\theta) (r-T_\alpha), m-1) \cdot \frac{C}{(1+l+m)^{d/2}}.
\end{align}
where we have also used Lemma \ref{9l4.2} to bound the last probability in \eqref{9e2.62}. It follows that
\begin{align}\label{9e2.640}
I_0\leq  C\E\Big(&\sum_{k=0}^\infty \sum_{\alpha\geq 1, \vert \alpha\vert =k} 1{\{T_{\alpha}<r, B^{\alpha}\neq \Delta\}}\nn \\
&\sum_{l=0}^\infty \sum_{m=0}^\infty 2^l \cdot 2^{m} \cdot  (\frac{N+\theta}{2N+\theta})^{l+m}  \pi((2N+\theta) (r-T_\alpha), l-1)\nn \\
&\times    \Pi((2N+\theta) (r-T_\alpha), m-1) \cdot \frac{1}{(1+l+m)^{d/2}}\Big).
\end{align}
By Lemma \ref{9l4.0}, the sum of $m$ is at most 
\begin{align*} 
  \sum_{m=0}^\infty    (\frac{2N+2\theta}{2N+\theta})^{m}   \Pi((2N+\theta) (r-T_\alpha), m-1) \cdot \frac{1}{(1+l+m)^{d/2}}\leq C\frac{1}{(1+l)^{d/2-1}}.
\end{align*}
Then the sum of $l$ gives
 \begin{align}\label{9e9.58}
&\sum_{l=0}^\infty(\frac{2N+2\theta}{2N+\theta})^{l+1}  \pi((2N+\theta) (r-T_\alpha), l-1) \frac{C}{(1+l)^{d/2-1}}\\
&\leq Ce^{\theta r} \sum_{l=0}^\infty \pi((2N+2\theta) (r-T_\alpha), l) \frac{1}{(1+l)^{d/2-1}}\leq  \frac{C }{(1+(2N+2\theta) (r-T_\alpha))^{d/2-1}},\nn
\end{align}
where the last inequality is by Lemma \ref{9l4.3} (i). Now conclude from the above to see
\begin{align}\label{e6.37}
I_0\leq  C\E\Big(\sum_{\alpha\geq 1} 1{\{T_{\alpha}<r, B^{\alpha}\neq \Delta\}}  \frac{1}{(1+(2N+2\theta) (r-T_\alpha))^{d/2-1}}\Big).
\end{align}
\begin{lemma}\label{9l3.2a}
For any Borel function $f: \R^+\to \R^+$ and $r\geq u\geq 0$, we have
\begin{align*}
 \E\Big(\sum_{\alpha\geq 1} 1{\{u<T_{\alpha}<r, B^{\alpha}\neq \Delta\}} f((2N+2\theta) (r-T_\alpha))\Big)\leq e^{\theta r}   \int_0^{(2N+2\theta)(r-u)}     f(y) dy.
 \end{align*}
\end{lemma}
\begin{proof}
Apply Lemma \ref{9l3.2} to see that the expectation above is equal to
\begin{align}\label{9e2.71}
&  (2N+\theta)  \int_0^r  \E\Big(\sum_{\alpha\geq 1} 1{\{T_{\pi\alpha}<s\leq T_{\alpha}, B^{\alpha}\neq \Delta\}}\Big) 1(u<s) f((2N+2\theta) (r-s))ds \nn\\
&=  (2N+\theta)  \int_u^r  e^{\theta s}  f((2N+2\theta) (r-s))ds  \leq e^{\theta r}   \int_0^{(2N+2\theta)(r-u)}     f(y) dy,
\end{align}
where the equality is by Lemma \ref{9l2.0} with $\phi\equiv 1$.
\end{proof}
Use the above lemma with $f(y)=1/(1+y)^{d/2-1}$ and $u=0$ to see that \eqref{e6.37} becomes
\begin{align*} 
I_0\leq   C  \int_0^{(2N+2\theta)r}   \frac{1}{(1+y)^{d/2-1}}   dy\leq   C I((2N+2\theta)r),
\end{align*}
as required.
\end{proof}

\section{Refined estimates for the collision term}\label{9s4}

This section is devoted to refining the bounds given in Lemma \ref{9l4.1} and preparing for the convergence of the collision term in the following section. To start with, by recalling \eqref{9e2.23}, we define the collision from different ancestors by
\begin{align}\label{9e2.25}
J_0(t)=& \frac{1}{N\psi(N)} \sum_{\beta,\gamma: \beta_0\neq \gamma_0}  1(T_\beta\leq t,T_{\pi \gamma}<T_\beta ) 1(B^\beta-B^\gamma \in \cN_N).
\end{align}
By assuming the initial particles are sufficiently spread out, we show that $\E(J_0(t)) \to 0$ as $N\to \infty$.  
\begin{lemma}\label{9l4.1a}
If the initial condition $X_0^{0,N}$ converges to some atomless $X_0$ in $M_F(\R^d)$, then for any $t\geq 0$, we have $\E(J_0(t)) \to 0$ as $N\to \infty$.
\end{lemma}

\begin{proof}
By letting $\beta_0=i$ and $\gamma_0=j$ for some $1\leq i\neq j\leq NX_0^0(1)$, we have
\begin{align}\label{9ea7.45}
\E(J_0(t))\leq& \frac{1}{N\psi(N)} \sum_{i\neq j} \sum_{\beta: \beta_0=i}\sum_{\gamma: \gamma_0=j}  \E\Big(1(T_{\pi\beta}< t)1(T_{\pi \gamma}<t ) 1(B^\beta-B^\gamma \in \cN_N)\Big),
\end{align}
where we have bounded $1(T_\beta\leq t,T_{\pi \gamma}<T_\beta)$ by $1(T_{\pi\beta}< t, T_{\pi \gamma}<t)$ to get symmetry between $\beta$ and $\gamma$.
Similar to \eqref{9e2.22}, we may bound the above sum by
\begin{align}\label{9a1.3}
\E(J_0(t))\leq&\frac{1}{N\psi(N)} \sum_{i\neq j}  \sum_{l=0}^\infty \sum_{m=0}^\infty   (\frac{2N+2\theta}{2N+\theta})^{l+m}   \cdot \Pi((2N+\theta) t, l)\\
&\times  \Pi((2N+\theta) t, m)  \cdot \P(N^{1/2}(x_j-x_i)+V_{l+m}^N \in [-1,1]^d):=\frac{1}{N\psi(N)} \sum_{i\neq j} I_0(i,j).\nn
\end{align}
It remains to show
\begin{align*}
\frac{1}{N\psi(N)} \sum_{i\neq j} I_0(i,j)\to 0 \text{ as } N\to\infty.
\end{align*}
For any $i\neq j$ we define
\begin{align}\label{9a1.4}
 I_1(i,j):=&  \sum_{l=0}^{ANt} \sum_{m=0}^{ANt}   (\frac{2N+2\theta}{2N+\theta})^{l+m}   \cdot \Pi((2N+\theta) t, l)\nn\\
&\times  \Pi((2N+\theta) t, m)  \cdot \P(N^{1/2}(x_j-x_i)+V_{l+m}^N \in [-1,1]^d),
\end{align}
where $A>0$ is as in Lemma \ref{9l4.3}(ii). By the symmetry between $m$ and $l$, we obtain
\begin{align*}
I_0(i,j)-I_1(i,j)\leq &2 \sum_{l>ANt} \sum_{m=0}^{\infty}   (\frac{2N+2\theta}{2N+\theta})^{l+m}   \cdot \Pi((2N+\theta) t, l)\nn\\
&\times  \Pi((2N+\theta) t, m)  \cdot \frac{C}{(1+l+m)^{d/2}},
\end{align*}
where we have applied Lemma \ref{9l4.2} to bound the last term in \eqref{9a1.4}.
By Lemma \ref{9l4.0}, the sum of $m$ gives at most $Ce^{\theta t}/(1+l)^{d/2-1}\leq C$. We are left with
\begin{align}\label{9a1.5}
I_0(i,j)-I_1(i,j)\leq & C \sum_{l>ANt}   (\frac{2N+2\theta}{2N+\theta})^{l}   \cdot \P(\Gamma_l<(2N+\theta)t)  \leq C\cdot C_0 2^{-Nt},
\end{align}
where the last inequality is by Lemma \ref{9l4.3}(ii). Hence
\begin{align*}
\frac{1}{N\psi(N)} \sum_{i\neq j} [I_0(i,j)-I_1(i,j)]\leq  \frac{1}{N\psi(N)}  [NX_0^0(1)]^2 C\cdot C_0 2^{-Nt}\to 0.
\end{align*}
It remains to show
\begin{align*}
\frac{1}{N\psi(N)} \sum_{i\neq j} I_1(i,j)\to 0 \text{ as } N\to\infty.
\end{align*}
To get bounds for $I_1(i,j)$, we simply bound $ \Pi((2N+\theta) t, l)$ and $ \Pi((2N+\theta) t, m)$ by $1$ to obtain
\begin{align*} 
I_1(i,j)\leq &  e^{2A\theta t}  \sum_{l=0}^{ANt} \sum_{m=0}^{ANt}    \P(N^{1/2}(x_j-x_i)+V_{l+m}^N \in [-1,1]^d).
\end{align*}
Rewrite the above sum by letting $n=l+m$ to get 
\begin{align} \label{9e2.42}
I_1(i,j)&\leq     e^{2A\theta t} \sum_{n=0}^{2ANt} \sum_{m=0}^{n}  \P(N^{1/2}(x_j-x_i)+V_{n}^N \in [-1,1]^d)\nn \\
&\leq e^{2A\theta t}  \sum_{n=0}^{2ANt} (n+1)  \P(N^{1/2}(x_j-x_i)+V_{n}^N \in [-1,1]^d).
\end{align}
Notice that if $\vert x_j-x_i\vert >\eps$, we need $n\geq \eps N^{1/2}-1$ so that $V_{n}^N$ gets to $N^{1/2}(x_j-x_i)+ [-1,1]^d$. Use this observation to see that \eqref{9e2.42} becomes
\begin{align} \label{9e2.43}
I_1(i,j)&\leq  e^{2A\theta t} \sum_{n=\eps N^{1/2}-1}^{2ANt} (n+1) \frac{C}{(n+1)^{d/2}}\nn\\
& + e^{2A\theta t} 1(\vert x_j-x_i\vert \leq \eps)  \sum_{n=0}^{2ANt} (n+1) \frac{C}{(n+1)^{d/2}},
\end{align}
where we have applied Lemma \ref{9l4.2} to bound the last term in \eqref{9e2.42}. 
Sum $i\neq j$ to get
\begin{align} 
\frac{1}{N\psi(N)}&\sum_{i\neq j} I_1(i,j)\leq \frac{e^{2A\theta t}}{N\psi(N)} (NX_0^{0,N}(1))^2 \sum_{n=\eps N^{1/2}-1}^{2ANt} (n+1) \frac{C}{(n+1)^{d/2}}\nn\\
  &+  \frac{e^{2A\theta t}}{N\psi(N)} N^2 (X_0^{0,N}\times X_0^{0,N}) (\{(x,y):\vert x-y\vert \leq \eps\})  \sum_{n=0}^{2ANt} (n+1) \frac{C}{(n+1)^{d/2}}.\nn
\end{align}
When $d\geq 5$, by recalling $\psi(N)\sim CN$ from \eqref{9e9.04}, one can check that the first term above converges to $0$ as $N\to \infty$. For the second term, we may use the fact that $X_0^{0, N}\to X_0$ in $M_F(\R^d)$ to see its limsup as $N\to \infty$ can be bounded by
\begin{align*}
C\cdot (X_0\times X_0) (\{(x,y):\vert x-y\vert \leq \eps\}),
\end{align*}
which can be arbitrarily small by picking $\eps>0$ small since we assume $X_0$ has no atoms. The proof for the case $d\geq 5$ is now complete.

Turning to $d=4$, we need a bit more care on the probability concerning $V_n^N$. The following result is from Lemma 5.2 of \cite{DP99}.
\begin{lemma}\label{9l5.02}
There are constants $0<\delta_0, c,C<\infty$ so that if $0<z/n<\delta_0$, then
\begin{align*}
\P(\vert V_n^N\vert \geq z)\leq C\exp(-cz^2/n).
\end{align*}
\end{lemma}

Similar to \eqref{9e2.43}, we apply Lemma \ref{9l4.2} to bound \eqref{9e2.42} by 
\begin{align} \label{9e2.44}
 I_1(i,j)&\leq  e^{2A\theta t}   1(\vert x_i-x_j\vert \leq \eps)  \sum_{n=0}^{2ANt} (n+1) \frac{C}{(n+1)^{2}}\nn\\
 &+  e^{2A\theta t} 1(\vert x_i-x_j\vert >\eps) \sum_{n=\eps N^{1/2}-1}^{N/(\log N)^3} (n+1) \cdot \P(N^{1/2}(x_j-x_i)+V_{n}^N \in [-1,1]^4)\nn\\
  &+  e^{2A\theta t}  \sum_{n=N/(\log N)^3}^{2ANt} (n+1) \frac{C}{(n+1)^{2}}.
\end{align}
Sum $i\neq j$ to see
\begin{align*} 
\frac{1}{N\psi(N)}  \sum_{i\neq j} I_1(i,j)&\leq  \frac{e^{2A\theta t}}{N\psi(N)} N^2 (X_0^{0}\times X_0^{0}) (\{(x,y):\vert x-y\vert \leq \eps\})  \sum_{n=0}^{2ANt} (n+1) \frac{C}{(n+1)^{2}}\nn\\
 &+ \frac{e^{2A\theta t}}{N\psi(N)} \sum_{\substack{i\neq j\\ \vert x_i-x_j\vert >\eps}} \sum_{n=\eps N^{1/2}-1}^{N/(\log N)^3} (n+1) \cdot \P(N^{1/2}(x_j-x_i)+V_{n}^N \in [-1,1]^4)\nn\\
  &+ \frac{e^{2A\theta t}}{N\psi(N)}(NX_0^{0}(1))^2 \sum_{n=N/(\log N)^3}^{2ANt} (n+1) \frac{C}{(n+1)^{2}}:=S_1+S_2+S_3.
\end{align*}
For $S_1$, again we may use  $X_0^{0,N}\to X_0$ in $M_F(\R^d)$ to see  
\begin{align}\label{e1.11}
\limsup_{N\to \infty} S_1\leq C\cdot (X_0^{0}\times X_0^{0}) (\{(x,y):\vert x-y\vert \leq \eps\}),
\end{align}
which can be arbitrarily small by picking $\eps>0$ small. 

For $S_3$, we recall that $\psi(N)\sim CN\log N$ in $d=4$ and then bound the sum of $1/(n+1)$ by the integral of $x^{-1}$ to get as $N\to \infty$,
\begin{align}\label{e1.12}
S_3\leq \frac{C}{\log N}  (X_0^{0}(1))^2 \int_{N/(\log N)^3}^{3ANt}  \frac{1}{x} dx\leq \frac{C(X_0^{0}(1))^2}{\log N}   C\log \log N \to 0.
\end{align}

It remains to show that $S_2$ converges to $0$. Let $z=n^{1/2}\log N$ and use $n\geq \eps N^{1/2}-1$ to see
\begin{align*}
z=n^{1/2}\log N=n\log N/n^{1/2}\leq n\log N/(\eps N^{1/2}-1)^{1/2}\leq \delta_0 n.
\end{align*}
Apply Lemma \ref{9l5.02} with the above $z$ with $z/n<\delta_0$ to get
\begin{align} \label{9e2.45}
\P(\vert V_n^N\vert \geq n^{1/2}\log N)\leq C\exp(-c(\log N)^2)=CN^{-c\log N}.
\end{align}
Use $n\leq N/(\log N)^3$ to see
\begin{align*}  
z=n^{1/2}\log N \leq N^{1/2}/(\log N)^{1/2}\leq \frac{\eps}{2} N^{1/2}.
\end{align*}
Hence for $\vert x_i- x_j\vert >\eps$, we have
\begin{align*}  
 \P(N^{1/2}(x_j-x_i)+V_{n}^N \in [-1,1]^4)\leq \P(\vert V_n^N\vert \geq   n^{1/2}\log N)\leq  CN^{-c\log N},
 \end{align*}
 where the last inequality is by \eqref{9e2.45}.
 Finally, we use the above to see $S_2$ is at most
 \begin{align*}  
S_2\leq  \frac{e^{2A\theta t}}{N\psi(N)}  (NX_0^{0,N}(1))^2  N^2 \cdot CN^{-c\log N} \to 0 \text{ as } N\to \infty.
 \end{align*}
 Together with \eqref{e1.11} and \eqref{e1.12}, the proof for $d=4$ is complete.
\end{proof}

The next step is to deal with collisions from the same ancestor, i.e. $\beta_0=\gamma_0$. We will calculate the contribution of collisions from distant relatives. To make it precise, following \eqref{9e2.25} we define
\begin{align}\label{9e2.26}
J(t,\tau)= \frac{1}{N\psi(N)} \sum_{\beta,\gamma: \beta_0= \gamma_0}  &1(T_\beta\leq t,T_{\pi \gamma}<T_\beta ) \nn\\
&\cdot 1(T_{\beta \wedge \gamma}\leq T_\beta-\tau) \cdot 1(B^\beta-B^\gamma \in \cN_N).
\end{align}
It is clear that $J(t,\tau)=0$ if $\tau \geq t$ since $T_{\beta \wedge \gamma}\geq T_{\beta_0}>0$. So we only need to consider $\tau <t$.
\begin{lemma}\label{9l4.3a}
For any $t>0$, there is some constant $C>0$ so that for any $\tau< t$,
\begin{align}
\E(J(t,\tau))\leq CX_0^{0}(1) \frac{1}{\psi_0(N)} \int_{(2N+\theta)\tau}^{(2N+\theta)t} \frac{1}{(1+y)^{d/2-1}}dy.
\end{align}
\end{lemma}
\begin{proof}
Recall $\text{nbr}_{\beta,\gamma}(r)$ from \eqref{9e2.24}. 
Similar to \eqref{9e2.31}, we may apply Lemma \ref{9l3.2} to see 
\begin{align*}
\E(J(t,\tau))&=\frac{1}{N\psi(N)} \sum_{\beta,\gamma: \beta_0= \gamma_0} (2N+\theta)\E\Big(\int_0^t \text{nbr}_{\beta,\gamma}(r) \cdot 1{(T_{\beta \wedge \gamma}\leq r-\tau)}dr \Big)\nn\\
&\leq \frac{C}{\psi(N)} \sum_{\beta,\gamma: \beta_0= \gamma_0} \int_\tau^t \E\Big(\text{nbr}_{\beta,\gamma}(r) \cdot 1{(T_{\beta \wedge \gamma}\leq r-\tau)} \Big)dr,
\end{align*}
where the last inequality follows since the integrand is zero when $r\leq \tau$.
Let $\beta_0=\gamma_0=i$ for some $1\leq i\leq NX_0^{0,N}(1)$. Use translation invariance to get  
\begin{align}\label{9e2.54}
\E(J(t,\tau))\leq &\frac{C}{\psi(N)} NX_0^{0}(1) \int_\tau^t \E\Big(\sum_{\beta, \gamma\geq1}\text{nbr}_{\beta,\gamma}(r) \cdot 1{(T_{\beta \wedge \gamma}\leq r-\tau)} \Big)dr.
\end{align}
The calculation for the expectation above is similar to that in \eqref{9e6.11}. One may continue all the arguments in the same way as that from \eqref{9e2.62} to \eqref{e6.37} except that $T_{\alpha}<r$ in \eqref{e6.37} is replaced by $T_{\alpha}<r-\tau$. So we may simply jump to \eqref{e6.37} to conclude
\begin{align*}
&\E\Big(\sum_{\beta, \gamma\geq 1} \text{nbr}_{\beta,\gamma}(r) \cdot 1{(T_{\beta \wedge \gamma}\leq r-\tau)} \Big) \\
&\leq  C\E\Big(\sum_{\alpha\geq 1} 1{\{T_{\alpha}<r-\tau, B^{\alpha}\neq \Delta\}}  \frac{1}{(1+(2N+2\theta) (r-T_\alpha))^{d/2-1}}\Big).
\end{align*}
Apply Lemma \ref{9l3.2} to see the right-hand side above is equal to
\begin{align}\label{9e2.75}
&  C  (2N+\theta)  \int_0^{r-\tau} \E\Big(\sum_{\alpha\geq 1} 1{\{T_{\pi\alpha}<s\leq T_{\alpha}, B^{\alpha}\neq \Delta\}}\Big) \frac{1}{(1+(2N+2\theta) (r-s))^{d/2-1}}ds \nn\\
&=C   (2N+\theta)  \int_0^{r-\tau}  e^{\theta s} \frac{1}{(1+(2N+2\theta) (r-s))^{d/2-1}}ds \leq C     \int_{(2N+2\theta)\tau}^{(2N+2\theta)r}     \frac{1}{(1+y)^{d/2-1}}dy.
\end{align}
Finally we apply \eqref{9e2.75}  in \eqref{9e2.54} to get
\begin{align*} 
\E(J(t,\tau)) &\leq \frac{C}{\psi(N)} NX_0^{0}(1) \int_\tau^t  \int_{(2N+\theta)\tau}^{(2N+\theta)r} \frac{1}{(1+y)^{d/2-1}}dy dr\nn\\
&\leq \frac{C}{\psi_0(N)} X_0^{0}(1)    \int_{(2N+\theta)\tau}^{(2N+\theta)t} \frac{1}{(1+y)^{d/2-1}} dy,
\end{align*}
as required. 
\end{proof}

\section{Convergence of the collision term}\label{9s5}

We are ready to give the convergence of the collision term in this section and prove Lemma \ref{9l2.6}. Recall from \eqref{9e2.1} that
\begin{align}\label{9e3.03}
K_{t}^n(\phi)=& \frac{N+\theta}{N} \frac{1}{2N+\theta}\sum_{\beta} a_\beta^n(t)  \phi(B^\beta+W^\beta) 1(B^\beta+W^\beta\in \overline{\cR}^{X^{n-1}}_{T_\beta^-}).
\end{align}
The proof proceeds by approximating $K_t^n(\phi)$ in several steps until we reach $b_d \int_0^t X_r^1(\phi) dr$. The first one will be obtained by adjusting the constant, replacing $ \phi(B^\beta+W^\beta)$ by $ \phi(B^\beta)$ and replacing $\overline{\cR}^{X^{n-1}}_{T_\beta^-}$ by ${\cR}^{X^{n-1}}_{T_\beta^-}$ in \eqref{9e3.03}, that is, we define
\begin{align}\label{9e3.2}
K_{t}^{n,0}(\phi):=&\frac{1}{2N+\theta}  \sum_{\beta} a_\beta^n(t)  \phi(B^\beta) 1(B^\beta+W^\beta\in \cR^{X^{n-1}}_{T_\beta^-}).
\end{align}

\begin{lemma}\label{9l5.1}
Let $n=1$ or $2$. For all $0<t<\infty$, 
\begin{align*}
\lim_{N\to \infty} \E\Big(\sup_{s\leq t} \Big\vert K_s^{n}(\phi)-K_s^{n,0}(\phi)\Big\vert \Big)=0.
\end{align*}
\end{lemma}

The second step is to replace the indicator $1(B^\beta+W^\beta\in \cR^{X^{n-1}}_{T_\beta^-})$ by its conditional expectation. For any $m\geq 0$, recall from \eqref{9e3.1} that
\begin{align}\label{9e3.4}
\cR^{X^{m}}_{T_\beta^-}=\{B^\gamma: T_{\pi \gamma}\leq T_\beta^-,\quad \zeta_\gamma^{m}>T_{\pi \gamma}\}=\{B^\gamma: T_{\pi \gamma}< T_\beta,\quad \zeta_\gamma^{m}>T_{\pi \gamma}\}.
\end{align}
Define
\begin{align} \label{9e3.33}
\nu_m(\beta)=\Big\vert \{B^\gamma: T_{\pi \gamma}<T_\beta, B^\gamma-B^\beta \in \cN_N, \zeta_{\gamma}^m>T_{\pi \gamma}\}\Big\vert 
\end{align}
to be the number of neighbors of $B^\beta$ that have been visited by $X^m$ up to time $T_\beta^-$. Then by conditioning on $\cF_{T_\beta^-}$, we get
\begin{align} \label{9e3.31}
&\E\Big(1(B^\beta+W^\beta\in \cR^{X^{m}}_{T_\beta^-})\Big\vert \cF_{T_\beta^-}\Big)=\frac{\nu_m(\beta)}{\psi(N)},
\end{align}
where we have used that $W^\beta$ is uniform on $\cN_N$ and independent of $\cF_{T_\beta^-}$. Given the above, we set
\begin{align}\label{9e3.3}
K_{t}^{n,1}(\phi):=&\frac{1}{2N+\theta}   \sum_{\beta} a_\beta^n(t)  \phi(B^\beta)  \frac{\nu_{n-1}(\beta)}{\psi(N)}.
\end{align}

\begin{lemma}\label{9l5.2}
Let $n=1$ or $2$. For all $0<t<\infty$, 
\begin{align*}
\lim_{N\to \infty} \E\Big(\sup_{s\leq t} \Big\vert K_s^{n,0}(\phi)-K_s^{n,1}(\phi)\Big\vert \Big)=0.
\end{align*}
\end{lemma}

The number $\nu_{n-1}(\beta)$ in \eqref{9e3.3} counts the number of sites that have been occupied in the neighborhood of $B^\beta$, which might have been visited by more than one particle. However, since birth events on the same site are rare, one would expect their difference to be small. Define
\begin{align}
\text{nbr}_{\beta,\gamma}^m=1(T_{\pi \gamma}<T_\beta,  \zeta_\gamma^m>T_{\pi \gamma}, B^\beta-B^\gamma \in \cN_N).
\end{align}
Our next step would be to replace $\nu_{m}(\beta)$ by $\sum_{\gamma}\text{nbr}_{\beta,\gamma}^{m}$ as the latter is easier to compute. Moreover, the event in $\text{nbr}_{\beta,\gamma}^m$ represents a collision between two particles while our work in Section \ref{9s4} already gives bounds on the contributions of such collisions from distant relatives. Lemma \ref{9l4.3a} motivates us to define the cutoff time $\tau_N>0$ such that
\begin{align} \label{9e3.34}
\begin{cases} 
N\tau_N \to \infty, \tau_N \to 0, &\text{ in } d\geq 5;\\
\tau_N=1/\log N, &\text{ in } d=4.
\end{cases}  
\end{align}
Throughout the rest of the section, we always assume $\tau_N<t$.
By letting $\tau=\tau_N$ in \eqref{9e3.34}, we get
\begin{align}\label{9ec3.5}
\E(J(t,\tau_N))\leq  CX_0^{0,N}(1) \frac{1}{\psi_0(N)} \int_{(2N+\theta)\tau_N}^{(2N+\theta)t} \frac{1}{(1+y)^{d/2-1}}dy \to 0 \text{ as } N\to\infty.
\end{align}
So we define our third approximation of $K_t^n(\phi)$ by
\begin{align}\label{9e3.5}
K_{t}^{n,2}(\phi)=&\frac{1}{2N+\theta}   \sum_{\beta} a_\beta^n(t)  \phi(B^\beta)  \frac{1}{\psi(N)}\sum_{\gamma}\text{nbr}_{\beta,\gamma}^{n-1} 1(T_{\gamma \wedge \beta}>T_\beta-\tau_N)\nn\\
=&\frac{1}{2N+\theta}   \frac{1}{\psi(N)} \sum_{\beta} 1(T_\beta\leq t, \zeta_\beta^n>T_{\pi\beta})  \phi(B^\beta)  \nn\\
&\quad \sum_{\gamma}1(T_{\pi \gamma}<T_\beta,  \zeta_\gamma^{n-1}>T_{\pi\gamma}, B^\beta-B^\gamma \in \cN_N) 1(T_{\gamma \wedge \beta}>T_\beta-\tau_N).
\end{align}

\begin{lemma}\label{9l5.3}
Let $n=1$ or $2$. For all $0<t<\infty$, 
\begin{align*}
\lim_{N\to \infty} \E\Big(\sup_{s\leq t} \Big\vert K_s^{n,1}(\phi)-K_s^{n,2}(\phi)\Big\vert \Big)=0.
\end{align*}
\end{lemma}

Moving to the fourth step, we will use a time change with the cutoff time $\tau_N$: In the second expression of $K_{t}^{n,2}(\phi)$ in \eqref{9e3.5}, we will replace $\{\zeta_\beta^n>T_{\pi\beta}\}$ by $\{\zeta_\beta^n>T_\beta-\tau_N, B^\beta\neq \Delta\}$ and $\{\zeta_\gamma^{n-1}>T_{\pi\gamma}\}$  by $\{\zeta_\gamma^{n-1}>T_\beta-\tau_N, B^\gamma\neq \Delta\}$ and define
\begin{align}\label{9e3.6}
& K_{t}^{n,3}(\phi)=\frac{1}{2N+\theta}  \frac{1}{\psi(N)}  \sum_{\beta} 1(T_\beta\leq t, \zeta_\beta^n>T_\beta-\tau_N)  \phi(B^\beta)\nn \\
&\quad \sum_{\gamma}1(T_{\pi \gamma}<T_\beta,  \zeta_\gamma^{n-1}>T_\beta-\tau_N, B^\beta-B^\gamma \in \cN_N) 1(T_{\gamma \wedge \beta}>T_\beta-\tau_N),
\end{align}
where, as usual, $B^\beta, B^\gamma \neq \Delta$ are implicit in $\{B^\beta-B^\gamma \in \cN_N\}$.

\begin{lemma}\label{9l5.4}
Let $n=1$ or $2$. For all $0<t<\infty$, 
\begin{align*}
\lim_{N\to \infty} \E\Big(\sup_{s\leq t} \Big\vert K_s^{n,2}(\phi)-K_s^{n,3}(\phi)\Big\vert \Big)=0.
\end{align*}
\end{lemma}

By recalling from \eqref{9e10.03} that $X_t^1\leq X_t^2\leq X_t^0$, one can check for $n=1$ or $n=2$,
\begin{align*} 
&\Big\{\zeta_\beta^n>T_\beta-\tau_N,  \zeta_\gamma^{n-1}>T_\beta-\tau_N, T_{\gamma \wedge \beta}>T_\beta-\tau_N \Big\}\\
=&\Big\{\zeta_{\gamma \wedge \beta}^1>T_\beta-\tau_N,  T_{\gamma \wedge \beta}>T_\beta-\tau_N \Big\}\\
=&\Big\{\zeta_{ \beta}^1>T_\beta-\tau_N,  T_{\gamma \wedge \beta}>T_\beta-\tau_N \Big\}.
\end{align*}
The above equalities follow since $\beta$ and $\gamma$ share the same ancestor at time $T_\beta-\tau_N<T_{\gamma \wedge \beta}$. Hence \begin{align}\label{9e3.7}
& K_{t}^{n,3}(\phi)=\frac{1}{2N+\theta}  \frac{1}{\psi(N)}  \sum_{\beta} 1(T_\beta\leq t, \zeta_\beta^1>T_\beta-\tau_N)  \phi(B^\beta)\nn \\
&\quad \times \sum_{\gamma}1(T_{\pi \gamma}<T_\beta,   B^\beta-B^\gamma \in \cN_N, T_{\gamma \wedge \beta}>T_\beta-\tau_N)= K_{t}^{1,3}(\phi).
\end{align}
Define
\begin{align}\label{9e3.8}
F_\beta(r)=\frac{1}{\psi_0(N)}\sum_{\gamma}1(T_{\pi \gamma}<r,   B^\beta-B^\gamma \in \cN_N, T_{\gamma \wedge \beta}>r-\tau_N)
\end{align}
so that 
\begin{align}\label{9e3.9}
& K_{t}^{n,3}(\phi)=\frac{1}{N(2N+\theta)}  \sum_{\beta} 1(T_\beta\leq t, \zeta_\beta^1>T_\beta-\tau_N)  \phi(B^\beta)F_\beta(T_\beta).
\end{align}
Next, define
\begin{align}\label{9e3.10}
& G_r(\phi)=\frac{1}{N}  \sum_{\beta} 1(T_{\pi\beta}<r\leq T_\beta, \zeta_\beta^1>r-\tau_N)  \phi(B^\beta) F_\beta(r).
\end{align}
By Lemma \ref{9l3.2}, we get
\begin{align}\label{9e3.11}
K_{t}^{n,3}(\phi)-\int_0^t G_r(\phi) dr \text{ is a martingale}.
\end{align}
Replace $\phi(B^\beta)$ in \eqref{9e3.10} by $\phi(B^\beta_{(r-\tau_N)^+})$  and define
\begin{align}\label{9e3.12}
& G_r^\tau(\phi)=\frac{1}{N}  \sum_{\beta} 1(T_{\pi\beta}<r\leq T_\beta, \zeta_\beta^1>r-\tau_N)  \phi(B^\beta_{(r-\tau_N)^+}) F_\beta(r).
\end{align}
 \begin{lemma}\label{9l5.5}
Let $n=1$ or $2$. For all $0<t<\infty$, 
\begin{align*}
\lim_{N\to \infty} \E\Big(\sup_{s\leq t} \Big\vert K_s^{n,3}(\phi)-\int_0^t G_r^\tau(\phi)ds\Big\vert \Big)=0.
\end{align*}
\end{lemma}
The last approximation term is obtained by replacing $F_\beta(r)$ by $1$ in \eqref{9e3.12}, that is,
\begin{align}\label{9e3.13}
& X_r^{1,\tau}(\phi):=\frac{1}{N}  \sum_{\beta} 1(T_{\pi\beta}<r\leq T_\beta, \zeta_\beta^1>r-\tau_N)  \phi(B^\beta_{(r-\tau_N)^+}),
\end{align}
where $\{B^\beta\neq \Delta\}$ is implicitly given by letting $\phi(B^\beta_{(r-\tau_N)^+})=0$ if $B^\beta=\Delta$.
In order to compare $X_r^{1,\tau}(\phi)$ with $G_r^\tau(\phi)$, we define 
\begin{align*}
\cA(s)=\{\alpha: T_{\pi\alpha}<s\leq T_\alpha, \zeta_\alpha^1>s\}
\end{align*}
to be the set of particles alive in $X^1$. Let 
\begin{align}\label{9e4.3}
\{\alpha\}_r=\{\beta: \beta\geq \alpha, T_{\pi\beta}<r\leq T_\beta, B^\beta\neq \Delta\}
\end{align}
be the set of descendants of $\alpha$ alive at time $r$ in the branching random walk. Now rewrite $G_r^\tau(\phi)$ from \eqref{9e3.12} as
\begin{align} \label{e1.24}
& G_r^\tau(\phi)=\frac{1}{N}  \sum_{\alpha \in \cA(r-\tau_N)}   \phi(B^\alpha_{(r-\tau_N)^+}) \sum_{\beta\in \{\alpha\}_r} F_\beta(r).
\end{align}
Similarly we may rewrite $X_r^{1,\tau}(\phi)$ as
\begin{align} \label{e1.25}
& X_r^{1,\tau}(\phi)=\frac{1}{N}  \sum_{\alpha \in \cA(r-\tau_N)}   \phi(B^\alpha_{(r-\tau_N)^+}) \vert \{\alpha\}_r\vert.
\end{align}
Define
\begin{align} \label{9e4.2}
&Z_\alpha(r)=\sum_{\beta\in \{\alpha\}_r} F_\beta(r)\quad \text{ and } \quad b_d^\tau=\frac{\E(Z_1(\tau_N))}{\E(\vert \{1\}_{\tau_N}\vert )}.
 \end{align}
Combine \eqref{e1.24} and  \eqref{e1.25} to get
\begin{align} \label{9e6.61}
& G_r^\tau(\phi)-b_d^\tau X_r^{1,\tau}(\phi)=\frac{1}{N}  \sum_{\alpha \in \cA(r-\tau_N)}   \phi(B^\alpha_{(r-\tau_N)^+}) (Z_\alpha(r)-b_d^\tau \vert \{\alpha\}_r\vert ).
\end{align}
By translation invariance, we condition on $\cF_{r-\tau_N}$ to see that on the event $\{\alpha \in \cA(r-\tau_N)\}$,  
\begin{align} \label{9e3.14}
& \E\Big(\big(Z_\alpha(r)-b_d^\tau \vert \{\alpha\}_r\vert \big) \Big\vert \cF_{r-\tau_N}\Big)= \E\Big(Z_1(\tau_N)- b_d^\tau\vert \{1\}_{\tau_N}\vert \Big)=0.
\end{align}
  
 \begin{lemma}\label{9l5.6}
Let $n=1$ or $2$. For all $0<t<\infty$, 
\begin{align*}
\lim_{N\to \infty} \E\Big(\sup_{s\leq t} \Big\vert  \int_0^s G_r^\tau(\phi)dr-b_d^\tau \int_0^s X_r^{1,\tau}(\phi)dr\Big\vert \Big)=0.
\end{align*}
\end{lemma}

Finally, by showing the last two lemmas below, we arrive at our destinated term $b_d \int_0^s X_r^{1}(\phi)dr$.

 \begin{lemma}\label{9l5.7}
$\lim_{N\to\infty} b_d^\tau=b_d$.
\end{lemma}

 \begin{lemma}\label{9l5.8}
For all $0<t<\infty$, 
\begin{align*}
\lim_{N\to \infty} \E\Big(\sup_{s\leq t} \Big\vert   \int_0^s X_r^{1,\tau}(\phi)dr- \int_0^s X_r^{1}(\phi)dr\Big\vert \Big)=0.
\end{align*}
\end{lemma}

At this stage, we have finished the proof of Lemma \ref{9l2.6} and hence proved our main result Theorem \ref{9t0}. It remains to prove Lemmas \ref{9l5.1}-\ref{9l5.8}.

\section{First part of the proofs of lemmas}\label{9s6}

We first give the relatively ``simpler'' proofs of Lemmas \ref{9l5.1}, \ref{9l5.2},  \ref{9l5.3}, \ref{9l5.5} in this section.

\begin{proof}[Proof of Lemma \ref{9l5.1}]
Since $K_t^{n}(\phi)$ is increasing in $t$, we get
\begin{align*}
&\E\Big(\sup_{s\leq t} \Big\vert K_s^{n}(\phi)-\frac{N}{N+\theta}K_s^{n}(\phi)\Big\vert \Big)=\frac{\theta}{N+\theta}\E(K_t^{n}(\phi)).
\end{align*}
Recall from \eqref{9ea5.69} the inequality for $K_t^{n}(\phi)$. Together with the upper bound from Lemma \ref{9l2.5}, we conclude from the above that
\begin{align*}
\E\Big(\sup_{s\leq t} \Big\vert K_s^{n}(\phi)-\frac{N}{N+\theta}K_s^{n}(\phi)\Big\vert \Big)\leq \frac{C}{N} \|\phi\|_\infty (X_0^0(1)+X_0^0(1)^2)\to 0.
\end{align*}
Next, define for each $s\geq 0$
\begin{align*} 
\hat{K}_{s}^{n,0}(\phi)=&\frac{1}{2N+\theta}  \sum_{\beta} a_\beta^n(t)  \phi(B^\beta) 1(B^\beta+W^\beta\in \overline{\cR}^{X^{n-1}}_{T_\beta^-}).
\end{align*}
Then the difference between $\frac{N}{N+\theta}K_s^{n}(\phi)$ and $\hat{K}_s^{n,0}(\phi)$ is that $\phi(Y^\beta+W^\beta)$  is replaced by $\phi(Y^\beta)$. Since $\phi \in C_b^3(\R^d)$ and $W^\beta$ is uniform on $\cN_N\subseteq [-N^{-1/2}, N^{-1/2}]^d$, by setting
\begin{align}\label{9ec5.37}
\eta_N=\sup {\{\vert \phi(x+y)-\phi(x)\vert : y\in [-N^{-1/2}, N^{-1/2}]^d, {x\in \R^d}\}},
\end{align}
we have $\vert \phi(Y^\beta+W^\beta)-\phi(Y^\beta)\vert \leq \eta_N$ and $\eta_N\leq CN^{-1/2}$. It follows that
\begin{align*}
\E\Big(\sup_{s\leq t} \Big\vert \frac{N}{N+\theta}K_s^{n}(\phi)-\hat{K}_s^{n,0}(\phi)\Big\vert \Big)
&\leq C \eta_N \E\Big(  \frac{1}{N} \sum_{\beta} a_\beta^0(t)  1(B^\beta+W^\beta\in \overline{\cR}^{X^{0}}_{T_\beta^-}) \Big)\\
&\leq CN^{-1/2} C(X_0^0(1)+X_0^0(1)^2) \to 0,
\end{align*}
where the last inequality uses Lemma \ref{9l2.5} again. Finally, we recall the assumption on $K_0^N$ from \eqref{9e0.93} to see
\begin{align*}
&\E\Big(\sup_{s\leq t} \Big\vert \hat{K}_s^{n,0}(\phi)-K_s^{n,0}(\phi)\Big\vert \Big) \leq \|\phi\|_\infty \E\Big(  \frac{1}{N} \sum_{\beta} a_\beta^0(t)  1(B^\beta+W^\beta\in K_0^N) \Big)\to 0,
\end{align*}
thus completing the proof.
\end{proof}

\begin{lemma}\label{9l6.1}
If $G_\beta$ is measurable with respect to $\cF_{T_\beta}$ that satisfies 
\begin{align*}\E(G_\beta\vert \cF_{T_\beta^-})=0,\text{ and } \vert G_\beta\vert \leq K 1(T_\beta\leq \zeta_\beta^0)\end{align*}
for some absolute constant $K>0$,
 then $M_t=N^{-1}\sum_{\beta} 1(T_\beta\leq t) G_\beta$ is an $\cF_t$-martingale and there is some constant $C>0$ so that for any $t\geq 0$,
 \begin{align*}
\E\Big(\sup_{s\leq t} M_s^2\Big) \leq C\frac{1}{N^2} \E\Big(\sum_\beta 1(T_\beta \leq t) G_\beta^2\Big).
 \end{align*}
\end{lemma}
\begin{proof}
The proof is immediate from Lemma 3.5 of \cite{DP99} and its proof therein.
\end{proof}
\begin{proof}[Proof of Lemma \ref{9l5.2}]
Recall $K_t^{n,1}(\phi)$ from \eqref{9e3.3} and define
\begin{align*}
M_t&:=K_t^{n,0}(\phi)-K_t^{n,1}(\phi) =\frac{1}{2N+\theta}\sum_{\beta} a_\beta^n(t) \phi(B^\beta)   \Big(h_{\beta}^{n-1}-\frac{\nu_{n-1}(\beta)}{\psi(N)}\Big),
\end{align*}
where 
\begin{align*}
h_{\beta}^{n-1}=1(B^\beta+W^\beta\in \cR^{X^{n-1}}_{T_\beta^-}).
\end{align*}
By \eqref{9e3.31} and Lemma \ref{9l6.1}, we get $(M_t, t\geq 0)$ is an $\cF_t$-martingale and 
\begin{align*} 
\E\Big(\sup_{s\leq t}M_s^2\Big)\leq \frac{C}{N^2} \| \phi\| _\infty^2 \sum_{\beta} \E\Big( a_\beta^0(t)    \Big(h_{\beta}^{n-1}-\frac{\nu_{n-1}(\beta)}{\psi(N)}\Big)^2\Big).
\end{align*}
Notice that $h_{\beta}^{n-1} \in \{0,1\}$. Using \eqref{9e3.31} again, we have  
\begin{align*}
\E\Big(   \Big(h_{\beta}^{n-1}-\frac{\nu_{n-1}(\beta)}{\psi(N)}\Big)^2\Big\vert \cF_{T_\beta^-}\Big)\leq \E(h_{\beta}^{n-1}\vert \cF_{T_\beta^-}),
\end{align*}
thus giving
\begin{align*}
\E\Big(\sup_{s\leq t}M_s^2\Big)\leq &\frac{C}{N^2} \| \phi\| _\infty^2 \sum_{\beta} \E\Big( a_\beta^0(t)  h_{\beta}^{n-1}\Big)\leq C\| \phi\| _\infty^2 \frac{1}{N}(X_0^0(1)+X_0^0(1)^2) \to 0,
\end{align*}
where the last inequality is by Lemma \ref{9l2.5}. The proof is complete.
\end{proof}

\begin{proof}[Proof of Lemma \ref{9l5.3}]
Recall $\nu_m(\beta)$ from \eqref{9e3.33} and define
\begin{align*}
\nu_{m,\tau}(\beta)=\Big\vert \{B^\gamma: T_{\pi \gamma}<T_\beta, B^\gamma-B^\beta \in \cN_N, \zeta_{\gamma}^m>T_{\pi \gamma}, T_{\gamma \wedge \beta}>T_\beta-\tau_N\}\Big\vert .
\end{align*}
Replace $\nu_{n-1}(\beta)$ in $K_{t}^{n,1}(\phi)$ by $\nu_{n-1,\tau}(\beta)$ and set
\begin{align*}
\hat{K}_{t}^{n,1}(\phi)=&\frac{1}{2N+\theta}   \sum_{\beta} a_\beta^n(t)  \phi(B^\beta)  \frac{\nu_{n-1,\tau}(\beta)}{\psi(N)}.
\end{align*}
It follows that
\begin{align*}
&\sup_{s\leq t}\vert \hat{K}_{s}^{n,1}(\phi)-{K}_{s}^{n,1}(\phi)\vert \leq \frac{1}{N}   \frac{1}{\psi(N)} \| \phi\| _\infty   \sum_{\beta} a_\beta^0(t) \vert \nu_{n-1}(\beta)-\nu_{n-1,\tau}(\beta)\vert \nn\\
&\leq \frac{1}{N\psi(N)} \| \phi\| _\infty \sum_{\beta} 1(T_\beta\leq t)\nn\\
&\quad\quad\quad \sum_{\gamma} 1(T_{\pi \gamma}<T_\beta) 1( B^\gamma-B^\beta \in \cN_N)1( T_{\gamma \wedge \beta}\leq T_\beta-\tau_N)\nn\\
&= \| \phi\| _\infty(J_0(t)+J(t,\tau_N)),
\end{align*}
where the last equality is from \eqref{9e2.25} and \eqref{9e2.26}. By Lemma \ref{9l4.1a}, Lemma \ref{9l4.3a} and \eqref{9ec3.5}, we conclude  
\begin{align*}
\lim_{N\to \infty} \E\Big(\sup_{s\leq t} \Big\vert \hat{K}_s^{n,1}(\phi)-K_s^{n,1}(\phi)\Big\vert \Big)=0.
\end{align*}
It remains to calculate the difference between $\hat{K}_{t}^{n,1}(\phi)$ and ${K}_{t}^{n,2}(\phi)$ given by
\begin{align}\label{9e3.41}
\sup_{s\leq t} \vert \hat{K}_{s}^{n,1}(\phi)-{K}_{s}^{n,2}(\phi)\vert &\leq \frac{1}{N}   \frac{1}{\psi(N)} \| \phi\| _\infty \sum_{\beta} 1(T_\beta\leq t, B^\beta\neq \Delta)  \nn\\
&\times   \vert \nu_{n-1,\tau}(\beta)-\sum_{\gamma}\text{nbr}_{\beta,\gamma}^{n-1} 1(T_{\gamma \wedge \beta}>T_\beta-\tau_N)\vert .
\end{align}
Define
\begin{align}\label{9e3.42}
J_1(t)=&\frac{1}{N\psi(N)} \sum_{\beta} 1(T_\beta\leq t, B^\beta\neq \Delta)\sum_{\gamma}  1(T_{\pi\gamma}<T_\beta, B^\gamma-B^\beta \in \cN_N) \nn\\
&\quad 1(T_{\gamma \wedge \beta}>T_\beta-\tau_N) \sum_{\delta}1(T_{\pi\delta} <T_\beta, B^\delta=B^\gamma)1(T_{\delta \wedge \beta}>T_\beta-\tau_N).
\end{align}
We claim  
\begin{align}\label{9ec3.42}
\sup_{s\leq t}\vert \hat{K}_{s}^{n,1}(\phi)-{K}_{s}^{n,2}(\phi)\vert \leq C\| \phi\| _\infty J_1(t).
\end{align}
To see this, we note the last term on the right-hand side of \eqref{9e3.41} is from the multiple occupancy of particles: If $k\geq 2$ particles have visited a neighboring site of $B^\beta$, then this will contribute at most $k-1$ to \eqref{9e3.41} and at least $k^2$ to the summation of $\gamma,\delta$ in \eqref{9e3.42}, thus giving \eqref{9ec3.42}. It suffices to show $\E(J_1(t)) \to 0$. \\

Apply Lemma \ref{9l3.2} to get
\begin{align} \label{9e4.31}
\E(J_1(t))=&\frac{2N+\theta}{N\psi(N)} \int_0^t \sum_{\beta,\gamma, \delta}\E\Big(1(T_{\pi\beta}<r\leq T_\beta)  1(T_{\pi\gamma}<r, B^\gamma-B^\beta \in \cN_N)\nn\\
&\quad 1(T_{\gamma \wedge \beta}>r-\tau_N)1(T_{\delta \wedge \beta}>r-\tau_N)1(T_{\pi\delta} <r) 1(B^\delta=B^\gamma)\Big)dr.
\end{align}
Denote the terms inside the expectation above by $K_{\beta,\gamma, \delta}(r)$. The indicators $1(T_{\gamma \wedge \beta}>r-\tau_N)$ and $1(T_{\delta \wedge \beta}>r-\tau_N)$ ensure that $\delta_0=\beta_0=\gamma_0$.   By translation invariance, we get
\begin{align} \label{e1.31} 
\sum_{\beta,\gamma, \delta} \E\Big(K_{\beta,\gamma, \delta}(r)\Big) =    NX_0^0(1)\E\Big(\sum_{\beta,\gamma, \delta\geq 1}K_{\beta,\gamma, \delta}(r)\Big).
\end{align}
 
\begin{lemma}\label{9l5.0}
There is some constant $C>0$ such that for any $0< r\leq t$,
 \begin{align} \label{9e4.61}
\E\Big(\sum_{\beta,\gamma, \delta\geq 1} K_{\beta,\gamma, \delta}(r)\Big) \leq   C\tau_N.
\end{align}
\end{lemma}
Use \eqref{e1.31} and \eqref{9e4.61} in \eqref{9e4.31} to get
\begin{align} \label{9e4.60}
\E(J_1(t))\leq &\frac{2N+\theta}{N\psi(N)} \int_0^t NX_0^0(1)   C\tau_N  dr\leq \frac{CX_0^0(1)}{\psi_0(N)} \tau_N \to 0 \text{ as } N\to \infty.
\end{align}
The proof of Lemma \ref{9l5.3} is complete as noted above. 
\end{proof}
It remains to prove Lemma \ref{9l5.0}.
\begin{proof}[Proof of Lemma \ref{9l5.0}]
By symmetry of $\delta$ and $\gamma$, we may assume   $\delta \wedge \beta \leq \gamma\wedge \beta$. There are two cases here: (i) $\delta \wedge \beta < \gamma\wedge \beta$; (ii) $\delta \wedge \beta = \gamma\wedge \beta$. Denote by $K^{(i)}(r)$ (resp. $K^{(ii)}(r)$) for the sum of $K_{\beta,\gamma, \delta}(r)$ when $\delta,\gamma,\beta$ satisfy case (i) (resp. case (ii)).  \\
\begin{figure}[ht]
  \begin{center}
    \includegraphics[width=0.75 \textwidth]{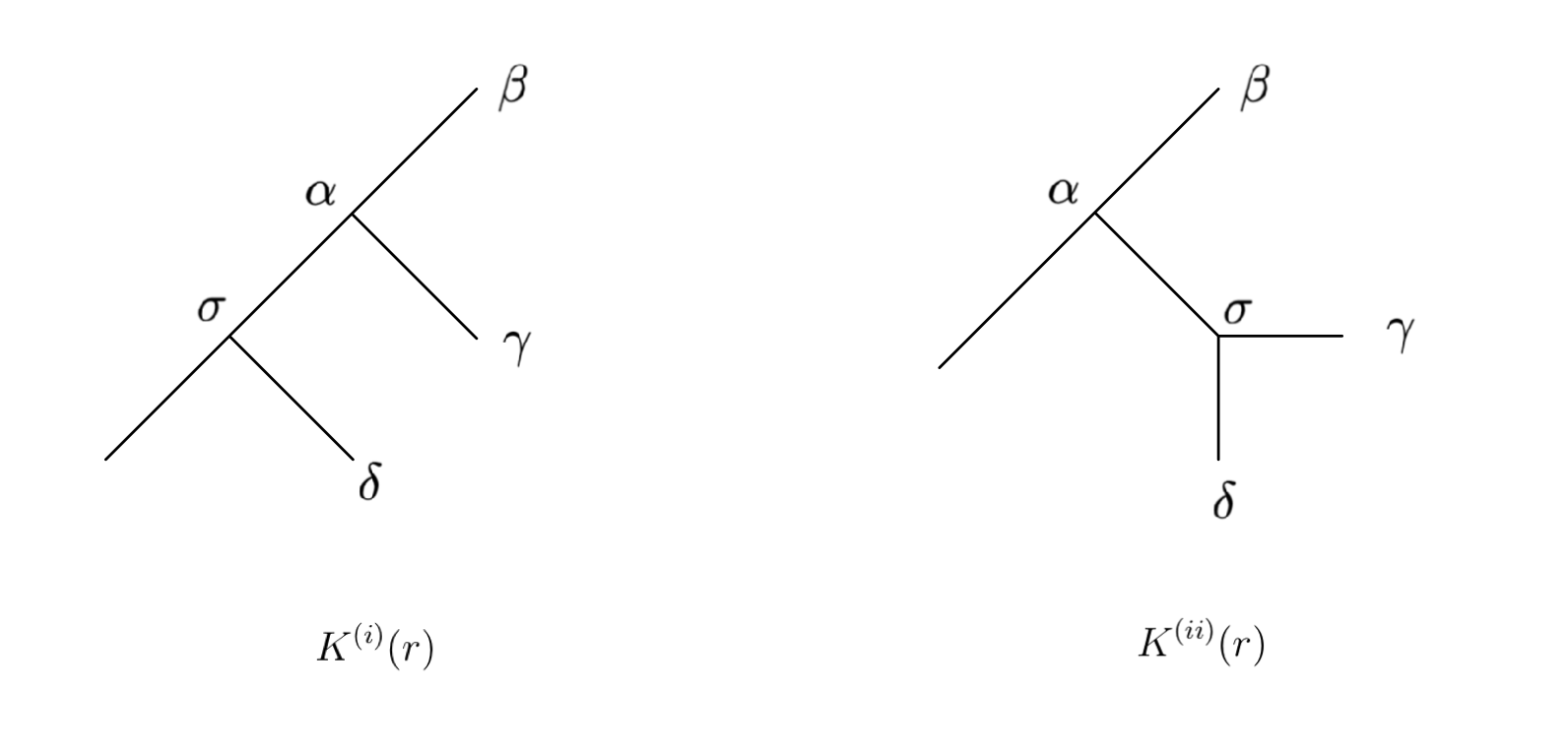}
    \caption[Branching Particle System]{\label{fig1}   Two cases for $K_{\beta,\gamma, \delta}(r)$.
      }
  \end{center}
\end{figure} 

\no {\bf Case (i).} We first deal with $K^{(i)}(r)$. Let $\sigma=\delta\wedge \beta$ and $\alpha=\gamma\wedge \beta$ so that $\sigma<\alpha$. Let $\vert \sigma\vert =k$ for some $k\geq 0$ and   $\vert \alpha\vert =k+j$ for some $j\geq 1$. Since $\delta\geq \sigma$, we let $\vert \delta\vert =k+n$ for some $n\geq 0$. Set $\vert \beta\vert =\vert \alpha\vert +l$ and $\vert \gamma\vert =\vert \alpha\vert +m$ for some $l,m\geq 0$.  See Figure \ref{fig1} for illustration. The sum of $\delta, \gamma, \beta$ for $K^{(i)}(r)$ can be written as 
\begin{align}   \label{9e4.32}
\E(K^{(i)}(r))=& \sum_{k=0}^\infty \sum_{\substack{\sigma\geq 1,\\ \vert \sigma\vert =k}}  \sum_{n=0}^\infty \sum_{\substack{\delta\geq \sigma,\\ \vert \delta\vert =k+n}}  
\sum_{j=1}^\infty\sum_{\substack{\alpha\geq \sigma,\\ \vert \alpha\vert =k+j}}   \sum_{l=0}^\infty   \sum_{m=0}^\infty\sum_{\substack{\beta\geq \alpha,\\ \vert \beta\vert =\vert \alpha\vert +l}}\sum_{\substack{ \gamma\geq\alpha,\\ \vert \gamma\vert =\vert \alpha\vert +m}} \nn\\
&\E\Big(1(T_{\pi\beta}<r\leq T_\beta)  1(T_{\pi\gamma}<r, B^\gamma-B^\beta \in \cN_N)\nn\\
&\quad  1(T_{\sigma}>r-\tau_N)1(T_{\pi\delta} <r, B^\delta=B^\gamma)\Big),
\end{align}
where we have removed $1(T_{\alpha}>r-\tau_N)$ due to $1(T_{\sigma}>r-\tau_N)$ and $\sigma<\alpha$. 
 We note on the event $\{B^\beta, B^\gamma, B^{\delta}\neq \Delta\}$,  the spatial events $\{B^\beta-B^\gamma\in \cN_N\}$ and $\{B^\delta=B^\gamma\}$ are independent of the branching events concerning $T_{\alpha}, T_\beta$, etc. Therefore the expectation above is at most
 \begin{align} \label{9ec1.00}  
&\E\Big(1(T_{\pi\beta}<r\leq T_\beta)  1(T_{\pi\gamma}<r) 1(T_{\sigma}>r-\tau_N)1(T_{\pi\delta} <r)\Big)  \nn\\
&\times (\frac{N+\theta}{2N+\theta})^{k+j+l+m+n}  \times  \E^{\beta, \gamma,{\delta}}\Big(1(B^\beta-B^\gamma\in \cN_N)  1(B^\delta=B^\gamma) \Big),
\end{align}
where the second term gives an upper bound for the probability of the event $\{B^\beta, B^\gamma, B^{\delta}\neq \Delta\}$. We also use $\E^{\beta, \gamma,{\delta}}$ to denote the  expectation conditioning on $\{B^\beta, B^\gamma, B^{\delta}\neq \Delta\}$. 

For the first expectation above concerning the branching events, we condition on $\cH_\alpha$ to get  
 \begin{align*}   
&\E\Big(1(T_{\pi\beta}<r\leq T_\beta)  1(T_{\pi\gamma}<r) 1(T_{\sigma}>r-\tau_N)1(T_{\pi\delta} <r)\Big\vert \cH_\alpha\Big)  \\
&\leq 1(T_\alpha<r) \cdot 1(T_{\sigma}>r-\tau_N) \cdot \P(T_{\pi\beta}-T_{\alpha} <r-T_{\alpha} \leq T_\beta-T_{\alpha} \vert \cH_\alpha)\nn\\
&\quad \quad\cdot  \P(T_{\pi\gamma}-T_{\alpha} <r-T_{\alpha}<\tau_N \vert \cH_\alpha)  \cdot  \P(T_{\pi\delta}-T_{\sigma} <r-T_{\sigma}<\tau_N\vert \cH_\sigma)  \\
&= 1(T_\alpha<r) \cdot 1(T_{\sigma}>r-\tau_N) \cdot\pi((2N+\theta)(r-T_\alpha), l-1) \nn\\
&\quad \quad\cdot  \Pi((2N+\theta)\tau_N,m-1) \cdot \Pi((2N+\theta)\tau_N,n-1),
\end{align*}
where in the first inequality we have used $T_{\alpha}>T_{\sigma}>r-\tau_N$.
So the first expectation in \eqref{9ec1.00} is at most
 \begin{align} \label{9ec1.12}    
&\E\Big(1(\Gamma_{k+j+1}<(2N+\theta)r) \cdot1(\Gamma_{k+1}>(2N+\theta)(r-\tau_N))  \cdot\pi((2N+\theta)r-\Gamma_{k+j+1}, l-1)\Big) \nn\\
&\quad \quad\times  \Pi((2N+\theta)\tau_N,m-1) \times \Pi((2N+\theta)\tau_N,n-1).
\end{align}

Turning to the second expectation in \eqref{9ec1.00} concerning the spatial events, we let $\hat{B}^\beta$ (resp. $\hat{B}^\gamma$) be the position of $B^\beta-B^{\alpha}$ (resp. $B^\gamma-B^{\alpha}$) with the value of $W^{\alpha}$ subtracted if it shows up in $B^\beta$ (resp. ${B}^\gamma$).  To be specific, we define (recall $B^\beta$ from \eqref{ea3.22})
\begin{align}  \label{9ea1.71}
\hat{B}^\beta=\sum_{j=\vert \alpha\vert +1}^{\vert \beta\vert -1} W^{\beta\vert j}1(e_{\beta\vert j}=\beta_{j+1})  \text{ and } \hat{B}^\gamma=\sum_{j=\vert \alpha\vert +1}^{\vert \gamma\vert -1} W^{\gamma\vert j} 1(e_{\gamma\vert j}=\gamma_{j+1}) .
\end{align} 
It follows that $B^\beta-B^\gamma-(\hat{B}^\beta- \hat{B}^\gamma)=W^\alpha$ and hence
\begin{align}   \label{9ea2.82}
1(B^\gamma-B^\beta \in \cN_N)\leq 1(\sqrt{N}(\hat{B}^\beta- \hat{B}^\gamma)\in  [-2,2]^d).
\end{align}  
Use the above to get the second expectation in \eqref{9ec1.00} is bounded by
\begin{align}\label{e1.10}
\E^{\beta, \gamma,{\delta}}\Big(1(\sqrt{N}(\hat{B}^\beta- \hat{B}^\gamma)\in  [-2,2]^d) 1(B^\delta=B^\gamma) \Big)
\end{align}  
Next, we claim that
\begin{align}\label{9ec1.10}
\P^{\beta, \gamma,{\delta}}\Big(B^\delta=B^\gamma\Big|\sigma(\hat{B}^\beta, \hat{B}^\gamma)\Big)\leq \frac{1}{\psi(N)} \frac{C}{(1+j+n)^{d/2}}.
\end{align}
To see this, we recall $\sigma=\delta\wedge \gamma$ and $W^\sigma$ is a jump that might have moved the $\delta$ or the $\gamma$ line at time $T_\sigma$ so that $B^\delta_{T_\sigma}-B^\gamma_{T_\sigma}\in \{W^\sigma, -W^\sigma\}$. In order that $B^\delta=B^\gamma$, we  require
\begin{align}\label{9ec1.14}
(B^\delta-B^\delta_{T_\sigma})-(B^\gamma-B^\gamma_{T_\sigma}) \in \cN_N,
\end{align}
and then let $W^\sigma$ take the exact value in $\cN_N$ to make $B^\delta=B^\gamma$, which occurs with probability $1/\psi(N)$.  Notice that
\begin{align*} 
&\P^{\beta, \gamma,{\delta}}\Big((B^\delta-B^\delta_{T_\sigma})-(B^\gamma-B^\gamma_{T_\sigma}) \in \cN_N\Big|\sigma(\hat{B}^\beta, \hat{B}^\gamma)\Big)\\
&=\P^{\beta, \gamma,{\delta}}\Big((B^\delta-B^\delta_{T_\sigma})-(B^\gamma-\hat{B}^\gamma-B^\gamma_{T_\sigma}) \in \hat{B}^\gamma+\cN_N\Big|\sigma(\hat{B}^\beta, \hat{B}^\gamma)\Big)\leq   \frac{C}{(1+j+n)^{d/2}},\nn
\end{align*}
where the last inequality uses  \eqref{ea3.22} and Lemma \ref{9e4.32}. Hence the desired inequality in \eqref{9ec1.10} follows.
 
Apply \eqref{9ec1.10} to see that \eqref{e1.10} is at most
\begin{align}\label{9ec1.11}
 &\frac{C}{\psi(N)} \frac{1}{(1+j+n)^{d/2}}\E^{\beta, \gamma,{\delta}}\Big(1(\sqrt{N}(\hat{B}^\beta- \hat{B}^\gamma)\in  [-2,2]^d)   \Big)\nn\\
 &\leq  \frac{C}{\psi(N)} \frac{1}{(1+j+n)^{d/2}}\frac{1}{(1+l+m)^{d/2}},
\end{align}  
where we have used Lemma \ref{9l4.2} and \eqref{9ea1.71} in the last inequality. Now conclude from  \eqref{e1.10} and \eqref{9ec1.11} that
 \begin{align} \label{9ec1.13}
& \E^{\beta, \gamma,{\delta}}\Big(1(B^\beta-B^\gamma\in \cN_N)  1(B^\delta=B^\gamma) \Big)\leq  \frac{C}{\psi(N)} \frac{1}{(1+l+m)^{d/2}}\frac{1}{(1+j+n)^{d/2}}.
\end{align}

 Combine \eqref{9ec1.00}, \eqref{9ec1.12} and  \eqref{9ec1.13} to see that \eqref{9e4.32} becomes 
 \begin{align} \label{9e9.69}
&\E( K^{(i)}(r))\leq  \frac{C}{\psi(N)} \sum_{k=0}^\infty  \sum_{j=1}^\infty\sum_{n=0}^\infty \sum_{l=0}^\infty \sum_{m=0}^\infty (\frac{2N+2\theta}{2N+\theta})^{k+j+l+m+n}  \nn\\
&\quad \times  \frac{1}{(1+j+n)^{d/2}} \frac{1}{(1+l+m)^{d/2}}\E\Big(1(\Gamma_{k+j+1}<(2N+\theta)r)  \nn\\
&\quad \times  1(\Gamma_{k+1}>(2N+\theta)(r-\tau_N))  \cdot \pi((2N+\theta)r-\Gamma_{k+j+1}, l-1)\Big) \nn\\
&\quad \times  \Pi((2N+\theta)\tau_N,m-1) \times \Pi((2N+\theta)\tau_N,n-1).  
\end{align}
By Lemma \ref{9l4.0}, the sum of $m$ gives
 \begin{align*} 
&\sum_{m=0}^\infty  (\frac{2N+2\theta}{2N+\theta})^{m}  \Pi((2N+\theta)\tau_N ,m-1)\frac{1}{(1+l+m)^{d/2}}  \leq \frac{C}{(1+l)^{d/2-1}}\leq C.
\end{align*}
The sum for $l$ is at most
 \begin{align} \label{9ec1.19}
&\sum_{l=0}^\infty (\frac{2N+2\theta}{2N+\theta})^{l}\pi((2N+\theta)r-\Gamma_{k+j+1}, l-1)\nn\\
&\leq 2e^{\theta r}\sum_{l=0}^\infty \pi\Big((2N+2\theta)r-\frac{2N+2\theta}{2N+\theta}\Gamma_{k+j+1}, l\Big) =2e^{\theta r}.
\end{align}
Next, turning to the sum of $n$,  we use Lemma \ref{9l4.0} again to see
 \begin{align*} 
&\sum_{n=0}^\infty  (\frac{2N+2\theta}{2N+\theta})^{n}  \Pi((2N+\theta)\tau_N ,n-1)\frac{1}{(n+j+1)^{d/2}}  \leq \frac{Ce^{\theta \tau_N}}{(1+j)^{d/2-1}}.
\end{align*}
 Combine the above to see that  
 \begin{align}\label{9ec2.32}  
\E( K^{(i)}(r))&\leq  \frac{C}{\psi(N)}   \sum_{k=0}^\infty  \sum_{j=1}^\infty  (\frac{2N+2\theta}{2N+\theta})^{k+j}  \frac{1}{(j+1)^{d/2-1}} \nn\\
& \E\Big(1(\Gamma_{k+j+1}<(2N+\theta)r) \cdot  1(\Gamma_{k+1}>(2N+\theta)(r-\tau_N))  \Big),
\end{align}
Bound the expectation on the right-hand side  by 
 \begin{align} \label{9ec1.02}
&\E\Big(1((2N+\theta)(r-\tau_N)<\Gamma_{k+1}<(2N+\theta)r) \cdot  1(\Gamma_{k+j+1}-\Gamma_{k+1}<(2N+\theta)\tau_N)  \Big)\nn\\
&=\P((2N+\theta)(r-\tau_N)<\Gamma_{k+1}<(2N+\theta)r)\cdot   \Pi((2N+\theta)\tau_N, j).
\end{align}
So \eqref{9ec2.32} becomes
 \begin{align} \label{e2.32}
\E( K^{(i)}(r))\leq  \frac{C}{\psi(N)}   &\sum_{k=0}^\infty (\frac{2N+2\theta}{2N+\theta})^{k}\P((2N+\theta)(r-\tau_N)<\Gamma_{k+1}<(2N+\theta)r)    \nn\\
& \times \sum_{j=1}^\infty  (\frac{2N+2\theta}{2N+\theta})^{j} \Pi((2N+\theta)\tau_N, j)  \frac{1}{(1+j)^{d/2-1}} .
\end{align}

\begin{lemma}\label{9l8.7}
For any $t>0$, there is some constant $C>0$ so that for any $0< r\leq t$,  
\begin{align*}
&\sum_{m=0}^\infty   (\frac{2N+2\theta}{2N+\theta})^{m}  \Pi((2N+\theta)r, m) \frac{1}{(1+m)^{d/2-1}}\leq C I(N).
\end{align*}
\end{lemma}
\begin{proof}
The $m=0$ term gives $1$. For $m\geq 0$, write $\Pi((2N+\theta)r, m)$ as $\P(\Gamma_{m}< (2N+\theta)r)$. Then the sum of $m>ANr$ is at most $C2^{-Nr}$ by Lemma \ref{9l4.3}. Next, the sum of $j\leq ANr$ is at most 
\begin{align*}
&\sum_{m=0}^{ANr}   e^{A\theta r}   \frac{1}{(1+m)^{d/2-1}}\leq Ce^{A\theta t} I(ANr)\leq CI(N).
\end{align*}
The proof is complete.
\end{proof}

By the above lemma, the sum of $j$ in \eqref{e2.32} gives at most $CI(N)$.
The sum of $k$ gives
 \begin{align} \label{9ec1.04}
&\sum_{k=0}^\infty  (\frac{2N+2\theta}{2N+\theta})^{k} \P((2N+\theta)(r-\tau_N)<\Gamma_{k+1}<(2N+\theta)r)  \nn \\
&=\sum_{k=0}^\infty  \sum_{\alpha\geq 1, \vert \alpha\vert =k} \E\Big(1(r-\tau_N<T_\alpha<r, B^\alpha\neq\Delta)\Big)\leq  Ce^{\theta r}N\tau_N,
\end{align}
where the last inequality is by Lemma \ref{9l3.2a}. Therefore \eqref{e2.32} becomes 
  \begin{align*} 
\E( K^{(i)}(r))&\leq  \frac{C}{\psi(N)}    C I(N)  N\tau_N\leq C\tau_N.
\end{align*}

\no {\bf Case (ii).}  Next, we move to case (ii) where $\delta\wedge \beta=\gamma\wedge \beta$. Since we assume $\delta\wedge \beta\leq \gamma\wedge \beta$, we must have $\delta$ branches off $\gamma$ after $\gamma$ branches off $\beta$. Let $\alpha=\gamma\wedge \beta$ and $\vert \alpha\vert =k$ for some $k\geq 0$. Let $l\geq 0$ such that $\vert \beta\vert =k+l$. Set $\sigma=\delta\wedge \gamma$ and let $\vert \sigma\vert =k+j$ for some $j\geq 0$.  Let $n,m\geq 0$ so that $\vert \delta\vert =\vert \sigma\vert +n$ and $\vert \gamma\vert =\vert \sigma\vert +m$.  See Figure \ref{fig1}. Now we may write the sum of $\E(K^{(ii)}(r))$ as
\begin{align*}  
\E(K^{(ii)}(r))=& \sum_{k=0}^\infty \sum_{\substack{\alpha\geq 1,\\ \vert \alpha\vert =k}}  \sum_{l=0}^\infty \sum_{\substack{\beta\geq \alpha,\\ \vert \beta\vert =k+l}} \sum_{j=0}^\infty\sum_{\substack{\sigma\geq \alpha,\\ \vert \sigma\vert =k+j}} \sum_{n=0}^\infty \sum_{\substack{\delta\geq \sigma,\\ \vert \delta\vert =k+j+n}}  
     \sum_{m=0}^\infty\sum_{\substack{ \gamma\geq\alpha,\\ \vert \gamma\vert =k+j+m }} \nn\\
&\E\Big(1(T_{\pi\beta}<r\leq T_\beta)  1(T_{\pi\gamma}<r, B^\gamma-B^\beta \in \cN_N)\nn\\
&\quad \quad 1(T_{\alpha}>r-\tau_N)1(T_{\pi\delta} <r, B^\delta=B^\gamma)\Big).
\end{align*}
Similar to \eqref{9ec1.00}, the expectation above is at most
 \begin{align} \label{9ec2.44}
&\E\Big(1(T_{\pi\beta}<r\leq T_\beta)  1(T_{\pi\gamma}<r) 1(T_{\alpha}>r-\tau_N)1(T_{\pi\delta} <r)\Big)  \nn\\
&\times (\frac{N+\theta}{2N+\theta})^{k+j+l+m+n}  \times  \E^{\beta, \gamma,{\delta}}\Big(1(B^\beta-B^\gamma\in \cN_N)  1(B^\delta=B^\gamma) \Big).
\end{align}
For the first expectation above, we condition on $\cH_\sigma$ to get  
 \begin{align*}   
&\E\Big(1(T_{\pi\beta}<r\leq T_\beta)  1(T_{\pi\gamma}<r) 1(T_{\alpha}>r-\tau_N)1(T_{\pi\delta} <r)\Big\vert \cH_\sigma\Big)  \\
&\leq 1(T_\sigma<r) \cdot 1(T_{\alpha}>r-\tau_N) \cdot \P(T_{\pi\beta}-T_{\alpha} <r-T_{\alpha} \leq T_\beta-T_{\alpha} \vert \cH_\alpha)\nn\\
&\quad \quad\times  \P(T_{\pi\gamma}-T_{\sigma} <r-T_{\sigma}<\tau_N \vert \cH_\sigma)  \times  \P(T_{\pi\delta}-T_{\sigma} <r-T_{\sigma}<\tau_N\vert \cH_\sigma)  \\
&\leq  1(T_\sigma<r) \cdot 1(T_{\alpha}>r-\tau_N) \cdot  \pi((2N+\theta)(r-T_\alpha), l-1) \nn\\
&\quad \quad\times  \Pi((2N+\theta)\tau_N,m-1) \times \Pi((2N+\theta)\tau_N,n-1),
\end{align*}
where in the first inequality we have used $T_\sigma\geq T_\alpha>r-\tau_N$.
Hence the first expectation in \eqref{9ec2.44} is at most
  \begin{align}    \label{e2.44}
&\E\Big(1(\Gamma_{k+j+1}<(2N+\theta)r) \cdot1(\Gamma_{k+1}>(2N+\theta)(r-\tau_N))  \cdot\pi((2N+\theta)r-\Gamma_{k+1}, l-1)\Big) \nn\\
&\quad \quad\times  \Pi((2N+\theta)\tau_N,m-1) \times \Pi((2N+\theta)\tau_N,n-1).
\end{align}
For the second expectation in \eqref{9ec2.44}, we follow \eqref{9ea1.71} to define (recall $\sigma=\gamma \wedge \delta$)
\begin{align}  \label{e1.71}
\bar{B}^\delta=\sum_{j=\vert \sigma\vert +1}^{\vert \delta\vert -1} W^{\delta\vert j}1(e_{\delta\vert j}=\delta_{j+1})  \text{ and } \bar{B}^\gamma=\sum_{j=\vert \sigma\vert +1}^{\vert \gamma\vert -1} W^{\gamma\vert j} 1(e_{\gamma\vert j}=\gamma_{j+1})
\end{align} 
so that $B^\delta-B^\gamma-(\hat{B}^\delta- \hat{B}^\gamma)=W^\sigma$. Hence
\begin{align*}   
1(B^\delta=B^\beta)= 1(\bar{B}^\delta- \bar{B}^\gamma=W^\sigma).
\end{align*}  
Condition on $\sigma(\bar{B}^\delta, \bar{B}^\gamma,W^\sigma)$ to get
\begin{align*}
&\P^{\beta, \gamma,{\delta}}(B^\gamma-B^\beta\in \cN_N\vert \sigma(\bar{B}^\delta, \bar{B}^\gamma,W^\sigma))\\
&= \P^{\beta, \gamma,{\delta}}(B^\beta-B^\alpha-(B^\gamma-\bar{B}^\gamma-B^\alpha)\in \bar{B}^\gamma+\cN_N\vert \sigma(\bar{B}^\delta, \bar{B}^\gamma,W^\sigma))\leq \frac{C}{(1+j+l)^{d/2}}.
\end{align*}
It follows that
 \begin{align} \label{e2.82}
&\E^{\beta, \gamma,{\delta}}\Big(1(B^\beta-B^\gamma\in \cN_N)  1(B^\delta=B^\gamma) \Big)\nn\\
&\leq \frac{C}{(1+j+l)^{d/2}} \P^{\beta, \gamma,{\delta}}(\bar{B}^\delta- \bar{B}^\gamma=W^\sigma)\leq \frac{C}{(1+j+l)^{d/2}}\frac{1}{\psi(N)}\frac{C}{(1+m+n)^{d/2}}.
\end{align}
Combine \eqref{9ec2.44}, \eqref{e2.44} and \eqref{e2.82} to conclude
 \begin{align*}  
&\E( K^{(ii)}(r))\leq  \frac{C}{\psi(N)} \sum_{k=0}^\infty  \sum_{l=0}^\infty \sum_{j=0}^\infty\sum_{n=0}^\infty \sum_{m=0}^\infty (\frac{2N+2\theta}{2N+\theta})^{k+j+l+m+n}  \nn\\
&\frac{1}{(1+n+m)^{d/2}} \frac{1}{(1+j+l)^{d/2}}\E\Big(1(\Gamma_{k+j+1}<(2N+\theta)r)  \nn\\
& \quad  \times 1(\Gamma_{k+1}>(2N+\theta)(r-\tau_N))  \cdot\pi((2N+\theta)r-\Gamma_{k+1}, l-1)\Big) \nn\\
&\quad  \times  \Pi((2N+\theta)\tau_N,m-1) \times \Pi((2N+\theta)\tau_N,n-1).
\end{align*}
 By Lemma \ref{9l4.0}, the summation of $n$ gives
 \begin{align*} 
&\sum_{n=0}^\infty (\frac{2N+2\theta}{2N+\theta})^{n}  \Pi((2N+\theta)\tau_N,n-1) \frac{1}{(1+m+n)^{d/2}}\leq \frac{C }{(1+m)^{d/2-1}}.
\end{align*}
Use Lemma \ref{9l8.7} to get the sum for $m$ is at most
 \begin{align*} 
&\sum_{m=0}^\infty (\frac{2N+2\theta}{2N+\theta})^{m}  \Pi((2N+\theta)\tau_N,m-1) \frac{1}{(1+m)^{d/2-1}}\leq CI(N).
\end{align*}
Now we are left with
 \begin{align*}  
&\E( K^{(ii)}(r))\leq  \frac{C}{\psi(N)}  I(N)  \sum_{k=0}^\infty   \sum_{j=0}^\infty \sum_{l=0}^\infty  (\frac{2N+2\theta}{2N+\theta})^{k+j+l} \frac{1}{(1+j+l)^{d/2}} \nn\\
& \E\Big(1(\Gamma_{k+j+1}<(2N+\theta)r)  1(\Gamma_{k+1}>(2N+\theta)(r-\tau_N)) \cdot\pi((2N+\theta)r-\Gamma_{k+1}, l-1)\Big).
\end{align*}
Similar to \eqref{9ec1.02}, we may bound the expectation above by
 \begin{align*} 
&\E\Big(1((2N+\theta)(r-\tau_N)<\Gamma_{k+1}<(2N+\theta)r) \cdot\pi((2N+\theta)r-\Gamma_{k+1}, l-1)  \Big) \cdot   \Pi((2N+\theta)\tau_N, j).
\end{align*}
So the sum for $j$ gives
 \begin{align*} 
&\sum_{j=0}^\infty (\frac{2N+2\theta}{2N+\theta})^{j}  \Pi((2N+\theta)\tau_N, j) \frac{1}{(1+j+l)^{d/2}}\leq \frac{C }{(1+l)^{d/2-1}}\leq C,
\end{align*}
where the first inequality is by Lemma \ref{9l4.0}. Use \eqref{9ec1.19} to see that the sum of $l$ is at most
 \begin{align*} 
&\sum_{l=0}^\infty (\frac{2N+2\theta}{2N+\theta})^{l} \pi((2N+\theta)r-\Gamma_{k+1}, l-1) \leq C.
\end{align*}
We are left with
 \begin{align*}  
&\E( K^{(ii)}(r))\leq  \frac{C}{\psi(N)}  I(N)  \sum_{k=0}^\infty    (\frac{2N+2\theta}{2N+\theta})^{k} \P\Big((2N+\theta)(r-\tau_N)<\Gamma_{k+1}<(2N+\theta)r\Big).
\end{align*}
By \eqref{9ec1.04}, the above is at most
 \begin{align*}  
&\E( K^{(ii)}(r))\leq  \frac{C}{\psi(N)}  I(N)   \times CN\tau_N \leq C\tau_N,
\end{align*}
as required.
\end{proof}

\begin{proof}[Proof of Lemma \ref{9l5.5}]
Apply Lemma \ref{9l3.2} to see
\begin{align*}
M_t:=K_{t}^{n,3}(\phi)-\int_0^t G_r(\phi)ds\text{ is a martingale},
\end{align*}
whose predictable quadratic variation is given by
\begin{align*}
\la M\ra_t=\frac{1}{N^2(2N+\theta)}\int_0^t  \sum_{\beta} 1(T_{\pi\beta}<r\leq T_\beta, \zeta_\beta^1>r-\tau_N)  \phi(B^\beta)^2 F_\beta(r)^2 dr.
\end{align*}
Using the $L^2$ maximal inequality and   recalling $F_\beta(r)$ from \eqref{9e3.8}, we get
\begin{align}\label{9e6.1}
&\E\Big(\sup_{s\leq t} M_s^2\Big) \leq C\E(\la M\ra_t)\leq C\| \phi\| _\infty^2\frac{1}{N^3} \int_0^t \E\Big( \sum_{\beta} 1(T_{\pi\beta}<r\leq T_\beta) \nn\\
&\times \sum_{\gamma}1(T_{\pi \gamma}<r, B^\beta-B^\gamma\in \cN_N) 1(T_{\gamma\wedge \beta}>r-\tau_N)\nn\\
&\times \sum_{\delta} 1(T_{\pi \delta}<r, B^\beta-B^\delta\in \cN_N)1(T_{\delta\wedge \beta}>r-\tau_N) \Big) dr,
\end{align}
where we have used $\psi_0(N)\geq 1$.
Again the indicators $1(T_{\gamma \wedge \beta}>r-\tau_N)$ and $1(T_{\delta \wedge \beta}>r-\tau_N)$ ensure   $\delta_0=\beta_0=\gamma_0$.  We may set $\delta_0=\beta_0=\gamma_0=1$ so that for each $0<r<t$, the expectation in \eqref{9e6.1} is bounded by
\begin{align}\label{9ec1.21} 
&NX_0^0(1) \sum_{\beta,\gamma,\delta\geq 1}  \E\Big(1(T_{\pi\beta}<r\leq T_\beta) 1(T_{\pi \gamma}<r, B^\beta-B^\gamma\in \cN_N)\nn\\
&   1(T_{\gamma\wedge \beta}>r-\tau_N)  1(T_{\delta\wedge \beta}>r-\tau_N)  1(T_{\pi \delta}<r)1(\sqrt{N}(B^\gamma-B^\delta)\in [-2,2]^d) \Big),
\end{align}
where we have used $\{B^\beta-B^\delta\in \cN_N\}$ and $\{B^\beta-B^\gamma\in \cN_N\}$ to get $\sqrt{N}(B^\gamma-B^\delta)\in [-2,2]^d$. One can check the events inside the expectation in \eqref{9ec1.21} is almost identical to $K_{\beta,\gamma, \delta}(r)$  in \eqref{9e4.31} except that $B^\gamma=B^\delta$ is now replaced by $\sqrt{N}(B^\gamma-B^\delta)\in [-2,2]^d$. All the calculations from Lemma \ref{9l5.0} will continue to work except that there will be an extra $\psi(N)$ as we do not have an exact equality here (see, e.g., \eqref{9ec1.10}). Lemma \ref{9l5.0} then implies \eqref{9ec1.21} is at most $NX_0^0(1) (C\tau_N \cdot \psi(N))$ and \eqref{9e6.1} becomes
\begin{align*} 
&\E\Big(\sup_{s\leq t} M_s^2\Big)  \leq C\| \phi\| _\infty^2\frac{1}{N^3} \int_0^t  NX_0^0(1) (C\tau_N\cdot \psi(N))   dr \nn\\
&\leq C\| \phi\| _\infty^2   X_0^0(1)  \frac{\psi_0(N) \tau_N}{N} \to 0 \text{ as } N\to \infty.
\end{align*}
By applying Cauchy-Schwartz, we get
\begin{align}\label{9e6.2}
&\E\Big(\sup_{s\leq t}  \Big\vert K_{s}^{n,3}(\phi)-\int_0^s G_r(\phi)dr\Big\vert \Big)  \leq \sqrt{\E\Big(\sup_{s\leq t} M_s^2\Big) }  \to 0 \text{ as } N\to \infty.
\end{align}
 It suffices to show
\begin{align*}
\lim_{N\to \infty} \int_0^t \E(\vert G_r^\tau(\phi) -G_r(\phi)\vert )dr=0.
\end{align*}
Recall $G_r^\tau (\phi)$ from \eqref{9e3.12}. If $r\leq \tau_N$, by recalling the definition of $F_\beta(r)$, we have
\begin{align*} 
 \E(G_r^\tau(\phi))&\leq \frac{1}{N} \| \phi\| _\infty \E\Big(\sum_{\beta} 1(T_{\pi\beta}<r\leq T_\beta) \frac{1}{\psi_0(N)}  \sum_{\gamma} 1(T_{\pi\gamma}<r, B^\gamma-B^\beta\in\cN_N, T_{\beta\wedge \gamma}>r-\tau_N)\Big)\nn\\
&\leq \frac{1}{\psi(N)} \| \phi\| _\infty \E\Big(\sum_{\beta,\gamma: \beta_0=\gamma_0} \text{nbr}_{\beta,\gamma}(r)\Big),
\end{align*}
where in the first inequality we have dropped $1(\zeta_\beta^1>r-\tau_N)$ and in the second inequality we have replaced $1(T_{\beta\wedge \gamma}>r-\tau_N)$ by $1(\beta_0=\gamma_0)$. Use Lemma \ref{9l4.1} (ii) to see that
\begin{align}\label{9e9.30}
& \E(G_r^\tau(\phi)) \leq \frac{1}{\psi(N)} \| \phi\| _\infty CNX_0^0(1)   I((2N+2\theta)r).
\end{align}
Hence it follows that 
\begin{align}\label{9e9.31}
&\int_0^{\tau_N} \E(G_r^\tau(\phi)) dr\leq \int_0^{\tau_N} C \| \phi\| _\infty \frac{1}{\psi_0(N)}   X_0^0(1)   I((2N+2\theta)r) dr\nn\\
&\leq \tau_N \frac{C}{\psi_0(N)} X_0^0(1)   I((2N+2\theta)\tau_N)  \leq CX_0^0(1) \tau_N   \to 0.
\end{align}
Similarly, we may obtain the same bound for $\E(G_r(\phi))$ as in \eqref{9e9.30}. We conclude that 
\begin{align}\label{9e9.25}
&\int_0^{\tau_N} \E(\vert G_r^\tau(\phi)-G_r(\phi)\vert ) dr\leq \int_0^{\tau_N} \E(G_r^\tau(\phi)) dr+\int_0^{\tau_N} \E(G_r(\phi)) dr   \to 0.
\end{align}

It remains to bound $\E(\vert G_r^\tau(\phi) -G_r(\phi)\vert )$ for $r>\tau_N$. Use $\phi\in C_b^3$ to obtain
\begin{align}\label{9e6.3}
&\E(\vert G_r^\tau(\phi)-G_r(\phi)\vert )\leq \frac{C}{N}\frac{1}{\psi_0(N)}  \E\Big(\sum_{\beta} 1(T_{\pi\beta}<r\leq T_\beta)  \vert B^\beta- B^\beta_{r-\tau_N}\vert  \nn\\
&\quad \times \sum_{\gamma}  1(T_{\pi \gamma}<r, B^\beta-B^\gamma\in \cN_N) 1(T_{\gamma\wedge \beta}>r-\tau_N)\Big)\nn\\
&=\frac{C}{N}\frac{1}{\psi_0(N)} NX_0^0(1) \E\Big(\sum_{\beta,\gamma\geq 1} \text{nbr}_{\beta,\gamma}(r) 1(T_{\gamma\wedge \beta}>r-\tau_N) \vert B^\beta- B^\beta_{r-\tau_N}\vert   
 \Big),
\end{align}
where the last equality follows by translation invariance.

\begin{lemma}\label{9l9.1}
For any $\tau_N<r<t$, we have
\begin{align*}
\E\Big(\sum_{\beta,\gamma\geq 1} \text{nbr}_{\beta,\gamma}(r) 1(T_{\gamma\wedge \beta}>r-\tau_N) \vert B^\beta- B^\beta_{r-\tau_N}\vert   
 \Big)\leq C \sqrt{ \tau_N} I(N).
 \end{align*}
\end{lemma}
Use the above in \eqref{9e6.3} to get
\begin{align*} 
&\E(\vert G_r^\tau(\phi)-G_r(\phi)\vert )\leq   C \frac{1}{\psi_0(N)} X_0^0(1)   \sqrt{ \tau_N} I(N)\leq CX_0^0(1)  \sqrt{ \tau_N},
\end{align*}
thus giving 
\begin{align}\label{9e9.26}
&\int_{\tau_N}^t \E(\vert G_r^\tau(\phi)-G_r(\phi)\vert ) dr \leq  \int_{\tau_N}^t CX_0^0(1)  \sqrt{ \tau_N} dr\to 0.
\end{align}
The proof of Lemma \ref{9l5.5} is now complete in view of \eqref{9e6.2}, \eqref{9e9.25}, \eqref{9e9.26}.
 \end{proof}
It remains to prove Lemma \ref{9l9.1}. 
\begin{proof}[Proof of Lemma \ref{9l9.1}]
By letting $\alpha=\beta\wedge \gamma$, we may write the sum of $\beta,\gamma$ as   
\begin{align*} 
\sum_{k=0}^\infty \sum_{\alpha\geq 1, \vert \alpha\vert =k}  \sum_{l=0}^\infty \sum_{m=0}^\infty\sum_{\substack{\beta\geq \alpha,\\ \vert \beta\vert =k+l}}\sum_{\substack{ \gamma\geq \alpha,\\ \vert \gamma\vert =k+m}} \E\Big(\text{nbr}_{\beta,\gamma}(r) 1(T_{\alpha}>r-\tau_N) \vert B^\beta- B^\alpha_{r-\tau_N}\vert \Big).
\end{align*}
For any $\alpha, \beta, \gamma$ as in the summation, by conditioning on $\cH_{\alpha}$, we have
\begin{align} \label{9e9.15}
&\E\Big(\text{nbr}_{\beta,\gamma}(r)1(T_{\alpha}>r-\tau_N) \vert B^\beta- B^\alpha_{r-\tau_N}\vert \Big\vert \cH_{\alpha}\Big)\nn\\
&\leq  1{\{r-\tau_N<T_{\alpha}<r, B^{\alpha}\neq \Delta\}}\cdot    (\frac{N+\theta}{2N+\theta})^{l+m}\nn\\
&\times \P(T_{\pi \beta}-T_{\alpha}<r-T_{\alpha}\leq T_\beta-T_{\alpha}\vert \cH_{\alpha})\nn\\
&\times \P(T_{\pi \gamma}-T_{\alpha}<r-T_{\alpha}\vert \cH_{\alpha})\nn\\
&\times \E\Big(1( B^\beta- B^\gamma \in \cN_N) \cdot \vert B^\beta- B^\alpha_{r-\tau_N}\vert \Big\vert \cH_{\alpha}\Big).
\end{align}
The first probability equals $\pi((2N+\theta)(r-T_\alpha), l-1)$ and the second $\Pi((2N+\theta)(r-T_\alpha), m-1)$. For the last expectation, we recall
$\hat{B}^\beta, \hat{B}^\gamma$ from \eqref{9ea1.71}. Apply \eqref{9ea2.82} to get
\begin{align*} 
 \E\Big(1( B^\beta- B^\gamma \in \cN_N) \cdot \vert B^\beta- B^\alpha_{r-\tau_N}\vert \Big\vert \cH_{\alpha}\Big)  \leq   \E\Big(1( \sqrt{N}(\hat{B}^\beta- \hat{B}^\gamma) \in [-2,2]^d) \cdot \vert B^\beta- B^\alpha_{r-\tau_N}\vert \Big\vert \cH_{\alpha}\Big).
\end{align*}
Use $\vert B^\beta- B^\alpha\vert \leq |\hat{B}^\beta|+CN^{-1/2}$ to get
\begin{align*} 
\vert B^\beta- B^\alpha_{r-\tau_N}\vert&\leq \vert B^\alpha- B^\alpha_{r-\tau_N}\vert+  \vert \hat{B}^\beta\vert +CN^{-1/2} \leq  \vert B^\alpha- B^\alpha_{r-\tau_N}\vert+  \vert \hat{B}^\gamma\vert +CN^{-1/2},
\end{align*}
where the second inequality is by $\sqrt{N}(\hat{B}^\beta- \hat{B}^\gamma) \in [-2,2]^d$. It follows that
\begin{align*} 
 &\E\Big(1( B^\beta- B^\gamma \in \cN_N) \cdot \vert B^\beta- B^\alpha_{r-\tau_N}\vert \Big\vert \cH_{\alpha}\Big)  \\
   &  \leq \E\Big(1( \sqrt{N}(\hat{B}^\beta- \hat{B}^\gamma) \in [-2,2]^d) \cdot (\vert \hat{B}^\gamma\vert +CN^{-1/2}) \Big)  +\vert B^\alpha- B^\alpha_{r-\tau_N}\vert \cdot \P(\sqrt{N}(\hat{B}^\beta- \hat{B}^\gamma) \in [-2,2]^d).
\end{align*}
 The second term above is at most $\vert B^\alpha- B^\alpha_{r-\tau_N}\vert \cdot {C}/{(1+l+m)^{d/2}}$ by Lemma \ref{9l4.2}. For the first term above, we condition on $\hat{B}^\gamma$ and use Lemma \ref{9l4.2} again to get
\begin{align*} 
&\E\Big(1( \sqrt{N}(\hat{B}^\beta- \hat{B}^\gamma) \in [-2,2]^d) \cdot (\vert \hat{B}^\gamma\vert +CN^{-1/2}) \Big) \nn\\
&\leq \frac{C}{(1+l)^{d/2}} \E(\vert N^{-1/2}V_{m-1}^N\vert +CN^{-1/2})\leq \frac{C}{(1+l)^{d/2}}  N^{-1/2} \cdot (m+1)^{1/2}.
\end{align*}

We conclude from the above that \eqref{9e6.11} is at most
\begin{align}\label{9e9.21}
&\E\Big(\sum_{\alpha\geq 1} 1{\{r-\tau_N<T_{\alpha}<r, B^{\alpha}\neq \Delta\}} \cdot  C  \sum_{l=0}^\infty \sum_{m=0}^\infty    (\frac{2N+2\theta}{2N+\theta})^{l+m} \nn\\
&\times \pi((2N+\theta)(r-T_\alpha), l-1) \cdot \Pi((2N+\theta)(r-T_\alpha), m-1) \nn\\
&\times \Big[\frac{N^{-1/2}}{(1+l)^{d/2}}  (m+1)^{1/2}+\vert B^\alpha- B^\alpha_{r-\tau_N}\vert \cdot \frac{1}{(1+l+m)^{d/2}}\Big] \Big):=I_1+I_2,
\end{align}
where $I_1$ (resp. $I_2$) denotes the summation of the first term (resp. second term) in the last square bracket. It suffices to prove $I_i\leq C \sqrt{ \tau_N} I(N)$ for $i=1,2$.

We first consider $I_1$. By Lemma \ref{9l4.3}(i), the sum of $l$ gives
 \begin{align} \label{9e9.16}
&\sum_{l=0}^\infty (\frac{2N+2\theta}{2N+\theta})^{l} \pi((2N+\theta)(r-T_\alpha), l-1) \frac{1}{(1+l)^{d/2}}\nn\\
&\leq  \frac{Ce^{\theta r}}{(1+(2N+2\theta)(r-T_\alpha))^{d/2}}   \leq  \frac{C}{(1+(2N+2\theta)(r-T_\alpha))^{2}}.
\end{align}
Use the definition of $\Pi(\lambda,m)$ to get the sum of $m$ equals
\begin{align*} 
&\sum_{m=0}^\infty (\frac{2N+2\theta}{2N+\theta})^{m} \Pi((2N+\theta)(r-T_\alpha), m-1) (1+m)^{1/2}\nn\\
&\leq 1+C\sum_{m=0}^\infty (\frac{2N+2\theta}{2N+\theta})^{m} (1+m)^{1/2} \sum_{k=m}^\infty \pi((2N+\theta)(r-T_\alpha), k) \nn\\
&\leq 1+C\sum_{k=0}^\infty  (\frac{2N+2\theta}{2N+\theta})^{k} \pi((2N+\theta)(r-T_\alpha), k) \sum_{m=0}^{k} (1+m)^{1/2}\nn\\
&\leq 1+Ce^{\theta r}  \sum_{k=0}^\infty  \pi((2N+2\theta)(r-T_\alpha), k)  \cdot k^{3/2}\leq 1+C \E(\xi^{3/2}),
\end{align*}
where $\xi$ is a Poisson r.v. with parameter $(2N+2\theta)(r-T_\alpha)$. Use Cauchy-Schwartz to get the above is at most
 \begin{align} \label{9e9.17}
 1+C  (\E(\xi^2) \cdot \E(\xi))^{1/2}&\leq  Ce^{\theta r} (1+(2N+2\theta)(r-T_\alpha))^{3/2}.
\end{align}
Combine \eqref{9e9.16} and \eqref{9e9.17} to get
 \begin{align} \label{9e9.22}
I_1\leq &C  N^{-1/2}\E\Big(\sum_{\alpha\geq 1} 1_{\{r-\tau_N<T_{\alpha}<r, B^{\alpha}\neq \Delta\}} \frac{1}{(1+(2N+2\theta)(r-T_\alpha))^{1/2}} \Big).
\end{align}
Apply Lemma \ref{9l3.2a} with $f(y)=1/(1+y)^{1/2}$ to see that
\begin{align*} 
I_1\leq C N^{-1/2} e^{\theta r}  \int_0^{(2N+2\theta)\tau_N}  \frac{1}{(1+y)^{1/2}} dy\leq C \sqrt{ \tau_N} \leq C \sqrt{ \tau_N} I(N).
\end{align*}

Next, we turn to $I_2$. The sum of $m$ gives
\begin{align*} 
&\sum_{m=0}^\infty (\frac{2N+2\theta}{2N+\theta})^{m} \Pi((2N+\theta)(r-T_\alpha), m) \frac{1}{(1+l+m)^{d/2}}\leq C  \frac{1}{(1+l)^{d/2-1}},
\end{align*}
where the inequality is by Lemma \ref{9l4.0}.
Apply Lemma \ref{9l4.3}(i) to see the sum of $l$ gives
 \begin{align*} 
&\sum_{l=0}^\infty (\frac{2N+2\theta}{2N+\theta})^{l} \pi((2N+\theta)(r-T_\alpha), l) \frac{1}{(1+l)^{d/2-1}}  \leq \frac{C }{(1+(2N+2\theta)(r-T_\alpha))^{d/2-1}},
\end{align*}
Combine the above to conclude
 \begin{align*}  
I_2&\leq C   \E\Big(\sum_{\alpha\geq 1} 1_{\{r-\tau_N<T_{\alpha}<r, B^{\alpha}\neq \Delta\}} \frac{\vert B^\alpha-B^{\alpha}_{r-\tau_N}\vert }{(1+(2N+2\theta)(r-T_\alpha))^{d/2-1}} \Big).
\end{align*}
Use Lemma \ref{9l3.2} to get
\begin{align*} 
I_2\leq &C  (2N+\theta) \int_{r-\tau_N}^r \E\Big(\sum_{\alpha\geq 1} 1_{\{T_{\pi\alpha}<s\leq T_{\alpha}, B^{\alpha}\neq \Delta\}} \frac{\vert B^\alpha_s-B^{\alpha}_{r-\tau_N}\vert }{(1+(2N+2\theta)(r-s))^{d/2-1}}\Big) ds.
\end{align*}
For each $r-\tau_N<s<r$, by conditioning on $\cF_{r-\tau_N}$, we may use the Markov property and Lemma \ref{9l2.0} to get
\begin{align} \label{9ec1.23}
&\E\Big(\sum_{\alpha\geq 1} 1{\{T_{\pi\alpha}<s\leq T_{\alpha}, B^{\alpha}\neq \Delta\}}  \vert B^\alpha_s-B^{\alpha}_{r-\tau_N}\vert \Big)\nn\\
&=  e^{\theta (s-(r-\tau_N))} \E(\vert B^N_{s-(r-\tau_N)}-B^{N}_0\vert) \leq   C\sqrt{s-(r-\tau_N)}\leq C \sqrt{\tau_N},
\end{align}
where $B^N(t)=B^N_t$ is the continuous time random walk in Lemma \ref{9l2.0}.
Although $\phi(x)=\vert x\vert $ is not a bounded function required to apply Lemma \ref{9l2.0}, we may work with $\phi_n(x)=\vert x\vert \wedge n$ and let $n\to\infty$.
It follows that 
\begin{align*} 
I_2\leq& C \sqrt{ \tau_N}   (2N+2\theta) \int_{r-\tau_N}^r \frac{1}{(1+(2N+2\theta)(r-s))^{d/2-1}} ds\nn\\
=&C \sqrt{ \tau_N}     \int_{0}^{(2N+2\theta)\tau_N} \frac{1}{(1+y)^{d/2-1}} dy \leq C \sqrt{ \tau_N} I(N),
\end{align*}
as required.
\end{proof}

\section{Convergence of the constant}\label{9s7}

This section is devoted to the proof of Lemma \ref{9l5.7}, from which we obtain the constant $b_d$ as in \eqref{9ec10.64}.
 Recall $b_d^\tau$ from \eqref{9e4.2} and $\tau_N$ from \eqref{9e3.34}. It suffices to find the limits of $\E(Z_1(\tau_N))$ and $\E(\vert \{1\}_{\tau_N}\vert )$ as $N\to\infty$. Use the definition of $\{1\}_{\tau_N}$ from \eqref{9e4.3} to get
\begin{align*}
\E(\vert \{1\}_{\tau_N}\vert )=\E\Big(\sum_{\beta \geq 1} 1(T_{\pi\beta}<\tau_N\leq T_\beta, B^\beta\neq \Delta)\Big)=e^{\theta \tau_N} \to 1 \text{ as } N\to \infty,
\end{align*}
where the second equality is by Lemma \ref{9l2.0}.
It remains to prove $\lim_{N\to \infty} \E(Z_1(\tau_N))=b_d$. Recall $Z_1(\tau_N)$ from \eqref{9e4.2} and $F_\beta(r)$ from \eqref{9e3.8} to see
\begin{align} \label{9ec5.59}
I_0:=&\psi_0(N) \E(Z_1(\tau_N))=\E\Big(\sum_{\beta\in \{1\}_{\tau_N}}  \sum_{\gamma}1(T_{\pi \gamma}<\tau_N,   B^\beta-B^\gamma \in \cN_N, T_{\gamma \wedge \beta}>0)\Big)\\
=& \E\Big( \sum_{\beta,\gamma \geq 1} 1(T_{\pi \beta}<\tau_N\leq T_\beta, T_{\pi \gamma}<\tau_N,   B^\beta-B^\gamma \in \cN_N)\Big)=  \E\Big(\sum_{\beta,\gamma \geq 1} \text{nbr}_{\beta,\gamma}(\tau_N)\Big),\nn
\end{align}
where in the second equality we have replaced the condition $T_{\gamma \wedge \beta}>0$ by $\gamma_0=\beta_0=1$. The last equality is by the definition of $\text{nbr}_{\beta,\gamma}(r)$ from \eqref{9e2.24}. 
We note that $B^\gamma-B^\beta\neq 0$, $T_{\pi \beta}<r\leq T_\beta$ and $T_{\pi \gamma}<r$ imply $\gamma\wedge \beta$ is not $\beta$, but it is possible that $\gamma\wedge \beta=\gamma$. Let $\vert \beta\wedge \gamma\vert =k$ for some $k\geq 0$. There are the following two cases:
\begin{align}\label{9e0.0}
\begin{cases}
&\text{ (i)  } \gamma=\beta|k \text{ and }  \vert \beta\vert =k+1+l \text{ for some $l\geq 0$};\\
& \text{(ii)   } \vert \beta\vert =k+1+l \text{ and } \vert \gamma\vert =k+1+m \text{ for some $l\geq 0$ and $m\geq 0$}.
\end{cases}
\end{align}
 Set $\alpha=\beta\wedge \gamma$ so that $|\alpha|=k$. For case (i), we get $\gamma=\alpha$. For case (ii), we have either $\beta_{k+1}=1, \gamma_{k+1}=0$ or $\gamma_{k+1}=1, \beta_{k+1}=0$ which are symmetric, hence we may proceed with $\beta_{k+1}=1, \gamma_{k+1}=0$. Now conclude from the above and \eqref{9e0.0} that
\begin{align}\label{9e2.61}
I_0&=\E\Big(\sum_{k=0}^\infty \sum_{\alpha\geq 1, \vert \alpha\vert =k} \sum_{l=0}^\infty \sum_{\substack{\beta>\alpha,\\ \vert \beta\vert =k+1+l}}  1_{\{\gamma=\alpha\}} \text{nbr}_{\beta,\gamma}(\tau_N)\Big)\nn\\
&+2\E\Big(\sum_{k=0}^\infty \sum_{\substack{\alpha\geq 1,\\ \vert \alpha\vert =k}} \sum_{l=0}^\infty \sum_{m=0}^\infty\sum_{\substack{\beta>\alpha,\\ \vert \beta\vert =k+1+l,\\ \beta_{k+1}=1}}\sum_{\substack{ \gamma> \alpha,\\ \vert \gamma\vert =k+1+m,\\ \gamma_{k+1}=0 }} \text{nbr}_{\beta,\gamma}(\tau_N)\Big):=D_1^N+D_2^N.
\end{align}
{\it Step 1.} For case (i) when $\gamma=\alpha<\beta$ with $|\alpha|=k$ and $|\beta|=k+1+l$, we get
\begin{align}\label{e3.22}
&\E(\text{nbr}_{\beta,\gamma}(\tau_N))=\E(\text{nbr}_{\beta,\alpha}(\tau_N))=\P(T_{\pi \beta}<\tau_N\leq T_\beta, B^\beta-B^\alpha \in \cN_N)\nn\\
&=\pi((2N+\theta) \tau_N), k+l+1) \cdot (\frac{N+\theta}{2N+\theta})^{k+l+1} \cdot\P(V_{l+1}^N\in  N^{1/2}\cN_N),
\end{align}
where the first term is by \eqref{9ea6.84}, and the second term gives the probability of $\{B^\beta, B^\alpha \neq \Delta\}$. The last term follows from \eqref{ea3.22}. Hence
\begin{align*}
D_1^N&= \sum_{k=0}^\infty   \sum_{l=0}^\infty  2^k \cdot 2^{l+1}  \cdot \pi((2N+\theta) \tau_N),  k+l+1) \cdot (\frac{N+\theta}{2N+\theta})^{k+l+1} \cdot\P(V_{l+1}^N\in  N^{1/2}\cN_N)\nn\\
&= e^{\theta \tau_N}  \sum_{l=0}^\infty    \P(\Gamma_{l+1}<(2N+2\theta) \tau_N)  \cdot\P(V_{l+1}^N\in  N^{1/2}\cN_N).
 \end{align*}
For $l>\sqrt{N\tau_N}$, we use Lemma \ref{9l4.2} to bound the last probability so that
\begin{align*} 
 &\sum_{l>\sqrt{N\tau_N}}    \P(\Gamma_{l+1}<(2N+2\theta) \tau_N))  \cdot\P(V_{l+1}^N\in  N^{1/2}\cN_N)\\
 &\leq  \sum_{l>\sqrt{N\tau_N}}    \frac{C}{(1+l)^{d/2}}\leq C(N\tau_N)^{1-d/2} \to 0.
 \end{align*}
 When $l<\sqrt{N\tau_N}$, by Markov's inequality we get
 \begin{align*} 
 \P(\Gamma_{l+1}>(2N+2\theta) \tau_N)\leq \frac{l+1}{(2N+2\theta) \tau_N}\leq \frac{1}{\sqrt{N\tau_N}} \to 0.
 \end{align*}
 Therefore we may replace $ \P(\Gamma_{l+1}<(2N+2\theta) \tau_N) $ by $1$.
 It follows that
 \begin{align}\label{e2.61}
\lim_{N\to \infty} D_1^N &=  \lim_{N\to \infty} e^{\theta \tau_N}  \sum_{l<\sqrt{N\tau_N}}  \P(V_{l+1}^N\in  N^{1/2}\cN_N)\nn\\
&=  \lim_{N\to \infty}  \sum_{l<\sqrt{N\tau_N}}  \P(V_{l+1}^N\in  [-1,1]^d-\{0\})=\sum_{l=0}^\infty  \P(V_{l+1}\in  [-1,1]^d-\{0\}).
 \end{align} 
 
 {\it Step 2.} Turning to $D_2^N$, we  let $\alpha, \beta, \gamma$ be as in the summation and then condition on $\cH_{\alpha}$ to get
\begin{align*}
\E(\text{nbr}_{\beta,\gamma}(r)\vert \cH_{\alpha})=& 1{\{T_{\alpha}<r, B^{\alpha}\neq \Delta\}}\cdot    (\frac{N+\theta}{2N+\theta})^{1+l+m}\nn\\
&\times \P(T_{\pi \beta}-T_{\alpha}<r-T_{\alpha}\leq T_\beta-T_{\alpha}\vert \cH_{\alpha})\nn\\
&\times \P(T_{\pi \gamma}-T_{\alpha}<r-T_{\alpha}\vert \cH_{\alpha})\nn\\
&\times \P(W^N+V_{l+m}^N\in  N^{1/2}\cN_N),
\end{align*}
where the second term is the probability that no death events occur on the family line after $\beta,\gamma$ split. In the last term, we set $W^N$ to  
be uniform on $\cN_N$, independent of $V_{n}^N$, to represent the step taken by $\alpha\vee e_\alpha$ that separates $\beta\vert (k+1)$ and $\gamma\vert (k+1)$ from their parent $\alpha$. After they split in generation $k+1$, we let $V_{l+m}^N$ be the steps taken independently on their family lines. Apply  \eqref{9ea2.00} and \eqref{9ea6.84} in \eqref{9e2.62} to see
\begin{align}\label{9e2.63}
&\E(\text{nbr}_{\beta,\gamma}(r)\vert \cH_{\alpha})= 1{\{T_{\alpha}<r, B^{\alpha}\neq \Delta\}}\cdot  (\frac{N+\theta}{2N+\theta})^{1+l+m} \pi((2N+\theta) (r-T_\alpha), l) \nn\\
&\times   \Pi((2N+\theta) (r-T_\alpha), m) \cdot \P(W^N+V_{l+m}^N\in N^{1/2}\cN_N).
\end{align}
Use \eqref{e3.22} and \eqref{9e2.63} in $D_2^N$ in \eqref{9e2.61} to get
\begin{align}\label{9e2.64}
D_2^N= 2\E\Big(&\sum_{k=0}^\infty \sum_{\alpha\geq 1, \vert \alpha\vert =k} 1{\{T_{\alpha}<r, B^{\alpha}\neq \Delta\}}\nn \\
&\sum_{l=0}^\infty \sum_{m=0}^\infty 2^l \cdot 2^{m} \cdot  (\frac{N+\theta}{2N+\theta})^{1+l+m}  \pi((2N+\theta) (r-T_\alpha), l)\nn \\
&\times    \Pi((2N+\theta) (r-T_\alpha), m) \cdot \P(W^N+V_{l+m}^N\in [-1,1]^d-\{0\})\Big).
\end{align}
 For notation ease, we let 
 \begin{align*} 
h_N(l,m):=\P(W^N+V_{l+m}^N\in [-1,1]^d-\{0\}).
\end{align*}
Use Fubini's theorem to get 
\begin{align*} 
D_2^N=&\sum_{l=0}^\infty \sum_{m=0}^\infty   (\frac{2N+2\theta}{2N+\theta})^{1+l+m}  h_N(l,m) \cdot \E\Big( \sum_{\alpha\geq 1} 1{\{T_{\alpha}<\tau_N, B^{\alpha}\neq \Delta\}} \\
&\times \pi((2N+\theta) (\tau_N-T_\alpha), l)  \cdot \Pi((2N+\theta) (\tau_N-T_\alpha), m) \Big).
\end{align*}
By Lemma \ref{9l3.2}, the expectation above is equal to 
\begin{align*} 
&(2N+\theta)\int_0^{\tau_N} \E\Big( \sum_{\alpha\geq 1} 1{\{T_{\pi \alpha}<r\leq T_{\alpha}, B^{\alpha}\neq \Delta\}} \nn\\
&\quad \quad \times  \pi((2N+\theta) (\tau_N-r), l)  \cdot \Pi((2N+\theta) (\tau_N-r), m) \Big) dr\nn\\
&=(2N+\theta)\int_0^{\tau_N} e^{\theta r}   \pi((2N+\theta) (\tau_N-r), l)  \cdot \Pi((2N+\theta) (\tau_N-r), m)  dr,
\end{align*}
thus giving
\begin{align*}
D_2^N&=\sum_{l=0}^\infty \sum_{m=0}^\infty   (\frac{2N+2\theta}{2N+\theta})^{1+l+m}  h_N(l,m) \cdot (2N+\theta)\nn\\
&\quad \int_0^{\tau_N} e^{\theta (\tau_N-r)}   \pi((2N+\theta)r, l)  \cdot \sum_{j=m}^\infty \pi((2N+\theta)r, j)  dr\nn\\
&=\sum_{l=0}^\infty \sum_{m=0}^\infty   h_N(l,m)\cdot (2N+2\theta) \int_0^{\tau_N} e^{\theta (\tau_N+r)}\nn\\ 
& \quad  \times \pi((2N+2\theta)r, l)  \cdot \sum_{j=m}^\infty (1+\eps_N)^{m-j}  \pi((2N+2\theta)r, j)  dr,
\end{align*}
where the last equality also uses $\eps_N=\frac{\theta}{2N+\theta}$.
Notice that $1\leq e^{\theta (\tau_N+r)} \leq e^{2\theta \tau_N}$ for any $0\leq r\leq \tau_N$. So we may replace $e^{\theta (\tau_N+r)}$ by $1$. Use a change of variable to simplify the above integral and get 
\begin{align} \label{9e9.35}
D_2^N& \sim\sum_{l=0}^\infty \sum_{m=0}^\infty    h_N(l,m)  \int_0^{(2N+2\theta)\tau_N}     \pi(r, l)    \sum_{j=m}^\infty  (1+\eps_N)^{m-j}  \pi(r, j)  dr\nn\\
& =\sum_{l=0}^\infty \sum_{j=0}^\infty  \int_0^{(2N+2\theta)\tau_N}     \pi(r, l)  \pi(r, j)  dr \sum_{m=0}^{j} (1+\eps_N)^{m-j}   h_N(l,m).
\end{align}
\begin{lemma}
For each $l\geq 0$ and $m\geq 0$, we have
\begin{align}\label{e1.45}
\lim_{N\to\infty} h_N(l,m)=h(l,m):=\P(W_0+V_{l+m}\in [-1,1]^d).
\end{align}
\end{lemma}
\begin{proof}
The proof is immediate by the weak convergence of $W^N$ to $W_0$ and $V_{n}^N$ to $V_n$. We also note that since $W_0$ is uniform on $[-1,1]^d$, we get $\P(W_0+V_{l+m}=0)=0$, so the limit is as in \eqref{e1.45}.
\end{proof}

When $d\geq 5$, for each $l,j \geq 0$, we may bound the summand with respect to $l,j$ by
 \begin{align*}
 \int_0^{\infty}     \pi(r, l)  \pi(r, j)  dr  \sum_{m=0}^{\infty}   \frac{C}{(l+m+1)^{d/2}}\leq  \int_0^{\infty}     \pi(r, l)  \pi(r, j)  dr  \frac{C}{(l+1)^{d/2-1}}.
\end{align*}
Notice that
 \begin{align*}
&\sum_{l=0}^\infty \sum_{j=0}^\infty \int_0^{\infty}     \pi(r, l)  \pi(r, j)  dr  \frac{C}{(l+1)^{d/2-1}}\nn\\
=& \int_0^{\infty}     \sum_{l=0}^\infty \pi(r, l)   \frac{C}{(l+1)^{d/2-1}} dr\leq  \int_0^{\infty}      \frac{C}{(1+r)^{d/2-1}} dr<\infty.
\end{align*}
By using Dominated Convergence Theorem, we may take the limit inside \eqref{9e9.35} to get 
 \begin{align*}
\lim_{N\to\infty} D_2^N&= \sum_{l=0}^\infty \sum_{j=0}^\infty  \int_0^{\infty}     \pi(r, l)  \pi(r, j)  dr  \sum_{m=0}^{j}    h(l,m)\nn\\
&=\sum_{l=0}^\infty \sum_{j=0}^\infty \frac{1}{2^{l+j+1}} \frac{(l+j)!}{l!j!} \sum_{m=0}^j   h(l,m).
\end{align*}
Since $\psi_0(N) \to 2^d$ as $N\to \infty$, we get (recall $I_0$ from \eqref{9ec5.59})
\begin{align*}
&\lim_{N\to \infty} \E(Z_1(\tau_N))=2^{-d}\lim_{N\to \infty} I_0=2^{-d}\lim_{N\to \infty} (D_1^N+D_2^N)\\
&=2^{-d} \sum_{l=0}^\infty  \P(V_{l+1}\in  [-1,1]^d-\{0\})+2^{-d} \sum_{l=0}^\infty \sum_{j=0}^\infty  \frac{1}{2^{l+j+1}} \frac{(l+j)!}{l!j!} \sum_{m=0}^j  h(l,m)=b_d,
\end{align*}
as required.

Turning to $d=4$, since $\psi_0(N)\sim 2^4 \log N$, we get that 
\begin{align*}  
J_0:=&\lim_{N\to \infty} \E(Z_1(\tau_N))=\lim_{N\to \infty} \frac{1}{2^4\log N} (D_1^N+D_2^N)\nn\\
 =&\lim_{N\to \infty} \frac{1}{2^4\log N} \int_0^{(2N+2\theta)\tau_N}    \sum_{l=0}^\infty   \pi(r, l) \sum_{j=0}^\infty \pi(r, j)   \sum_{m=0}^{j} (1+\eps_N)^{m-j}   h_N(l,m) dr,
\end{align*}
where the last equality is by \eqref{e2.61} and \eqref{9e9.35}. We claim that we may get rid of $(1+\eps_N)^{m-j}$ and conclude
\begin{align}  \label{9ec2.49}
J_0=\lim_{N\to \infty} \frac{1}{2^4\log N} \int_0^{(2N+2\theta)\tau_N}   \sum_{l=0}^\infty   \pi(r, l) \sum_{j=0}^{\infty} \pi(r, j)   \sum_{m=0}^{j}    h_N(l,m) dr.
\end{align}

To see this, we note if $j\leq AN\tau_N$, then $1\geq (1+\eps_N)^{m-j}\geq e^{-A\theta \tau_N} \geq 1-\eta$ for any $\eta>0$, allowing us to remove $(1+\eps_N)^{m-j}$ for the sum of $j\leq AN\tau_N$. It suffices to show the contribution from the sum of $j>AN\tau_N$ converges to $0$.
 
 Using Lemma \ref{9l4.2} to bound $h_N(l,m)$ by $C/(1+l+m)^2$, we get the sum of $m$ gives at most $C/(1+l)$. So the contribution from the sum of $j>AN\tau_N$ is at most
\begin{align}  \label{9ec2.03}
&\frac{1}{2^4\log N} \int_0^{(2N+2\theta)\tau_N}  \sum_{j>AN\tau_N} \pi(r, j)    \sum_{l=0}^\infty \pi(r, l)      \frac{C}{1+l} dr  \nn\\
&\leq \frac{C}{\log N}  \sum_{j>AN\tau_N}  \P({\Gamma_j}<(2N+2\theta)\tau_N) \int_0^{(2N+2\theta)\tau_N}   \frac{1}{1+r}  dr,
\end{align}
where in the last inequality we apply Lemma \ref{9l4.3} (i) to bound the sum for $l$ and bound $\pi(r,j)$ by $\Pi(r,j)=\P({\Gamma_j}<r)\leq  \P({\Gamma_j}<(2N+2\theta)\tau_N)$. Use Lemma \ref{9l4.3} (ii) to see the sum of $j$ is at most $C2^{-N\tau_N}$. The integral of $r$ can be bounded by $C\log N$. We  conclude \eqref{9ec2.03}  is at most $C2^{-N\tau_N}\to 0$, thus giving \eqref{9ec2.49}.\\

Next, the integral for $0<r<4\log N$ is at most
\begin{align*} 
&\frac{1}{2^4\log N}   \int_0^{4\log N}     \sum_{l=0}^\infty   \pi(r, l) \sum_{j=0}^{\infty} \pi(r, j)   \frac{C}{1+l}  dr \nn\\
&\leq\frac{1}{2^4\log N}    \int_0^{4\log N}     \frac{C}{1+r} dr\leq \frac{1}{\log N} C\log(\log N) \to 0,
\end{align*}
It follows that
\begin{align}  \label{9ec2.07}
J_0=\lim_{N\to \infty} \frac{1}{2^4\log N} \int_{4\log N}^{(2N+2\theta)\tau_N}   \sum_{l=0}^\infty   \pi(r, l) \sum_{j=0}^{\infty} \pi(r, j)   \sum_{m=0}^{j}    h_N(l,m) dr.
\end{align}
We claim that the above sum for $l\leq \log N$ and $j\leq \log N$ can be ignored. To see this, we note that the contribution  from $l\leq \log N$ is at most
\begin{align}  \label{9ec2.01}
&\frac{1}{2^4\log N} \int_{4\log N}^{(2N+2\theta)\tau_N}    \sum_{l=0}^{\log N}   \pi(r, l) \sum_{j=0}^\infty \pi(r, j)    \frac{C}{1+l} dr\nn\\
&= \frac{1}{2^4\log N}   \sum_{l=0}^{\log N}  \frac{C}{1+l}  \int_0^{\infty}  \pi(r, l)  dr\leq \frac{1}{2^4\log N} C\log (\log N) \to 0.
\end{align}
Let$\xi$ be a Poisson random variable with parameter $r>4\log N$. By Chebyshev's inequality,
\begin{align*}  
\sum_{j=0}^{ \log N} \pi(r, j)=\P(\xi\leq \log N)& \leq \P(\vert \xi-r\vert \geq r- \log N) \leq \frac{r}{(r-\log  N)^2}\leq \frac{C}{\log N}.
\end{align*}
Hence the sum for $j\leq   \log N$ is at most 
\begin{align}   \label{9ec2.10}
&\frac{1}{2^4\log N} \int_{4\log N}^{(2N+2\theta)\tau_N}  \sum_{l=0}^\infty   \pi(r, l) \sum_{j=0}^{ \log N} \pi(r, j)       \frac{C}{1+l} dr\nn\\
&\leq \frac{1}{2^4\log N} \frac{C}{\log N} \int_{4\log N}^{(2N+2\theta)\tau_N}   \frac{C}{1+r}  dr \to 0. 
\end{align}
Use the above to conclude  
\begin{align}  \label{9ec2.08}
J_0=\lim_{N\to \infty} \frac{1}{2^4\log N} \int_{4\log N}^{(2N+2\theta)\tau_N}  \sum_{l=\log N}^\infty   \pi(r, l) \sum_{j= \log N}^\infty \pi(r, j)   \sum_{m=0}^{j}    h_N(l,m) dr.
\end{align}

Now we are ready to replace $h_N(l,m)$ by its asymptotics when $l\geq \log N$. Define
\begin{align*}
g(l,m)= 2^d (2\pi/6)^{-d/2} (l+m)^{-d/2} \text{ for } l\geq \log N, m\geq 0.
\end{align*}
 
\begin{lemma}
(i) If $N\to\infty$ and $x_n/n^{1/2}\to x$ as $n\to \infty$, then for any Borel set with $\vert \partial B\vert =0$ and $\vert B\vert <\infty$, we have
\begin{align*}
n^{d/2}\P(x_n+V_{n}^N\in B) \to \vert B\vert  \cdot  (2\pi\sigma^2)^{-d/2} e^{-{\vert x\vert ^2}/{2\sigma^2}},
\end{align*}
where $\sigma^2=1/6$ is the limit of the variance of one component of $V_1^N$ as $N\to\infty$.\\
\no (ii) For any $\eps>0$ small, if $N$ is large, then 
\begin{align}  \label{9ec1.85}
h_N(l,m)/g(l,m)\in [1-\eps,1+\eps], \quad  \forall l\geq \log N, m\geq 0.
\end{align}
\end{lemma}
\begin{proof}
By Lemma 4.6 of \cite{BDS89} and its proof therein, we have (i) holds.  For the proof of (ii), we note (i) ensures that \eqref{9ec1.85} at least holds for finitely many $l,m$. Assume to the contrary that \eqref{9ec1.85} fails for a sequence $l_n,m_n \to \infty$. 
Then this would contradict the results in (i), thus giving the proof.
\end{proof}
By the above lemma with $d=4$, we may replace $h_N(l,m)$ in \eqref{9ec2.08} by $g(l,m)=2^4 \cdot  (\pi/3)^{-2} (l+m)^{-2}$ to get
\begin{align}  \label{9ec3.57}
J_0=\lim_{N\to \infty} \frac{9\pi^{-2}}{\log N} \int_{4\log N}^{(2N+2\theta)\tau_N}  \sum_{l=\log N}^\infty   \pi(r, l) \sum_{j=\log N}^\infty \pi(r, j)   \sum_{m=0}^{j}    \frac{1}{(l+m)^2} dr.
\end{align}
Since both $l\geq \log N$ and $j\geq  \log N$ are large, a bit of arithmetic implies
\begin{align*}
\sum_{m=0}^{j}  \frac{1}{(l+m)^2} \sim \Big(\frac{1}{1+l}-\frac{1}{1+l+j}\Big).
\end{align*}
Therefore  \eqref{9ec3.57} becomes
\begin{align*}
J_0=\lim_{N\to \infty} \frac{9\pi^{-2}}{\log N} \int_{4\log N}^{(2N+2\theta)\tau_N}  \sum_{l=\log N}^\infty   \pi(r, l) \sum_{j= \log N}^\infty \pi(r, j)  \Big(\frac{1}{1+l}-\frac{1}{1+l+j}\Big)   dr.
\end{align*}
Use the same reasoning as above (see, e.g., \eqref{9ec2.01} and \eqref{9ec2.10}) to write back the terms for $l\leq \log N$ and $j\leq  \log N$ to get
\begin{align*}
J_0=\lim_{N\to \infty} \frac{9\pi^{-2}}{\log N} \int_{4\log N}^{(2N+2\theta)\tau_N}  \sum_{l=0}^\infty   \pi(r, l) \sum_{j=0}^\infty \pi(r, j)  \Big(\frac{1}{1+l}-\frac{1}{1+l+j}\Big)   dr.
\end{align*}
If $\xi_1, \xi_2$ are two independent Poisson random variables with parameter $r\geq 4\log N$, then
 \begin{align*}
J_0=\lim_{N\to \infty} \frac{9\pi^{-2}}{\log N} \int_{4\log N}^{(2N+2\theta)\tau_N} \Big[\E\Big(\frac{1}{1+\xi_1}\Big)-\E\Big(\frac{1}{1+\xi_1+\xi_2}\Big)\Big]   dr.
\end{align*}
Since $r>4\log N$ is large, one can check by Cauchy-Schwartz that
 \begin{align*}
 \E\Big(\frac{1}{1+\xi_1}\Big)\sim \frac{1}{1+r} \quad \text{ and }  \quad \E\Big(\frac{1}{1+\xi_1+\xi_2}\Big)\sim \frac{1}{1+2r},
\end{align*}
 thus giving
 \begin{align*} 
 J_0&=\lim_{N\to \infty}\frac{9\pi^{-2}}{\log N}   \int_{4\log N}^{(2N+2\theta)\tau_N} \frac{1}{1+r}dr \nn\\
 &\quad -\lim_{N\to \infty}\frac{9\pi^{-2}}{\log N}   \int_{4\log N}^{(2N+2\theta)\tau_N} \frac{1}{1+2r}   dr =\frac{9}{2\pi^2}=b_4,
 \end{align*}
 as required.

\section{Proof of Lemma \ref{9l5.6}}\label{9s7.2}

We turn to the proof of Lemma \ref{9l5.6} in this section. It suffices to show that
\begin{align*} 
\lim_{N\to\infty} \int_{0}^t \E\Big(\vert G_r^\tau(\phi)-b_d^\tau X_r^{1,\tau}(\phi)\vert \Big)dr=0.
\end{align*}
We already know from \eqref{9e9.31} that
\begin{align*} 
&\lim_{N\to \infty} \int_0^{\tau_N} \E(G_r^\tau(\phi))dr = 0.
\end{align*}
Recall $X_r^{1,\tau}(\phi)$ from \eqref{9e3.13} to see that 
\begin{align*}
\E( X_r^{1,\tau}(\phi))&\leq \frac{1}{N} \| \phi\| _\infty \E\Big( \sum_{\beta} 1(T_{\pi\beta}<r\leq T_\beta, B^\beta \neq \Delta)\Big)\nn\\
&= \| \phi\| _\infty \E(X_r^0(1))=e^{\theta r} \| \phi\| _\infty X_0^0(1).
\end{align*}
Lemma \ref{9l5.7} gives $b_d^\tau \leq 1+b_d$ for $N$ large.  It follows that
\begin{align*} 
&\int_0^{\tau_N} \E\Big(b_d^\tau X_r^{1,\tau}(\phi) \Big)dr\leq (1+b_d) \| \phi\| _\infty   X_0^0(1) \int_0^{\tau_N} e^{\theta r}dr  \to 0.
\end{align*}
It remains to show
\begin{align*} 
\lim_{N\to\infty} \int_{\tau_N}^t \E\Big(\vert G_r^\tau(\phi)-b_d^\tau X_r^{1,\tau}(\phi)\vert \Big)dr=0.
\end{align*}

For any $r\geq \tau_N$, recall from \eqref{9e6.61} that
\begin{align*}
& G_r^\tau(\phi)-b_d^\tau X_r^{1,\tau}(\phi)=\frac{1}{N}  \sum_{\alpha \in \cA(r-\tau_N)}   \phi(B^\beta_{(r-\tau_N)^+}) (Z_\alpha(r)-b_d^\tau \vert \{\alpha\}_r\vert ).
\end{align*}
By \eqref{9e3.14}, the mean of the last term is zero conditioning on $\cF_{r-\tau_N}$. Also, one can check that $\{(Z_\alpha(r),\vert \{\alpha\}_r\vert )\}_{\alpha\in \cA(r-\tau_N)}$ are mutually independent and equal in law with $(Z_1(\tau_N), \vert \{1\}_{\tau_N} \vert)$ if conditioning on $\cF_{r-\tau_N}$. It follows that
\begin{align}\label{9e7.1}
&\E\Big(\Big[G_r^\tau(\phi)-b_d^\tau X_r^{1,\tau}(\phi)\Big]^2\Big)\leq \frac{1}{N^2} \| \phi\| _\infty^2 \cdot \E( \vert \cA(r-\tau_N)\vert )\cdot \E \Big(\Big(Z_1(\tau_N)-b_d^\tau \vert \{1\}_{\tau_N}\vert \Big)^2\Big).
\end{align}
For the first expectation above, we recall that $\vert \cA(r-\tau_N)\vert $ counts the number of particles alive at time $r-\tau_N$ in $X^1$, so
\begin{align}\label{9e7.2}
\E( \vert \cA(r-\tau_N)\vert )=\E(NX_{r-\tau_N}^{1,N}(1)) \leq \E(NX_{r-\tau_N}^{0,N}(1))= N X_0^0(1) e^{\theta(r-\tau_N)}.
\end{align}
For the second expectation in \eqref{9e7.1}, we use $(a-b)^2\leq 2a^2+2b^2, \forall a,b\in \R$ to get
\begin{align}\label{9e7.3}
\E \Big(\Big(Z_1(\tau_N)-b_d^\tau \vert \{1\}_{\tau_N}\vert \Big)^2\Big)\leq  2\E(Z_1(\tau_N)^2)+2\E(\vert \{1\}_{\tau_N}\vert ^2) .
\end{align}
Lemma 7.2 of \cite{DP99} gives
\begin{align}\label{9e7.4}
\E(\vert \{1\}_{\tau_N}\vert ^2)\leq C(1+N\tau_N)\leq CN\tau_N.
\end{align}
 \begin{lemma}\label{9l7.1}
There is some constant $C>0$ so that $\E(Z_1(\tau_N)^2)\leq CN\tau_N$.
\end{lemma}
By assuming Lemma \ref{9l7.1}, we may finish the proof of  Lemma \ref{9l5.6}.

\begin{proof}[Proof of  Lemma \ref{9l5.6}]
For any $r\geq \tau_N$, we   use \eqref{9e7.1}, \eqref{9e7.2}, \eqref{9e7.3}, \eqref{9e7.4} and Lemma \ref{9l7.1} to see that
\begin{align*}
&\E\Big(\Big[G_r^\tau(\phi)-b_d^\tau X_r^{1,\tau}(\phi)\Big]^2\Big)\leq \frac{C}{N^2} \| \phi\| _\infty^2 \cdot  N X_0^0(1) e^{\theta(r-\tau_N)} \cdot N\tau_N\leq C\tau_N e^{\theta r}.
\end{align*}
By Cauchy-Schwartz, we have
\begin{align*}
\E\Big(\Big\vert G_r^\tau(\phi)-b_d^\tau X_r^{1,\tau}(\phi)\Big\vert \Big)\leq C\tau_N^{1/2} e^{\theta r/2},
\end{align*}
thus giving
\begin{align*} 
&\int_{\tau_N}^t  \E\Big(\Big\vert G_r^\tau(\phi)-b_d^\tau X_r^{1,\tau}(\phi)\Big\vert \Big)dr \leq C\tau_N^{1/2} \int_0^t e^{\theta r/2} dr \to 0.
\end{align*}
The proof is complete.
\end{proof}
It remains to prove Lemma \ref{9l7.1}.

\begin{proof}[Proof of  Lemma \ref{9l7.1}]
Recall from \eqref{9ec5.59} to see that 
\begin{align}\label{9e7.10}
\E(\psi_0(N)^2 Z_{1}(\tau_N)^2) =&\E\Big(\sum_{\beta^1, \beta^2, \gamma^1, \gamma^2 \geq 1}  \text{nbr}_{\beta^1,\gamma^1}(\tau_N)\cdot \text{nbr}_{\beta^2,\gamma^2}(\tau_N)\Big),
\end{align}
where $$\text{nbr}_{\beta,\gamma}(\tau_N)=1(T_{\pi \beta}<\tau_N\leq T_\beta, T_{\pi \gamma}<\tau_N,   B^\beta-B^\gamma \in \cN_N).$$
Since $T_{\pi \beta^1}<\tau_N\leq T_{\beta^1}$ and $T_{\pi \beta^2}<\tau_N\leq T_{\beta^2}$ both hold, we cannot have $\beta^1<\beta^2$ or $\beta^2<\beta^1$. 
 By symmetry, we may assume $T_{\beta^1\wedge \gamma^1}\leq T_{\beta^2\wedge \gamma^2}$. We first consider the sum of $\beta^1, \gamma^1$ for $ \text{nbr}_{\beta^1,\gamma^1}(\tau_N)$. Let $\alpha^1=\beta^1 \wedge \gamma^1$ with $\vert \alpha^1\vert =k$ for some $k\geq 0$. Set $\vert \beta^1\vert =k+l$ and $\vert \gamma^1\vert =k+m$ for some $l\geq 0$, $m\geq 0$. Then we have
 \begin{align} \label{9e7.21}
\E(\psi_0(N)^2 Z_{1}(\tau_N)^2)= \E\Big(\sum_{k=0}^\infty &\sum_{\substack{\alpha^1\geq 1,\\ \vert \alpha^1\vert =k}} \sum_{l=0}^\infty \sum_{m=0}^\infty\sum_{\substack{\beta^1\geq \alpha^1,\\ \vert \beta^1\vert =k+l}}\sum_{\substack{ \gamma^1\geq \alpha^1,\\ \vert \gamma^1\vert =k+m }}  \text{nbr}_{\beta^1,\gamma^1}(\tau_N)\nn\\
& \quad \times \sum_{\beta^2, \gamma^2\geq 1} \text{nbr}_{\beta^2,\gamma^2}(\tau_N)\Big).
\end{align}
To get the desired upper bound, we need to know the positions of $\beta^2,\gamma^2$ on the family tree relative to $\alpha^1, \beta^1,\gamma^1$. For any finite subset $\Lambda \subset \cI$,   define
\begin{align*}
\tau(\Lambda; \beta)=\max\{\vert \gamma\wedge \beta\vert : \gamma\in \Lambda, \gamma_0=\beta_0\}
\end{align*}
to identify the last generation when $\beta$ had an ancestor in common with some $\gamma\in \Lambda$. 
For any pair of $(\alpha^1,\beta^1, \gamma^1)$ as in \eqref{9e7.21}, we consider the three cases for the generation when $\beta^2$ branches off the family tree of $\alpha^1, \beta^1,\gamma^1$: 
\begin{align*}
&\text{(1)}\quad  \tau(\{\alpha^1,\beta^1,\gamma^1\}; \beta^2)=\vert \alpha^1 \wedge \beta^2\vert ;\nn\\
&\text{(2)} \quad \tau(\{\alpha^1,\beta^1,\gamma^1\}; \beta^2)=\vert \beta^1 \wedge \beta^2\vert > \vert \alpha^1 \wedge \beta^2\vert ;\nn\\
&\text{(3)} \quad \tau(\{\alpha^1,\beta^1,\gamma^1\}; \beta^2)=\vert \gamma^1 \wedge \beta^2\vert >\vert \alpha^1 \wedge \beta^2\vert . 
\end{align*}
Denote by $R_1$ (resp. $R_2, R_3$) for the contribution to \eqref{9e7.21} from case (1) (resp. case (2), (3)). It follows that 
 \begin{align}  \label{e1}
\E(\psi_0(N)^2 Z_{1}(\tau_N)^2)\leq C\sum_{i=1}^3 \E(R_i).
\end{align}
We claim that
 \begin{align}  \label{e2}
 \E(R_i)\leq CN\tau_N I(N)^2,\quad \forall i\in\{1,2,3\}.
\end{align}
Then since $I(N)\leq C\psi_0(N)$, we may combine \eqref{e1} and \eqref{e2} to conclude
\[
\E(Z_{1}(\tau_N)^2)\leq CN\tau_N,
\]
thus giving Lemma \ref{9l7.1}. It remains to prove \eqref{e2}.  

\subsection{Bounds for $\E(R_1)$ }
For $R_1$, we are in the case when $\beta^2$ branches off the family tree of $\alpha^1, \beta^1,\gamma^1$ from $\alpha^1$. 
 We may let $\beta^2\wedge \alpha^1=\alpha^1\vert j$ for some $0\leq j\leq k$. Since we assume $T_{\beta^1\wedge \gamma^1}\leq T_{\beta^2\wedge \gamma^2}$, the only possible case is that $\tau(\{\alpha^1,\beta^1,\gamma^1,\beta^2\}; \gamma^2)=\vert \gamma^2 \wedge \beta^2\vert $, i.e. $\gamma^2$ has the most recent ancestor with $\beta^2$.  Let $\alpha^2=\gamma^2\wedge \beta^2$ so that $\alpha^2\geq \beta^2\wedge \alpha^1=\alpha^1|j$. Set $\vert \alpha^2\vert =j+n$ for some $n\geq 0$. Let $\vert \beta^2\vert =j+n+l'$ and $\vert \gamma^2\vert =j+n+m'$ for some $l'\geq 0$, $m'\geq 0$. 
 By conditioning on $\cH_{\alpha^1}\vee \cH_{\alpha^2}$, similar to the derivation of \eqref{e2.63}, we  obtain
 \begin{align}\label{9ec7.99}
&\E(\text{nbr}_{\beta^1,\gamma^1}(\tau_N)\text{nbr}_{\beta^2,\gamma^2}(\tau_N) \vert \cH_{\alpha^1}\vee \cH_{\alpha^2})\leq C\cdot 1_{(T_{\alpha^1}<\tau_N, B^{\alpha^1}\neq \Delta)} 1_{(T_{\alpha^2}<\tau_N, B^{\alpha^2}\neq \Delta)} \nn\\
&\times (\frac{N+\theta}{2N+\theta})^{l+m+l'+m'}\frac{C}{(1+l+m)^{d/2}}  \frac{C}{(1+l'+m')^{d/2}} \nn\\
&\times \pi((2N+\theta)(\tau_N-T_{\alpha^1}), l-1)\cdot  \Pi((2N+\theta)(\tau_N-T_{\alpha^1}), m-1)   \nn\\
&\times \pi((2N+\theta)(\tau_N-T_{\alpha^2}), l'-1) \cdot \Pi((2N+\theta)(\tau_N-T_{\alpha^2}), m'-1),
\end{align}
where we have used Lemma \ref{9l4.2} to bound the probabilities for the events $\{B^{\beta^i}-B^{\gamma^i} \in \cN_N\}$, $i=1,2$.  So we get
  \begin{align*}  
\E(R_1)\leq  C\E\Big(&\sum_{k=0}^\infty  \sum_{\substack{\alpha^1\geq 1,\\ \vert \alpha^1\vert =k} } \sum_{j=0}^{k}\sum_{n=0}^\infty\sum_{\substack{ \vert \alpha^2\vert =j+n,\\ \alpha^2\geq \alpha^1\vert j} } \prod_{i=1}^2 1{(T_{\alpha^i}<\tau_N, B^{\alpha_i}\neq \Delta)}\nn\\
&\sum_{l=0}^\infty  \sum_{m=0}^\infty\sum_{l'=0}^\infty \sum_{m'=0}^\infty  (\frac{2N+2\theta}{2N+\theta})^{l+m+l'+m'}\frac{1}{(1+l+m)^{d/2}}  \frac{1}{(1+l'+m')^{d/2}} \nn\\
& \times  \pi((2N+\theta)(\tau_N-T_{\alpha^1}), l-1)\cdot  \Pi((2N+\theta)(\tau_N-T_{\alpha^1}), m-1)   \nn\\
&\times \pi((2N+\theta)(\tau_N-T_{\alpha^2}), l'-1) \cdot \Pi((2N+\theta)(\tau_N-T_{\alpha^2}), m'-1)\Big). 
\end{align*}
Apply Lemma \ref{9l4.0} to see the sum of $m$ gives at most $C /(1+l)^{d/2-1}$. Then as in \eqref{9e9.58}, we use Lemma \ref{9l4.3} (i) to get the sum of $l$ is bounded above by $C /(1+(2N+2\theta)(\tau_N-T_{\alpha^1}))^{d/2-1}$. Similarly, the sum of $l',m'$ gives at most $C /(1+(2N+2\theta)(\tau_N-T_{\alpha^2}))^{d/2-1}$. These two bounds lead to 
 \begin{align*}  
\E(R_1)\leq C  \E\Big(&\sum_{k=0}^\infty \sum_{\substack{\alpha^1\geq 1,\\ \vert \alpha^1\vert =k} } \sum_{j=0}^{k}\sum_{n=0}^\infty \sum_{\substack{ \vert \alpha^2\vert =j+n,\\ \alpha^2\geq \alpha^1\vert j} }  \nn\\
&\prod_{i=1}^2 1(T_{\alpha^i}<\tau_N, B^{\alpha_i}\neq \Delta) (1+(2N+2\theta)(\tau_N-T_{\alpha^i}))^{1-d/2}\Big).
\end{align*}
By considering $\alpha=\alpha^1\wedge \alpha^2$ with $\vert \alpha\vert =j\geq 0$, we may rewrite the above as
 \begin{align}  \label{9e9.57}
\E(R_1)\leq C\E\Big( &\sum_{j=0}^\infty \sum_{\alpha \geq 1,\vert \alpha\vert =j } \sum_{k=0}^{\infty}  \sum_{\substack{\vert \alpha^1\vert =j+k,\\ \alpha^1\geq \alpha}} \sum_{n=0}^\infty \sum_{\substack{ \vert \alpha^2\vert =j+n,\\ \alpha^2\geq \alpha}}  \nn\\
&\prod_{i=1}^2 1(T_{\alpha^i}<\tau_N, B^{\alpha_i}\neq \Delta) (1+(2N+2\theta)(\tau_N-T_{\alpha^i}))^{1-d/2}\Big).
\end{align}
Conditioning on $\cH_{\alpha}$, we may use the symmetry between $\alpha^1$ and $\alpha^2$ to see that the above becomes
 \begin{align}  \label{9e6.71}
\E(R_1)\leq C\E\Big(& \sum_{\alpha \geq 1} 1(T_{\alpha}<\tau_N, B^{\alpha}\neq \Delta)   \\
&  \times \Big[\E\Big(   \sum_{\alpha^1\geq \alpha}  1(T_{\alpha^1}<\tau_N, B^{\alpha^1}\neq \Delta) (1+(2N+2\theta)(\tau_N-T_{\alpha^1}))^{1-d/2}\Big\vert \cH_\alpha\Big)\Big]^2\Big).\nn
\end{align}
On the event $\{T_{\alpha}<\tau_N, B^{\alpha}\neq \Delta\}$, we have
 \begin{align*}  
& \E\Big(   \sum_{\alpha^1\geq \alpha}  1(T_{\alpha^1}<\tau_N, B^{\alpha^1}\neq \Delta) (1+(2N+2\theta)(\tau_N-T_{\alpha^1}))^{1-d/2}\Big\vert \cH_\alpha\Big)\nn\\
\leq &1+2  \E\Big(\sum_{\alpha^1\geq (\alpha \vee 0)}   1(T_{\alpha^1}<\tau_N, B^{\alpha^1}\neq \Delta) (1+(2N+2\theta)(\tau_N- {T}_{ \alpha^1}))^{1-d/2}\Big\vert \cH_\alpha\Big)\nn\\
\leq &1+ 2\E\Big(\sum_{\delta\geq 1}   1(T_{\delta}<\tau_N- {T}_{ \alpha}, B^{\delta}\neq \Delta) (1+(2N+2\theta)(\tau_N- {T}_{ \alpha}-T_{\delta}))^{1-d/2}\Big)\nn\\
\leq&1+ 2 e^{\theta (\tau_N- {T}_{ \alpha})} \int_0^{(2N+2\theta)(\tau_N- {T}_{ \alpha})} \frac{1}{(1+y)^{d/2-1}} dy\leq C I(N),
\end{align*}
where the first term $1$ on the second line is to bound the case when $\alpha^1=\alpha$ while the factor $2$ is due to the symmetry between the two offspring, $\alpha\vee 0$ and $\alpha\vee 1$, of $\alpha$. The third line follows by replacing $\alpha \vee 0$ with $1$ and then translating with time $T_{\alpha}$. The second last inequality uses Lemma \ref{9l3.2a} with $r=\tau_N- {T}_{  \alpha}$ and $u=0$.
Hence \eqref{9e6.71} is at most
 \begin{align*}  
\E(R_1)\leq C I(N)^2\E\Big( \sum_{\alpha \geq 1 } 1(T_{\alpha}<\tau_N, B^{\alpha}\neq \Delta) \Big).
\end{align*}
By using Lemma \ref{9l3.2a} with $f(y)\equiv 1$, $r=\tau_N$ and $u=0$, we get
\begin{align}\label{9e9.66}
 \E\Big( \sum_{\alpha\geq 1 } 1(T_{\alpha}<\tau_N, B^{\alpha}\neq \Delta) \Big)\leq e^{\theta \tau_N} (2N+2\theta)\tau_N\leq CN\tau_N,
\end{align}
thus giving
\begin{align*}
\E(R_1)\leq   CN\tau_N I(N)^2,
\end{align*}
as required. 

\subsection{Bounds for $\E(R_2)$ }
Turning to $R_2$, we have $\beta^2$ branches off the family tree of $\alpha^1, \beta^1,\gamma^1$ from $\beta^1$, so we set $\beta^2\wedge \beta^1=\beta^1\vert j$ for some $k\leq j\leq k+l$. Since we assume $T_{\beta^1\wedge \gamma^1}\leq T_{\beta^2\wedge \gamma^2}$, there are three cases for the generation when $\gamma^2$ branches off the family tree of $\alpha^1, \beta^1,\gamma^1,\beta^2$: 
\begin{align}\label{9e9.91}
\text{(i)} &\quad  \tau(\{\alpha^1,\beta^1,\gamma^1,\beta^2\}; \gamma^2)=\vert  \gamma^2 \wedge \beta^2 \vert>\vert \gamma^2\wedge \beta^1\vert ;\nn\\
\text{(ii)}& \quad \tau(\{\alpha^1,\beta^1,\gamma^1,\beta^2\}; \gamma^2)=\vert \gamma^2 \wedge \beta^1\vert ;\nn\\
\text{(iii)} &\quad \tau(\{\alpha^1,\beta^1,\gamma^1,\beta^2\}; \gamma^2)=\vert \gamma^2 \wedge \gamma^1 \vert . 
\end{align}
Denote respectively by $R_2^{(i)}, R_2^{(ii)}, R_2^{(iii)}$ the contribution to $R_2$ from these three cases, for which we refer the reader to Figure \ref{fig2}.
\begin{figure}[ht]
  \begin{center}
    \includegraphics[width=1 \textwidth]{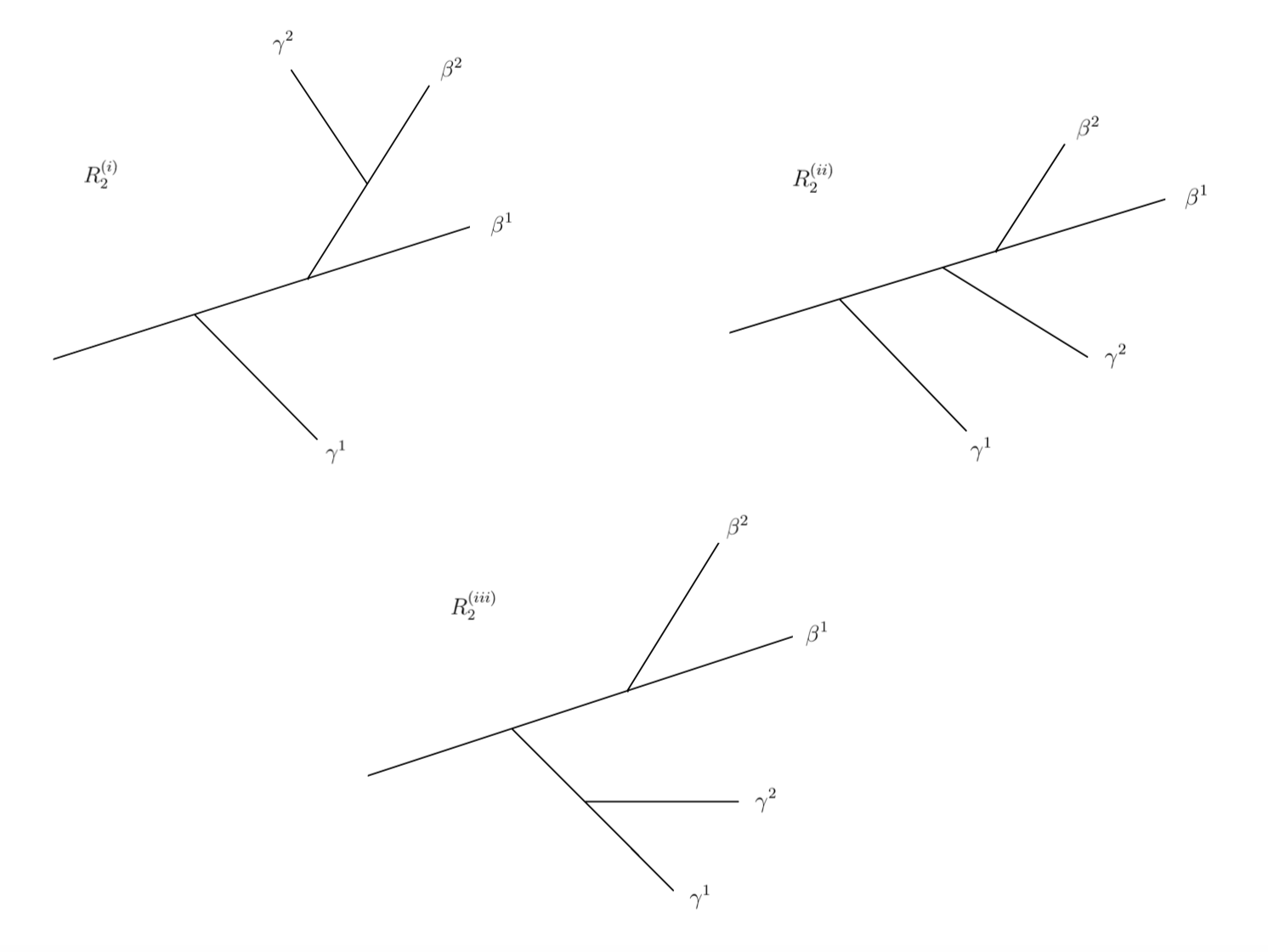}
    \caption[ ]{\label{fig2}   Three cases for $R_2$.
      }
  \end{center}
\end{figure}
It suffices to show that
\[
\E(R_2^{(i)}) \leq CN\tau_N I(N)^2, \quad \E(R_2^{(ii)}) \leq CN\tau_N I(N)^2, \quad \E(R_2^{(iii)}) \leq CN\tau_N I(N)^2.
\]

 \no {\bf Case (2.i).} Case (i) implies that $\gamma^2$ branches off $\beta^2$ only after $\beta^2$ branches off $\beta^1$, that is, $\beta^2 \wedge \gamma^2>\beta^2\wedge \beta^1=\beta^1\vert j$. Recall \eqref{9e7.21} to get
 \begin{align} \label{9ea1.00}
\E(R_2^{(i)})\leq C\E\Big(\sum_{k=0}^\infty &\sum_{\substack{\alpha^1\geq 1,\\ \vert \alpha^1\vert =k}} \sum_{l=0}^\infty \sum_{m=0}^\infty\sum_{\substack{\beta^1\geq \alpha^1,\\ \vert \beta^1\vert =k+l}}\sum_{\substack{ \gamma^1\geq \alpha^1,\\ \vert \gamma^1\vert =k+m }}  \text{nbr}_{\beta^1,\gamma^1}(\tau_N)\nn\\
& \quad \times \sum_{j=k}^{k+l}  \sum_{\beta^2\geq \beta^1\vert j}\sum_{\gamma^2\geq \beta^1\vert j} \text{nbr}_{\beta^2,\gamma^2}(\tau_N)\Big).
\end{align}

For each $k\leq j\leq k+l$, by conditioning on $\cH_{\beta^1}\vee\cH_{\gamma^1}$, we may replace $\beta^1\vert j$ by $1$ and use translation invariance to see that 
 \begin{align}   \label{9ec4.98}
&\E\Big(\sum_{\beta^2\geq \beta^1\vert j}\sum_{\gamma^2\geq \beta^1\vert j} \text{nbr}_{\beta^2,\gamma^2}(\tau_N) \Big\vert  \cH_{\beta^1}\vee\cH_{\gamma^1}\Big)
\nn\\
&=  \E\Big( \sum_{\beta^2\geq 1}\sum_{\gamma^2\geq 1}  \text{nbr}_{\beta^2,\gamma^2}(\tau_N-T_{\beta^1\vert j})\Big)\nn\\
&\leq  C  I((2N+2\theta) (\tau_N-T_{\beta^1\vert j}) )\leq CI(N),
\end{align}
where the second inequality is from \eqref{9e9.51}.
The sum of $j$ from $k$ to $k+l$ in \eqref{9ea1.00} gives an extra $1+l\leq 1+l+m$. We conclude \eqref{9ea1.00} becomes
 \begin{align}  \label{9ea1.08}
\E(R_2^{(i)})\leq CI(N) \E\Big(\sum_{k=0}^\infty &\sum_{\substack{\alpha^1\geq 1,\\ \vert \alpha^1\vert =k}} \sum_{l=0}^\infty \sum_{m=0}^\infty\sum_{\substack{\beta^1\geq \alpha^1,\\ \vert \beta^1\vert =k+l}}\sum_{\substack{ \gamma^1\geq \alpha^1,\\ \vert \gamma^1\vert =k+m }}    \text{nbr}_{\beta^1,\gamma^1}(\tau_N) \times  (1+l+m)\Big).
\end{align}
Similar to the derivation of \eqref{9e2.640}, we may bound the above by
\begin{align*} 
\E(R_2^{(i)})\leq CI(N)\E\Big(&\sum_{k=0}^\infty \sum_{\alpha^1\geq 1, \vert \alpha^1\vert =k} 1{\{T_{\alpha^1}<\tau_N, B^{\alpha^1}\neq \Delta\}}\nn \\
&\sum_{l=0}^\infty \sum_{m=0}^\infty    (\frac{2N+2\theta}{2N+\theta})^{l+m}  \pi((2N+\theta) (\tau_N-T_{\alpha^1}), l-1)\nn \\
&\times    \Pi((2N+\theta) (\tau_N-T_{\alpha^1}), m-1) \cdot \frac{1}{(1+l+m)^{d/2}} \times  (1+l+m)\Big).
\end{align*}
The sum of $m$ gives
 \begin{align*} 
& \sum_{m=0}^\infty(\frac{2N+2\theta}{2N+\theta})^{m} \Pi((2N+\theta)(\tau_N-T_{\alpha^1}), m-1)  \frac{1}{(1+l+m)^{d/2-1}} \leq  CI(N),
\end{align*}
where the inequality is by Lemma \ref{9l8.7}. The sum of $l$ gives at most $Ce^{\theta \tau_N}$.
We are left with 
 \begin{align}\label{9ea1.09}
\E(R_2^{(i)})& \leq Ce^{\theta \tau_N} I(N)^2 \E\Big( \sum_{\alpha^1\geq 1 } 1(T_{\alpha^1}<\tau_N, B^{\alpha^1}\neq \Delta) \Big)\leq CI(N)^2 N\tau_N.
\end{align}
where the last inequality is by \eqref{9e9.66}. \\

\no {\bf Case (2.ii).} Turning to Case (ii), recall $\beta^2\wedge \beta^1=\beta^1\vert j$ for some $k\leq j\leq k+l$.  Let $\gamma^2\wedge \beta^1=\beta^1\vert j'$ for some $k\leq j'\leq k+l$. Then we have
 \begin{align} \label{9ea1.01}
\E(R_2^{(ii)})\leq C\E\Big(\sum_{k=0}^\infty &\sum_{\substack{\alpha^1\geq 1,\\ \vert \alpha^1\vert =k}} \sum_{l=0}^\infty \sum_{m=0}^\infty\sum_{\substack{\beta^1\geq \alpha^1,\\ \vert \beta^1\vert =k+l}}\sum_{\substack{ \gamma^1\geq \alpha^1,\\ \vert \gamma^1\vert =k+m }}  \text{nbr}_{\beta^1,\gamma^1}(\tau_N)\nn\\
& \quad \times \sum_{j=k}^{k+l}\sum_{j'=k}^{k+l}  \sum_{\beta^2\geq \beta^1\vert j}\sum_{\gamma^2\geq\beta^1\vert j'} \text{nbr}_{\beta^2,\gamma^2}(\tau_N)\Big).
\end{align}

For each $k\leq j\leq k+l$, set $\vert \beta^2\vert =j+u$ and $\vert \gamma^2\vert =j'+n$ for some $u\geq 0$, $n\geq 0$. 
By conditioning on $\cH_{\beta^1}\vee \cH_ {\gamma^1}$, on the event $\{T_{\pi \beta^1}<\tau_N\leq T_{\beta^1}, B^{\beta^1}\neq \Delta\}$ from $\text{nbr}_{\beta^1,\gamma^1}(\tau_N)$, we have the sum of $\beta^2, \gamma^2$ equals
  \begin{align}  \label{9e9.81}
& \E\Big(\sum_{\beta^2\geq \beta^1\vert j}\sum_{\gamma^2\geq\beta^1\vert j'} \text{nbr}_{\beta^2,\gamma^2}(\tau_N) \Big\vert \cH_{\beta^1}\vee \cH_ {\gamma^1}\Big)
\nn\\
& \leq  C\sum_{u=0}^\infty \sum_{n=0}^\infty  (\frac{2N+2\theta}{2N+\theta})^{u+n}\pi((2N+\theta)(\tau_N-T_{\beta^1\vert j})^+, u-1) \nn\\
&\quad \quad \times  \Pi((2N+\theta)(\tau_N-T_{\beta^1\vert j'})^+, n-1) \frac{1}{(u+n+1)^{d/2}}\nn\\
& \leq  C \sum_{u=0}^\infty   (\frac{2N+2\theta}{2N+\theta})^{u}\pi((2N+\theta)(\tau_N-T_{\beta^1\vert j})^+, u)    \frac{1}{(u +1)^{d/2-1}}\nn\\
& \leq  \frac{C}{(1+(2N+2\theta)(\tau_N-T_{\beta^1\vert j})^+)^{d/2-1}},
\end{align}
where the second inequality follows by Lemma \ref{9l4.0} and the third inequality uses Lemma \ref{9l4.3}(i).  Apply \eqref{9e9.81} in \eqref{9ea1.01} to conclude that
 \begin{align}  \label{9e9.67}
\E(R_2^{(ii)})\leq C   &\sum_{k=0}^\infty  \sum_{\substack{\alpha^1\geq 1,\\ \vert \alpha^1\vert =k}}  \sum_{l=0}^\infty \sum_{m=0}^\infty\sum_{\substack{\beta^1\geq \alpha^1,\\ \vert \beta^1\vert =k+l}}\sum_{\substack{ \gamma^1\geq\alpha^1,\\ \vert \gamma^1\vert =k+m}}     (1+l+m)\\
&\quad   \times  \E\Big( \text{nbr}_{\beta^1,\gamma^1}(\tau_N) \sum_{j=k}^{k+l} \frac{1}{(1+(2N+\theta)(T_{\pi\beta^1}-T_{\beta^1\vert j})^+)^{d/2-1}}\Big),\nn
\end{align}
where the term $(1+l+m)$ is from the sum of $k\leq j'\leq k+l$. We also use $2N+2\theta\geq 2N+\theta$ and $T_{\pi\beta^1}<\tau_N$ to bound \eqref{9e9.81}. 
By conditioning on $\cH_{\alpha^1}$,  the expectation above is bounded by 
\begin{align*} 
&\E\Bigg\{1{\{T_{\alpha^1}<\tau_N, B^{\alpha^1}\neq \Delta\}}        (\frac{N+\theta}{2N+\theta})^{l+m}  \nn\\
&\times \E\Big[1(T_{\pi \beta^1}<\tau_N\leq T_{\beta^1}) \sum_{j=k}^{k+l} \frac{1}{(1+(2N+\theta)(T_{\pi\beta^1}-T_{\beta^1\vert j})^+)^{d/2-1}}\Big\vert \cH_{\alpha^1}\Big]\nn \\
&\times  \P\Big(T_{\pi \gamma^1}<\tau_N\Big\vert \cH_{\alpha^1}\Big)\cdot \frac{C}{(1+l+m)^{d/2}}\Bigg\},
\end{align*}
where we have used Lemma \ref{9l4.2} to bound the probability $\P(B^{\beta^1}-B^{\gamma^1}\in \cN_N\vert \cH_{\alpha^1})$. Use the Gamma random variables to see that the above is equal to
 \begin{align}  \label{9e9.82}
&\E\Bigg\{ 1{\{T_{\alpha^1}<\tau_N, B^{\alpha^1}\neq \Delta\}}        (\frac{N+\theta}{2N+\theta})^{l+m}  \nn\\
&\times \E\Big[1(\Gamma_{l-1}<(2N+\theta)(\tau_N-T_{\alpha^1})\leq \Gamma_{l}) \sum_{j=k}^{k+l} \frac{1}{(1+(\Gamma_{l-1}-\Gamma_{j-k})^+)^{d/2-1}}\Big]\nn \\
&\times  \Pi((2N+\theta)(\tau_N-T_{\alpha^1}),m-1) \cdot \frac{C}{(1+l+m)^{d/2}}\Bigg\}.
\end{align}
Apply the above in \eqref{9e9.67} to get
 \begin{align}  \label{9e9.83}
\E(R_2^{(ii)}) \leq C\E\Big(  & \sum_{\alpha^1\geq 1 } 1{\{T_{\alpha^1}<\tau_N, B^{\alpha^1}\neq \Delta\}}\sum_{l=0}^\infty  \sum_{m=0}^\infty      (\frac{2N+2\theta}{2N+\theta})^{l+m}   \nn\\
&\times\E\Big[1(\Gamma_{l-1}<(2N+\theta)(\tau_N-T_{\alpha^1})\leq \Gamma_{l}) \sum_{j=0}^{l} \frac{1}{(1+(\Gamma_{l-1}-\Gamma_{j})^+)^{d/2-1}}\Big] \nn \\
&\times \Pi((2N+\theta) (\tau_N-T_{\alpha^1}), m-1) \frac{1}{(1+l+m)^{d/2-1}}\Big),
\end{align}
The sum of $m$ above gives
 \begin{align}  \label{9ea1.03}
&\sum_{m=0}^\infty  (\frac{2N+2\theta}{2N+\theta})^{m}  \Pi((2N+\theta) (\tau_N-T_{\alpha^1}), m-1) \frac{1}{(1+l+m)^{d/2-1}} \leq CI(N), 
\end{align}
where the inequality is by Lemma \ref{9l8.7}. Turning to the sum of $l$, we have
\begin{align} \label{9e9.85}
&\sum_{l=0}^\infty  (\frac{2N+2\theta}{2N+\theta})^{l}    \E\Big[1(\Gamma_{l-1}<(2N+\theta)(\tau_N-T_{\alpha^1})\leq \Gamma_{l}) \sum_{j=0}^{l} \frac{1}{(1+(\Gamma_{l-1}-\Gamma_{j})^+)^{d/2-1}}\Big] \nn \\
&\leq 5+ \sum_{l=1}^\infty  (\frac{2N+2\theta}{2N+\theta})^{l+1}    \E\Big[1(\Gamma_{l}<(2N+\theta)(\tau_N-T_{\alpha^1})\leq \Gamma_{l+1}) \sum_{j=0}^{l+1} \frac{1}{(1+(\Gamma_{l}-\Gamma_{j})^+)^{d/2-1}}\Big] \nn \\
&\leq 5+Ce^{\theta \tau_N}+ 2\sum_{l=1}^\infty       (1+\eps_N)^{l}    \E\Big[1_{(\Gamma_l<(2N+\theta)(\tau_N-T_{\alpha^1})\leq \Gamma_{l+1}) }\sum_{j=0}^{l-1} \frac{1}{(1+\Gamma_{l}-\Gamma_{j})^{d/2-1}}\Big] \nn \\
&\leq  Ce^{\theta \tau_N}+Ce^{\theta(\tau_N-T_{\alpha^1})} I((2N+\theta)(\tau_N-T_{\alpha^1}))\leq CI(N),
\end{align}
where the term $5$ on the second line comes from $l=0,1$. On the third line, the term $Ce^{\theta \tau_N}$ is from the sum for $l\geq 1$ and  $j=l,l+1$. The second last inequality follows from Lemma 4.5 of \cite{DP99}. Combining \eqref{9ea1.03} and \eqref{9e9.85}, we conclude that \eqref{9e9.83} becomes
\begin{align} \label{9ea1.06}
\E(R_2^{(ii)}) &\leq CI(N)^2\E\Big(  \sum_{\alpha^1\geq 1 } 1{\{T_{\alpha^1}<\tau_N, B^{\alpha^1}\neq \Delta\}} \Big)  \leq CN\tau_N I(N)^2,
\end{align}
where the last inequality follows from \eqref{9e9.66}.\\

\no {\bf Case (2.iii).} Turning to Case (iii), since we assume $T_{\beta^1 \wedge \gamma^1}\leq T_{\beta^2 \wedge \gamma^2}$, we must have $\gamma^2\wedge \gamma^1\geq \alpha^1$. Set $\gamma^2\wedge \gamma^1=\gamma^1\vert j'$ for some $k\leq j'\leq k+m$.
Then 
 \begin{align} \label{9ea1.04}
\E(R_2^{(iii)})\leq C\E\Big(\sum_{k=0}^\infty &\sum_{\substack{\alpha^1\geq 1,\\ \vert \alpha^1\vert =k}} \sum_{l=0}^\infty \sum_{m=0}^\infty\sum_{\substack{\beta^1\geq \alpha^1,\\ \vert \beta^1\vert =k+l}}\sum_{\substack{ \gamma^1\geq \alpha^1,\\ \vert \gamma^1\vert =k+m }}  \text{nbr}_{\beta^1,\gamma^1}(\tau_N)\nn\\
& \quad \times \sum_{j=k}^{k+l}\sum_{j'=k}^{k+m}  \sum_{\beta^2\geq \beta^1\vert j}\sum_{\gamma^2\geq\gamma^1\vert j'}  \text{nbr}_{\beta^2,\gamma^2}(\tau_N)\Big).
\end{align}

For each $k\leq j\leq k+l$ and $k\leq j'\leq k+m$, set $\vert \beta^2\vert =j+l'$ and $\vert \gamma^2\vert =j'+m'$ for some $l'\geq 0$ and $m'\geq 0$.
By conditioning on $\cH_{\beta^1}\vee\cH_{\gamma^1}$, on the event as in $\text{nbr}_{\beta^1,\gamma^1}(\tau_N)$, we have the sum of $\beta^2, \gamma^2$ equals
  \begin{align*}  
& \E\Big(\sum_{\beta^2\geq \beta^1\vert j}\sum_{\gamma^2\geq\gamma^1\vert j'}  \text{nbr}_{\beta^2,\gamma^2}(\tau_N) \Big\vert  \cH_{\beta^1}\vee\cH_{\gamma^1}\Big)
\nn\\
&\leq C \sum_{l'=0}^\infty \sum_{m'=0}^\infty (\frac{2N+2\theta}{2N+\theta})^{l'+m'}\pi((2N+\theta)(\tau_N-T_{\beta^1\vert j})^+, l'-1) \nn\\
&\quad \quad \times  \Pi((2N+\theta)(\tau_N-T_{\gamma^1\vert j'})^+, m'-1) \frac{1}{(l'+m'+1)^{d/2}}\nn\\
&\leq \frac{C }{(1+(2N+2\theta)(\tau_N-T_{\beta^1\vert j})^+)^{d/2-1}},
\end{align*}
where the last inequality follows in a similar way as in \eqref{9e9.81}. Use $2N+2\theta\geq 2N+\theta$ and $T_{\pi\beta^1}<\tau_N$ to bound the above term and then conclude that \eqref{9ea1.04} becomes
 \begin{align}  
\E(R_2^{(iii)})\leq C&\sum_{k=0}^\infty  \sum_{\substack{\alpha^1\geq 1,\\ \vert \alpha^1\vert =k}}   \sum_{l=0}^\infty \sum_{m=0}^\infty\sum_{\substack{\beta^1\geq \alpha^1,\\ \vert \beta^1\vert =k+l}}\sum_{\substack{ \gamma^1\geq \alpha^1,\\ \vert \gamma^1\vert =k+m}}   \sum_{j=k}^{k+l}  (1+l+m)\nn\\
&\quad \E\Big( \text{nbr}_{\beta^1,\gamma^1}(\tau_N) \frac{1}{(1+(2N+\theta)(T_{\pi\beta^1}-T_{\beta^1\vert j})^+))^{d/2-1}}\Big),\nn
\end{align}
 where the term $(1+l+m)$ is from the sum of $k\leq j'\leq k+l$. The right-hand side above is exactly \eqref{9e9.67}. Use \eqref{9ea1.06} to conclude 
\begin{align*} 
\E(R_2^{(iii)}) &\leq C N\tau_N I(N)^2,
\end{align*}
as required.

\subsection{Bounds for $\E(R_3)$ }
 Finally for $R_3$, we have $\beta^2$ branches off the family tree of $\alpha^1, \beta^1,\gamma^1$ from $\gamma^1$, so we let $\beta^2\wedge \gamma^1=\gamma^1\vert j$ for some $k\leq j\leq k+m$. Again there are three cases where $\gamma^2$ branches:
\begin{align*} 
\text{(i)} &\quad  \tau(\{\alpha^1,\beta^1,\gamma^1,\beta^2\}; \gamma^2)=\vert  \gamma^2  \wedge \beta^2\vert> \vert  \gamma^2  \wedge \gamma^1\vert;\nn\\
\text{(ii)}& \quad \tau(\{\alpha^1,\beta^1,\gamma^1,\beta^2\}; \gamma^2)=\vert \gamma^2 \wedge \gamma^1   \vert ;\nn\\
\text{(iii)} &\quad \tau(\{\alpha^1,\beta^1,\gamma^1,\beta^2\}; \gamma^2)=\vert \gamma^2  \wedge\beta^1  \vert . 
\end{align*}
Denote by $R_3^{(i)}, R_3^{(ii)}, R_3^{(iii)}$ the contribution to $R_3$ from these three cases, for which we refer the reader to Figure \ref{fig3}. It suffices to show that
\[
\E(R_3^{(i)}) \leq CN\tau_N I(N)^2, \quad \E(R_3^{(ii)}) \leq CN\tau_N I(N)^2, \quad \E(R_3^{(iii)}) \leq CN\tau_N I(N)^2.
\]
 \begin{figure}[ht]
  \begin{center}
    \includegraphics[width=1 \textwidth]{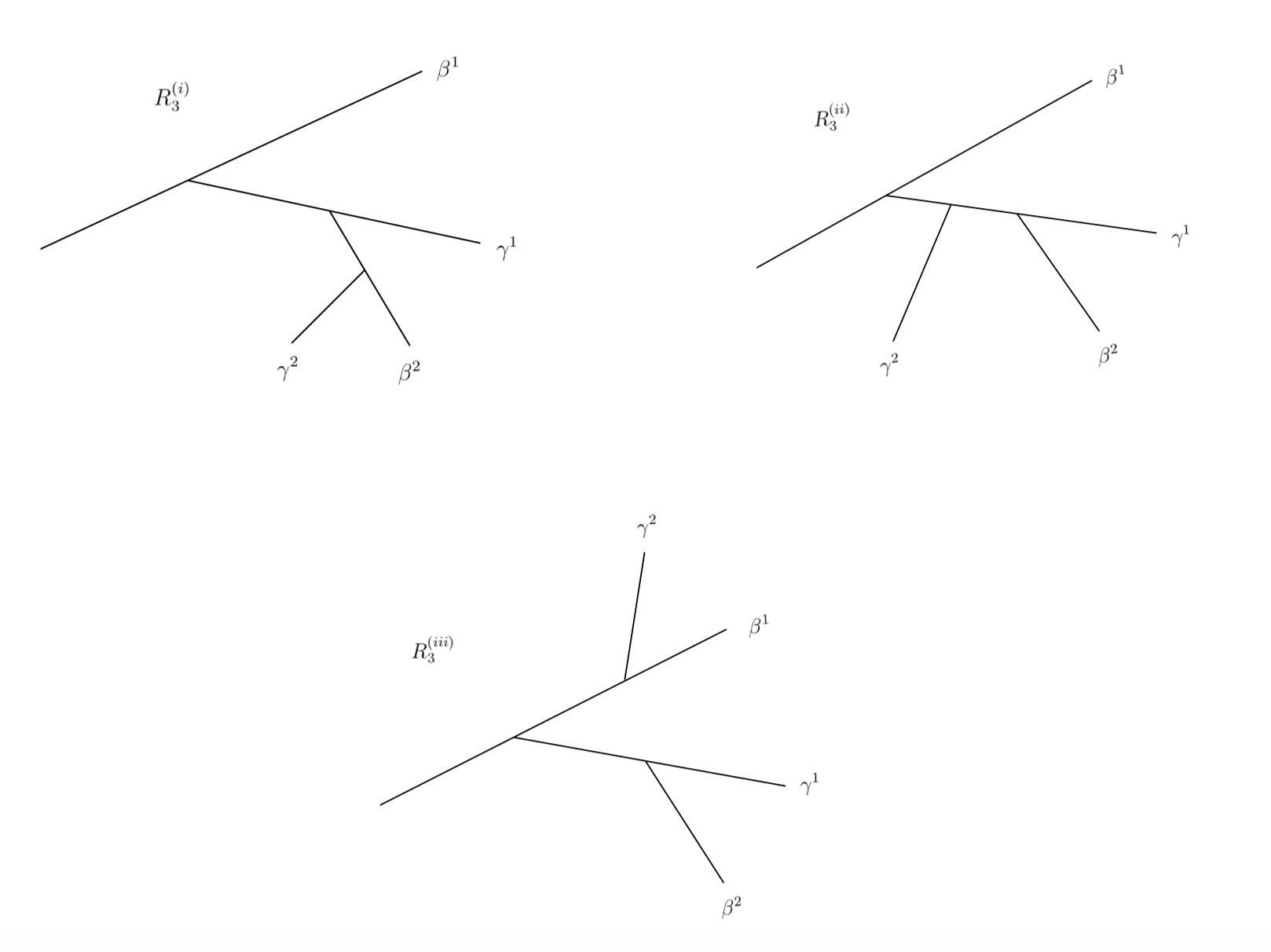}
    \caption[Branching Particle System]{\label{fig3}   Three cases for $R_3$.
      }
  \end{center}
\end{figure}

\no {\bf Case (3.i).} Case (i) implies that $\gamma^2$ branches off $\beta^2$ only after $\beta^2$ branches off $\gamma^1$. Recall \eqref{9e7.21} to get
 \begin{align} \label{9ea1.07}
\E(R_3^{(i)})\leq C\E\Big(\sum_{k=0}^\infty &\sum_{\substack{\alpha^1\geq 1,\\ \vert \alpha^1\vert =k}} \sum_{l=0}^\infty \sum_{m=0}^\infty\sum_{\substack{\beta^1\geq \alpha^1,\\ \vert \beta^1\vert =k+l}}\sum_{\substack{ \gamma^1\geq \alpha^1,\\ \vert \gamma^1\vert =k+m }}  \text{nbr}_{\beta^1,\gamma^1}(\tau_N)\nn\\
& \quad \times \sum_{j=k}^{k+m}   \sum_{\beta^2\geq  \gamma^1\vert j}\sum_{\gamma^2\geq \gamma^1\vert j} \text{nbr}_{\beta^2,\gamma^2}(\tau_N)\Big).
\end{align}
For each $k\leq j\leq k+m$, by similar reasonings used in the derivation of \eqref{9ec4.98}, we get  on the event $\text{nbr}_{\beta^1,\gamma^1}(\tau_N)$,   
 \begin{align*}  
&\E\Big(\sum_{\beta^2\geq\gamma^1\vert j}\sum_{\gamma^2\geq \gamma^1\vert j}   \text{nbr}_{\beta^2,\gamma^2}(\tau_N) \Big\vert  \cH_{\beta^1}\vee\cH_{\gamma^1}\Big)
\nn\\
&= \E\Big(   \sum_{\beta^2\geq 1}\sum_{\gamma^2\geq 1}   \text{nbr}_{\beta^2,\gamma^2}(\tau_N-T_{\gamma^1\vert j})\Big) \leq CI(N).
\end{align*}
The sum of $j$ from $k$ to $k+m$ gives at most $1+m\leq 1+m+l$. Now we may conclude that \eqref{9ea1.07} becomes
 \begin{align*}  
\E(R_3^{(i)})\leq  CI(N) \E\Big(\sum_{k=0}^\infty &\sum_{\substack{\alpha^1\geq 1,\\ \vert \alpha^1\vert =k}} \sum_{l=0}^\infty \sum_{m=0}^\infty\sum_{\substack{\beta^1\geq \alpha^1,\\ \vert \beta^1\vert =k+l}}\sum_{\substack{ \gamma^1\geq \alpha^1,\\ \vert \gamma^1\vert =k+m }}    \text{nbr}_{\beta^1,\gamma^1}(\tau_N) \times  (1+l+m)\Big),
\end{align*}
 which is exactly \eqref{9ea1.08}. Use \eqref{9ea1.09} to conclude
 \begin{align*}
 \E(R_3^{(i)})\leq CN\tau_N I(N)^2.
 \end{align*}
 
\no {\bf Case (3.ii).} For case (ii) we let $\gamma^2\wedge \gamma^1=\gamma^1\vert j'$ for some $k\leq j'\leq k+m$. 
Then
 \begin{align} \label{9ea1.10}
\E(R_3^{(ii)})\leq C\E\Big(\sum_{k=0}^\infty &\sum_{\substack{\alpha^1\geq 1,\\ \vert \alpha^1\vert =k}} \sum_{l=0}^\infty \sum_{m=0}^\infty\sum_{\substack{\beta^1\geq \alpha^1,\\ \vert \beta^1\vert =k+l}}\sum_{\substack{ \gamma^1>\alpha^1,\\ \vert \gamma^1\vert =k+1+m,\\ \gamma^1_{k+1}=0 }}  \text{nbr}_{\beta^1,\gamma^1}(\tau_N)\nn\\
& \quad \times \sum_{j=k}^{k+m}\sum_{j'=k}^{k+m}  \sum_{\beta^2\geq \gamma^1\vert j}\sum_{\gamma^2\geq\gamma^1\vert j'} \text{nbr}_{\beta^2,\gamma^2}(\tau_N)\Big).
\end{align}
For each $k\leq j\leq k+m$ and $k\leq j'\leq k+m$, set $\vert \beta^2\vert =j+u$ and $\vert \gamma^2\vert =j'+n$ for some $u\geq 0$, $n\geq 0$.
Similar to the derivation of \eqref{9e9.81}, by conditioning on $\cH_{\beta^1}\vee \cH_ {\gamma^1}$, on the event $\{T_{\pi \gamma^1}<\tau_N, B^{\gamma^1}\neq \Delta\}$ from $\text{nbr}_{\beta^1,\gamma^1}(\tau_N)$, we have the sum of $\beta^2, \gamma^2$ equals
  \begin{align}   \label{9ec5.49}
& \E\Big(\sum_{\beta^2\geq \gamma^1\vert j}\sum_{\gamma^2\geq\gamma^1\vert j'} \text{nbr}_{\beta^2,\gamma^2}(\tau_N) \Big\vert \cH_{\beta^1}\vee \cH_ {\gamma^1}\Big)
\nn\\
& \leq  C\sum_{u=0}^\infty \sum_{n=0}^\infty  (\frac{2N+2\theta}{2N+\theta})^{u+n}\pi((2N+\theta)(\tau_N-T_{\gamma^1\vert j})^+, u-1) \nn\\
&\quad \quad \times  \Pi((2N+\theta)(\tau_N-T_{\gamma^1\vert j'})^+, n-1) \frac{1}{(u+n+1)^{d/2}}\nn\\
& \leq  C \sum_{u=0}^\infty   (\frac{2N+2\theta}{2N+\theta})^{u}\pi((2N+\theta)(\tau_N-T_{\gamma^1\vert j})^+, u-1)    \frac{1}{(u +1)^{d/2-1}}\nn\\
& \leq   \frac{C}{(1+(2N+2\theta)(\tau_N-T_{\gamma^1\vert j})^+)^{d/2-1}},
\end{align}
where the second inequality follows by Lemma \ref{9l4.0} and the third inequality uses Lemma \ref{9l4.3}(i). Use $2N+2\theta\geq 2N+\theta$ and $T_{\pi\gamma^1}<\tau_N$ to bound the above term and then conclude that \eqref{9ea1.10} becomes 
   \begin{align}  \label{9e9.94}
\E(R_3^{(ii)}) \leq C &\sum_{k=0}^\infty  \sum_{\substack{\alpha^1\geq 1,\\ \vert \alpha^1\vert =k}} \sum_{l=0}^\infty \sum_{m=0}^\infty\sum_{\substack{\beta^1\geq \alpha^1,\\ \vert \beta^1\vert =k+l }}\sum_{\substack{ \gamma^1\geq \alpha^1,\\ \vert \gamma^1\vert =k+m }}   (1+l+m)\\
&\quad \E\Big( \text{nbr}_{\beta^1,\gamma^1}(\tau_N)  \sum_{j=k}^{k+m}\frac{1}{(1+(2N+\theta)(T_{\pi\gamma^1}-T_{\gamma^1\vert j})^+)^{d/2-1}}\Big),\nn
\end{align}
where the term $(1+l+m)$ is from the sum of $k\leq j'\leq k+m$.
 By conditioning on $\cH_{\alpha^1}$, we get that the expectation above is at most
\begin{align*} 
&\E\Bigg\{1{\{T_{\alpha^1}<\tau_N, B^{\alpha^1}\neq \Delta\}}        (\frac{N+\theta}{2N+\theta})^{l+m}  \nn\\
&\times \E\Big[1(T_{\pi \gamma^1}<\tau_N) \sum_{j=k}^{k+m} \frac{1}{(1+(2N+\theta)(T_{\pi\gamma^1}-T_{\gamma^1\vert j})^+)^{d/2-1}}\Big\vert \cH_{\alpha^1}\Big]\nn \\
&\times  \P\Big(T_{\pi \beta^1}<\tau_N\leq T_{\beta^1}\Big\vert \cH_{\alpha^1}\Big)\cdot \frac{C}{(1+l+m)^{d/2}}\Bigg\},
\end{align*}
where we have used Lemma \ref{9l4.2} to bound the probability $\P(B^{\beta^1}-B^{\gamma^1}\in \cN_N\vert \cH_{\alpha^1})$. Use the Gamma random variables to see that the above is at most
 \begin{align}  \label{9e9.96}
&\E\Bigg\{ 1{\{T_{\alpha^1}<\tau_N, B^{\alpha^1}\neq \Delta\}}        (\frac{N+\theta}{2N+\theta})^{l+m}  \nn\\
&\times \E\Big[1(\Gamma_{m-1}<(2N+\theta)(\tau_N-T_{\alpha^1})) \sum_{j=k}^{k+m} \frac{1}{(1+(\Gamma_{m-1}-\Gamma_{j-k})^+)^{d/2-1}}\Big]\nn \\
&\times  \pi((2N+\theta)(\tau_N-T_{\alpha^1}),l-1) \cdot \frac{C}{(1+l+m)^{d/2}}\Bigg\}.
\end{align}
Now we have \eqref{9e9.94} becomes
 \begin{align}  \label{9e9.97} 
\E(R_3^{(ii)}) &\leq C\E\Big(  \sum_{\alpha^1\geq 1 } 1{\{T_{\alpha^1}<\tau_N, B^{\alpha^1}\neq \Delta\}}\sum_{l=0}^\infty   \sum_{m=0}^\infty    (\frac{2N+2\theta}{2N+\theta})^{l+m}  \nn\\
&\times\E\Big[1(\Gamma_{m-1}<(2N+\theta)(\tau_N-T_{\alpha^1})) \sum_{j=0}^{m} \frac{1}{(1+(\Gamma_{m-1}-\Gamma_{j})^+)^{d/2-1}}\Big] \nn \\
&\times   \pi((2N+\theta) (\tau_N-T_{\alpha^1}), l-1) \frac{1}{(1+l+m)^{d/2-1}}\Big),
\end{align}
The sum of $l$ gives
\begin{align}  \label{e2.2}
&\sum_{l=0}^\infty (\frac{2N+2\theta}{2N+\theta})^{l} \pi((2N+\theta)(\tau_N-T_{\alpha^1}), l-1)  \frac{1}{(1+l+m)^{d/2-1}} \nn\\
&\leq \frac{C}{(1+(2N+2\theta)(\tau_N-T_{\alpha^1}))}\leq \frac{C}{(1+N(\tau_N-T_{\alpha^1}))},
\end{align}
where the first inequality is from $m\geq 0$, Lemma \ref{9l4.3} and $d\geq 4$.
The sum of $m$ gives
\begin{align}  \label{9ec9.99}
\sum_{m=0}^\infty    &   (\frac{2N+2\theta}{2N+\theta})^{m}   \E\Big[1(\Gamma_{m-1}<(2N+\theta)(\tau_N-T_{\alpha^1})) \sum_{j=0}^{m} \frac{1}{(1+(\Gamma_{m-1}-\Gamma_{j})^+)^{d/2-1}}\Big] \nn \\
\leq  &\sum_{m>AN(\tau_N-T_{\alpha^1})}       (\frac{2N+2\theta}{2N+\theta})^{m}   \E\Big[1(\Gamma_{m-1}<(2N+\theta)(\tau_N-T_{\alpha^1})) \times 2m\Big]\nn\\
 &\quad + \sum_{m\leq AN(\tau_N-T_{\alpha^1})}  e^{A\theta \tau_N}  \E\Big[ \sum_{j=0}^{m}  \frac{1}{(1+(\Gamma_{m-1}-\Gamma_j)^+)^{d/2-1}}\Big].
\end{align}
Lemma \ref{9l4.3} (ii) implies that the first summation above is at most $C2^{-N(\tau_N-T_{\alpha^1})}\leq C$. Turning to the second summation,
for each $m\leq AN(\tau_N-T_{\alpha^1})$, we have
\begin{align*}  
&\E\Big[ \sum_{j=0}^{m}  \frac{1}{(1+(\Gamma_{m-1}-\Gamma_j)^+)^{d/2-1}}\Big] \leq C+\E\Big[ \sum_{j=d}^{AN\tau_N}  \frac{1}{(1+\Gamma_j)^{d/2-1}}\Big]\nn\\
&\leq C+  C\sum_{j=d}^{AN\tau_N}  \frac{1}{(1+j)^{d/2-1}}\leq CI(N),
\end{align*}
where the second inequality uses Lemma 10.4 of \cite{DP99}.
Thus, we may bound \eqref{9ec9.99} by 
\begin{align}  \label{e2.1}
C+C(1+AN(\tau_N-T_{\alpha^1})) I(N)\leq   C(1+N(\tau_N-T_{\alpha^1}))I(N).
\end{align}
Combine \eqref{e2.2} and \eqref{e2.1} to see that \eqref{9e9.97} becomes
 \begin{align}\label{9e9.99}
\E(R_3^{(ii)})& \leq CI(N)  \E\Big( \sum_{\alpha^1\geq 1 } 1(T_{\alpha^1}<\tau_N, B^{\alpha^1}\neq \Delta) \Big)\nn\\
&\leq C N \tau_N I(N)\leq C N \tau_N I(N)^2,
\end{align}
where the second inequality follows by \eqref{9e9.66}.\\

\no {\bf Case (3.iii).} Turning to Case (iii), since we assumed at the beginning that $T_{\beta^1 \wedge \gamma^1}\leq T_{\beta^2 \wedge \gamma^2}$, we must have $\gamma^2\wedge \beta^1\geq \beta^1 \wedge \gamma^1=\beta^1\vert k$. Set $\gamma^2\wedge \beta^1=\beta^1\vert j'$ for some $k\leq j'\leq k+l$. 
Then
 \begin{align} \label{9ea1.11}
\E(R_3^{(iii)})\leq C\E\Big(\sum_{k=0}^\infty &\sum_{\substack{\alpha^1\geq 1,\\ \vert \alpha^1\vert =k}} \sum_{l=0}^\infty \sum_{m=0}^\infty\sum_{\substack{\beta^1\geq \alpha^1,\\ \vert \beta^1\vert =k+l}}\sum_{\substack{ \gamma^1\geq \alpha^1,\\ \vert \gamma^1\vert =k+m}}  \text{nbr}_{\beta^1,\gamma^1}(\tau_N)\nn\\
& \quad \times \sum_{j=k}^{k+m}\sum_{j'=k}^{k+l}  \sum_{\beta^2\geq \gamma^1\vert j}\sum_{\gamma^2\geq\beta^1\vert j'} \text{nbr}_{\beta^2,\gamma^2}(\tau_N)\Big).
\end{align}
 For each $k\leq j\leq k+m$ and $k\leq j'\leq k+l$, set $\vert \beta^2\vert =j+l'$ and $\vert \gamma^2\vert =j'+m'$ for some $l'\geq 0$ and $m'\geq 0$.
By conditioning on $\cH_{\beta^1}\vee\cH_{\gamma^1}$, on the event $\text{nbr}_{\beta^1,\gamma^1}(\tau_N)$, we have the sum of $\beta^2, \gamma^2$ is at most
  \begin{align*}  
& \E\Big(\sum_{\beta^2\geq \gamma^1\vert j}\sum_{\gamma^2\geq\beta^1\vert j'} \text{nbr}_{\beta^2,\gamma^2}(\tau_N) \Big\vert \cH_{\beta^1}\vee\cH_{\gamma^1}\Big)
\nn\\
& \leq C\sum_{l'=0}^\infty \sum_{m'=0}^\infty (\frac{2N+2\theta}{2N+\theta})^{l'+m'}\pi((2N+\theta)(\tau_N-T_{\gamma^1\vert j})^+, l'-1) \nn\\
&\quad \quad \times  \Pi((2N+\theta)(\tau_N-T_{\beta^1\vert j'})^+, m'-1) \frac{1}{(l'+m'+1)^{d/2}}\nn\\
&\leq \frac{C}{(1+(2N+2\theta)(\tau_N-T_{\gamma^1\vert j})^+)^{d/2-1}},
\end{align*}
where the last inequality follows in a similar reasoning in \eqref{9ec5.49}. Use $2N+2\theta\geq 2N+\theta$ and $T_{\pi\gamma^1}<\tau_N$ to bound the above term and then conclude that \eqref{9ea1.11} becomes 
   \begin{align*}  
\E(R_3^{(ii)}) \leq C &\sum_{k=0}^\infty  \sum_{\substack{\alpha^1\geq 1,\\ \vert \alpha^1\vert =k}} \sum_{l=0}^\infty \sum_{m=0}^\infty\sum_{\substack{\beta^1\geq \alpha^1,\\ \vert \beta^1\vert =k+l }}\sum_{\substack{ \gamma^1\geq \alpha^1,\\ \vert \gamma^1\vert =k+m }}   (1+l+m)\\
&\quad \E\Big( \text{nbr}_{\beta^1,\gamma^1}(\tau_N)  \sum_{j=k}^{k+m}\frac{1}{(1+(2N+\theta)(T_{\pi\gamma^1}-T_{\gamma^1\vert j})^+)^{d/2-1}}\Big),\nn
\end{align*}
where the term $(1+l+m)$ is from the sum of $k\leq j'\leq k+l$. The right-hand side above is exactly \eqref{9e9.94}. 
By \eqref{9e9.99}, we conclude 
\begin{align*} 
\E(R_3^{(iii)}) &\leq C N\tau_N I(N)^2,
\end{align*}
as required.  \\

The proof of Lemma \ref{9l7.1} is now complete as noted below \eqref{e2}.
\end{proof}

\section{Proofs of Lemmas \ref{9l5.4} and \ref{9l5.8}}\label{9s8}

The last section is devoted to proving the last two lemmas, Lemmas \ref{9l5.4} and \ref{9l5.8},  from Section \ref{9s5}. We first consider Lemma \ref{9l5.4}. Recall the definition of $\zeta_\alpha^m$ with $m\geq 0$ from \eqref{9ea3.01} to see that
\begin{align*}
1(\zeta_\alpha^m>T_{\pi\alpha})=1(\zeta_\alpha^m\geq T_{\alpha}).
\end{align*}
Use the above to rewrite  $K_{t}^{n,2}(\phi)$ from \eqref{9e3.5} as 
\begin{align*} 
K_{t}^{n,2}(\phi)=&\frac{1}{2N+\theta}  \frac{1}{\psi(N)}  \sum_{\beta} 1(T_\beta\leq t, \zeta_\beta^n\geq T_{\beta})  \phi(B^\beta)\nn\\
& \sum_{\gamma} 1(T_{\gamma \wedge \beta}>T_\beta-\tau_N) 1(T_{\pi\gamma}<T_\beta, \zeta_\gamma^{n-1}\geq T_{\gamma}) 1(B^{\gamma}-B^{\beta}\in \cN_N).
\end{align*}
By comparing to $K_{t}^{n,3}(\phi)$ as in \eqref{9e3.6}, we get
\begin{align} \label{9e8.01}
\sup_{s\leq t}\vert K_{s}^{n,2}(\phi)-K_{s}^{n,3}(\phi)\vert &\leq \frac{1}{2N+\theta}  \frac{1}{\psi(N)} \| \phi\| _\infty \sum_{\beta} 1(T_\beta\leq t)\sum_{\gamma}  1(T_{\pi\gamma}<T_\beta) \\
&\times    1(T_{\gamma \wedge \beta}>T_\beta-\tau_N)  1(B^{\gamma}-B^{\beta}\in \cN_N)\nn\\
&\times \Big[1(\zeta_\beta^n> T_{\beta}-\tau_N)1(\zeta_\gamma^{n-1}> T_{\beta}-\tau_N)-1(\zeta_\beta^n\geq T_{\beta})1(\zeta_\gamma^{n-1}\geq T_{\gamma})\Big].\nn
\end{align}
Since  $T_{\gamma \wedge \beta}>T_\beta-\tau_N$ implies $T_{\gamma}>T_\beta-\tau_N$, we may bound the square bracket above by 
\begin{align*}
1(\zeta_\beta^n \in (T_{\beta}-\tau_N, T_{\beta}))+1(\zeta_\gamma^{n-1} \in (T_{\beta}-\tau_N, T_{\gamma})).
\end{align*}
To ease the notation, we define
\begin{align*} 
\text{nbr}_{\beta}(\gamma)=1(T_{\gamma \wedge \beta}>T_\beta-\tau_N, T_{\pi \gamma}<T_{\beta}, B^\beta-B^\gamma \in \cN_N),
\end{align*}
and 
\begin{align} \label{9ea3.08}
I(\beta)=\frac{1}{\psi_0(N)} \sum_{\gamma} \text{nbr}_{\beta}(\gamma).
\end{align}
Using the above in \eqref{9e8.01} and then taking expectation,  we get
\begin{align} \label{9e8.02}
\E\Big(&\sup_{s\leq t}\vert K_{s}^{n,2}(\phi)-K_{s}^{n,3}(\phi)\vert \Big)\leq    \frac{1}{N^2} \| \phi\| _\infty  \E\Big(\sum_{\beta} 1(T_{\beta}\leq t, B^\beta\neq \Delta)   1_{\zeta_\beta^n \in (T_{\beta}-\tau_N, T_{\beta})} I(\beta)\Big)\nn\\
&+ \frac{1}{N\psi(N)} \| \phi\| _\infty  \E\Big(\sum_{\beta} 1(T_{ \beta}\leq t, B^\beta\neq \Delta)  \sum_{\gamma}\text{nbr}_{\beta}(\gamma) 1_{\zeta_\gamma^{n-1} \in (T_{\beta}-\tau_N, T_\gamma)} \Big).
\end{align}

\begin{lemma}\label{9l10.1}
For any $m\geq 0$ and $t\geq 0$, we have
\begin{align} \label{9ea1.15}
(i) \lim_{N\to \infty}  \frac{1}{N^2}   \E\Big(\sum_{\beta} 1(T_{\beta}\leq t, B^\beta\neq \Delta)  1_{\zeta_\beta^m \in (T_{\beta}-\tau_N, T_{\beta})}   (I(\beta)+1)\Big)=0.
\end{align}
and
\begin{align} \label{9ea1.16}
(ii) &\lim_{N\to \infty}  \frac{1}{N\psi(N)}   \E\Big(\sum_{\beta} 1(T_{ \beta}\leq t, B^\beta\neq \Delta)  \sum_{\gamma}\text{nbr}_{\beta}(\gamma) 1_{\zeta_\gamma^{m} \in (T_{\beta}-\tau_N, T_\gamma)} \Big)=0.
\end{align}
\end{lemma}
Unlike the case in \cite{DP99}, the symmetry between $\beta,\gamma$ here is not explicit, but we will show later in Section \ref{s9.2} that \eqref{9ea1.16} is an easy consequence of \eqref{9ea1.15}. For now, the proof of Lemma \ref{9l5.4} is immediate by \eqref{9e8.02} and Lemma \ref{9l10.1}. The extra $+1$ in  \eqref{9ea1.15} is intended for the proof of Lemma \ref{9l5.8}, which we now give.

\begin{proof}[Proof of Lemma \ref{9l5.8}]
Recall $X_r^{1,\tau}(\phi)$ from \eqref{9e3.13} and $X_r^1(\phi)$ from \eqref{9e0.03}. Using $\phi\in C_b^3(\R^d)$, we get
\begin{align*} 
&\vert X_r^{1,\tau}(\phi)-X_r^1(\phi)\vert \leq \frac{1}{N} \| \phi\| _\infty \sum_{\beta} 1(T_{\pi\beta}<r< T_\beta, B^\beta\neq \Delta)1(\zeta_\beta^1\in (r-\tau_N,r))\nn\\
&\quad \quad\quad +\frac{C}{N}   \sum_{\beta} 1(T_{\pi\beta}<r< T_\beta, B^\beta\neq \Delta) \vert B_r^\beta-B_{(r-\tau_N)^+}^\beta\vert .
\end{align*}
It follows that 
\begin{align} \label{9ea3.02}
&\E\Big(\int_0^t \vert X_r^{1,\tau}(\phi)-X_r^1(\phi)\vert  dr\Big)\nn\\
&\leq \frac{1}{N} \| \phi\| _\infty \int_0^t \E\Big(\sum_{\beta} 1(T_{\pi\beta}<r< T_\beta, B^\beta\neq \Delta)1(\zeta_\beta^1\in (r-\tau_N,r))\Big)dr\nn\\
&\quad  +\frac{C}{N}   \int_0^t \E\Big(\sum_{\beta} 1(T_{\pi\beta}<r< T_\beta, B^\beta\neq \Delta) \vert B_r^\beta-B_{(r-\tau_N)^+}^\beta\vert \Big)dr:=I_1+I_2.
\end{align}
To take care of the first term, we apply Lemma \ref{9l3.2} to see that
\begin{align*} 
I_1=\frac{ \| \phi\| _\infty }{N(2N+\theta)}\E\Big(\sum_{\beta} 1(T_{\beta}\leq t, B^\beta\neq \Delta)  1{(\zeta_\beta^n \in (T_\beta-\tau_N, T_\beta))} \Big).
\end{align*}
By comparing to \eqref{9ea1.15}, it is immediate that $I_1\to 0$ as $N\to\infty$. 

Turning to $I_2$, by considering $\beta_0=i$ for some $1\leq i\leq NX_0^0(1)$, we get 
\begin{align*} 
I_2=&\frac{C}{N}  NX_0^0(1) \int_0^t \E\Big(\sum_{\beta\geq 1} 1(T_{\pi\beta}<r< T_\beta, B^\beta\neq \Delta)\vert B_r^\beta-B_{(r-\tau_N)^+}^\beta\vert  \Big)  dr\nn\\
&\leq CX_0^0(1) \int_0^t  C  \sqrt{r-(r-\tau_N)^+} dr\leq CX_0^0(1)\sqrt{\tau_N} \to 0,
\end{align*}
where the first inequality follows from \eqref{9ec1.23}. The proof is now complete.
\end{proof}

It remains to prove Lemma \ref{9l10.1}.

\subsection{Proof of Lemma \ref{9l10.1}(i)}

Since $B^\beta=B^\beta_{T_\beta^-}\neq \Delta$, we must have $\zeta_\beta^0=T_{\beta}$ by definition. If $m=0$, we cannot have both $\zeta_\beta^0=T_{\beta}$ and $\zeta_\beta^0\in (T_\beta-\tau_N, T_\beta)$ hold. Hence \eqref{9ea1.15} is trivial for $m=0$. It suffices to prove for $m\geq 1$.
 In order that $\zeta_\beta^m \in (T_{\beta}-\tau_N, T_{\beta})$, there has to be some $i<\vert \beta\vert $ so that 
\begin{align*} 
T_{\beta\vert i}>T_{\beta}-\tau_N, \quad B^{\beta\vert i}+W^{\beta\vert i} \in \overline{\cR}^{X^{m-1}}_{T_{\beta\vert i}^-}, \quad e_{\beta\vert i}=\beta_{i+1}.
\end{align*}
The above means that at time $T_{\beta\vert i}$, the particle $\beta\vert i$ gives birth to $\beta\vert (i+1)$ which is displaced from $B^{\beta\vert i}$ by a distance of $W^{\beta\vert i}$. However, there exists some particle $\delta$ in $X^{m-1}$ which had already visited that location before time $T_{\beta\vert i}$, or that location lies in $K_0^N$. Hence we may bound $1\{\zeta_\beta^m \in (T_{\beta}-\tau_N, T_{\beta})\}$ by
\begin{align}\label{9ec9.53}
&\sum_{i=0}^{\vert \beta\vert -1} 1(T_{\beta\vert i}>T_{\beta}-\tau_N) \sum_{\delta} 1(T_{\pi\delta} <T_{\beta\vert i}, B^{\beta\vert i}+W^{\beta\vert i}=B^\delta)\nn\\
&\quad +\sum_{i=0}^{\vert \beta\vert -1} 1(T_{\beta\vert i}>T_{\beta}-\tau_N)  1(B^{\beta\vert i}+W^{\beta\vert i}\in K_0^N),
\end{align}
where we have dropped the condition $e_{\beta\vert i}=\beta_{i+1}$.
Use the above to get that the left-hand side term of \eqref{9ea1.15} is at most
\begin{align}  \label{9e8.03}
&\frac{1}{N^2}  \E\Big(\sum_{\beta} 1(T_{ \beta} \leq t, B^{\beta}\neq \Delta) \sum_{i=0}^{\vert \beta\vert -1} 1(T_{\beta\vert i}>T_\beta-\tau_N)  \nn\\
&\quad 1(B^{\beta\vert i}+W^{\beta\vert i} \in K_0^N) \times (I(\beta)+1) \Big) \nn\\
&+\frac{1}{N^2} \E\Big(\sum_{\beta} 1(T_{ \beta} \leq t, B^{\beta}\neq \Delta) \sum_{i=0}^{\vert \beta\vert -1} 1(T_{\beta\vert i}>T_\beta-\tau_N)  \nn\\
&\quad\sum_{\delta} 1(T_{\pi\delta} <T_{\beta\vert i}, B^{\beta\vert i}+W^{\beta\vert i}=B^\delta)     \Big)\nn\\
&+\frac{1}{N^2} \E\Big(\sum_{\beta} 1(T_{ \beta} \leq t, B^{\beta}\neq \Delta) \sum_{i=0}^{\vert \beta\vert -1} 1(T_{\beta\vert i}>T_\beta-\tau_N)  \nn\\
&\quad\sum_{\delta} 1(T_{\pi\delta} <T_{\beta\vert i}, B^{\beta\vert i}+W^{\beta\vert i}=B^\delta) \times I(\beta)  \Big):=I_0+I_1+I_2.
\end{align}
It suffices to show $I_i\to 0$ for $i=0,1,2$.

\subsubsection{Convergence of $I_0$}

The proof of $I_0\to 0$ relies largely on the assumption of $K_0^N$ from \eqref{9e0.93}:
 \begin{align}\label{9e3.49}
\lim_{N\to \infty} \frac{1}{N} \E\Big( \sum_{\beta} 1(T_\beta\leq t, B^\beta\neq \Delta)  1(B^\beta+W^\beta\in K_0^N)\Big)=0.
\end{align}
Consider $\beta_0=i_0$ for some $1\leq i_0\leq NX_0^0(1)$ and $|\beta|=l$ for some $l\geq 0$. It follows that
 \begin{align}\label{9e3.491}
&\frac{1}{N} \E\Big( \sum_{\beta} 1(T_\beta\leq t, B^\beta\neq \Delta)  1(B^\beta+W^\beta\in K_0^N)\Big)\nn\\
&=\frac{1}{N} \sum_{i_0=1}^{NX_0^0(1)} \sum_{l=0}^\infty  \sum_{\beta\geq i_0, |\beta|=l} \E\Big(1(T_\beta\leq t, B^\beta\neq \Delta)  1(B^\beta+W^\beta\in K_0^N)\Big)\nn\\
&=\frac{1}{N} \sum_{i_0=1}^{NX_0^0(1)} \sum_{l=0}^\infty (\frac{2N+2\theta}{2N+\theta})^l \cdot \P(\Gamma_{l+1}\leq (2N+\theta)t ) \cdot \P(x_{i_0}+V_{l}^N+W^N\in K_0^N),
\end{align}
where $W^N$ is uniform on $\cN_N$, independent of $V_{l}^N$. Lemma \ref{9l4.3}(ii) implies that the above sum for $l>ANt$ is at most $CX_0^0(1) 2^{-Nt} \to 0$. Hence one may easily check that \eqref{9e3.49} is equivalent to
 \begin{align}\label{9e3.490}
\lim_{N\to \infty}\frac{1}{N} \sum_{i_0=1}^{NX_0^0(1)} \sum_{l=0}^{ANt}   \P(\Gamma_{l+1}\leq (2N+\theta)t ) \cdot \P(x_{i_0}+V_{l}^N+W^N\in K_0^N)=0.
\end{align}

To prove $I_0\to 0$, it suffices to show that
 \begin{align} \label{9ea3.444}
I_0^{(a)}:=&\frac{1}{N^2}  \E\Big(\sum_{\beta} 1(T_{ \beta} \leq t, B^{\beta}\neq \Delta) \sum_{i=0}^{\vert \beta\vert -1} 1(T_{\beta\vert i}>T_\beta-\tau_N)  \nn\\
&\quad 1(B^{\beta\vert i}+W^{\beta\vert i} \in K_0^N)   \Big)\to 0,
\end{align}
and
\begin{align}  \label{9ea3.445}
I_0^{(b)}:=&\frac{1}{N^2}  \E\Big(\sum_{\beta} 1(T_{ \beta} \leq t, B^{\beta}\neq \Delta) \sum_{i=0}^{\vert \beta\vert -1} 1(T_{\beta\vert i}>T_\beta-\tau_N)  \nn\\
&\quad 1(B^{\beta\vert i}+W^{\beta\vert i} \in K_0^N) \times I(\beta)  \Big)\to 0.
\end{align}

\no {\bf  Case (a).} Turning to $I_0^{(a)}$, we let $\beta_0=i_0$ for some $1\leq i_0\leq NX_0^0(1)$ and set $|\beta|=l$ for some $l\geq 1$.  Then
\begin{align*} 
I_0^{(a)}=&\frac{1}{N^2}  \sum_{i_0=1}^{NX_0^0(1)} \sum_{l=1}^\infty (\frac{2N+2\theta}{2N+\theta})^l    \E\Big(1(\Gamma_{l+1}\leq (2N+\theta)t)  \nn\\
& \sum_{i=0}^{l-1} 1(\Gamma_{i+1}>\Gamma_{l+1}-(2N+\theta)\tau_N) \Big) \cdot \P(x_{i_0}+V_{i}^N+W^N\in K_0^N). 
\end{align*}
Use Fubini's theorem to exchange the sum of $i,l$ to get
\begin{align*}  
I_0^{(a)}=& \frac{1}{N^2}  \sum_{i_0=1}^{NX_0^0(1)} \sum_{i=0}^{\infty} (\frac{2N+2\theta}{2N+\theta})^i  \P(x_{i_0}+V_{i}^N+W^N\in K_0^N)\nn\\
& \sum_{l=i+1}^\infty (\frac{2N+2\theta}{2N+\theta})^{l-i}    \E\Big(1(\Gamma_{l+1}\leq (2N+\theta)t) 1(\Gamma_{l+1}-\Gamma_{i+1}<(2N+\theta)\tau_N) \Big). 
\end{align*}
Bounding $1(\Gamma_{l+1}\leq (2N+\theta)t)$ by $1(\Gamma_{i+1}\leq (2N+\theta)t)$, we get the above is at most
\begin{align*}  
I_0^{(a)}\leq & \frac{1}{N^2}  \sum_{i_0=1}^{NX_0^0(1)} \sum_{i=0}^{\infty} (\frac{2N+2\theta}{2N+\theta})^i  \P(x_{i_0}+V_{i}^N+W^N\in K_0^N)\nn\\
& \P(\Gamma_{i+1}\leq (2N+\theta)t) \sum_{l=1}^\infty (\frac{2N+2\theta}{2N+\theta})^{l}     \P(\Gamma_{l}<(2N+\theta)\tau_N). 
\end{align*}
By Lemma \ref{9l4.3}(ii), the sum of $l$ is at most
\begin{align*}  
C2^{-N\tau_N}+\sum_{l=1}^{AN\tau_N} (\frac{2N+2\theta}{2N+\theta})^{l}      \leq CN\tau_N\leq CN.
\end{align*}
It follows that
\begin{align*}  
I_0^{(a)}\leq & \frac{C}{N}  \sum_{i_0=1}^{NX_0^0(1)} \sum_{i=0}^{\infty} (\frac{2N+2\theta}{2N+\theta})^i  \P(x_{i_0}+V_{i}^N+W^N\in K_0^N) 
\P(\Gamma_{i+1}\leq (2N+\theta)t).
\end{align*}
In view of \eqref{9e3.491}, we conclude $I_0^{(a)}\to 0$.  \\

\no {\bf  Case (b).} Turning to $I_0^{(b)}$ from \eqref{9ea3.445}, we recall $I(\beta)$ from \eqref{9ea3.08} to see that
\begin{align}  \label{9ea5.667}
I_0^{(b)}&=\frac{1}{N\psi(N)} \E\Big(\sum_{\beta} 1(T_{ \beta} \leq t)\sum_{\gamma}1(T_{\pi\gamma}<T_\beta) 1(B^\beta-B^\gamma\in \cN_N) \nn\\
&1(T_{\gamma\wedge \beta}>T_\beta-\tau_N)  \sum_{i=0}^{ \vert \beta\vert-1} 1(T_{\beta\vert i}>T_\beta-\tau_N)  1(B^{\beta\vert i}+W^{\beta\vert i}\in K_0^N) \Big).
\end{align}
 Since $T_{\gamma\wedge \beta}>T_\beta-\tau_N$ implies $\gamma_0=\beta_0$, we may let $\alpha=\beta\wedge \gamma$ and use $T_{\alpha}\leq T_\beta< T_{\alpha}+\tau_N$ to get  
 \begin{align} \label{9ea5.666}
I_0^{(b)}&\leq \frac{1}{N\psi(N)} \E\Big(\sum_{\alpha}\sum_{\beta\geq\alpha} \sum_{\gamma\geq \alpha}1(T_{\alpha} \leq t)   1(T_\beta<T_{\alpha}+\tau_N) 1(T_{\pi\gamma}<T_{\alpha}+\tau_N)  \nn\\
& 1(B^\beta-B^\gamma\in \cN_N) \sum_{i=0}^{\vert \beta\vert -1}1(T_{\beta\vert i}>T_\beta-\tau_N) 1(B^{\beta\vert i}+W^{\beta\vert i}\in K_0^N) \Big).
\end{align}
Denote by $I_0^{(b,1)}$ (resp. $I_0^{(b,2)}$) the sum of $i\leq \vert \alpha\vert -1$ (resp. $ \vert \alpha\vert \leq  i \leq \vert \beta\vert -1$) in the above expression. It suffices to prove $I_0^{(b,1)}\to 0$ and $I_0^{(b,2)}\to 0$.\\

\no {\bf  Case (b,1).} For $I_0^{(b,1)}$, since $i\leq \vert \alpha\vert -1$, we get  $\beta\vert i=\alpha\vert i$ and $T_{\beta\vert i}, B^{\beta\vert i}, W^{\beta\vert i}$ are all measurable with respect to $\cH_{\alpha}$. Use Fubini's theorem and $T_\beta\geq T_{\alpha}$ to see
\begin{align} \label{9ea3.513}
I_0^{(b,1)}\leq \frac{1}{N\psi(N)} \E\Big(&\sum_{\alpha}1(T_{ \alpha} \leq t) \sum_{i=0}^{\vert \alpha\vert -1}1(T_{\alpha\vert i}>T_\alpha-\tau_N)  1(B^{\alpha\vert i}+W^{\alpha\vert i}\in K_0^N) \\
& \sum_{\beta\geq\alpha} \sum_{\gamma\geq \alpha}1(T_{\beta}-T_{\alpha}<\tau_N) 1(T_{\pi\gamma}-T_{\alpha}<\tau_N) 1({B}^\beta-{B}^\gamma\in \cN_N)\Big).\nn
\end{align}
By considering $\vert \beta\vert =\vert \alpha\vert +l$ and $\vert \gamma\vert =\vert \alpha\vert +m$ for some $l,m\geq 0$, we  get the sum of $\beta,\gamma$ is at most (recall $\hat{B}^\beta, \hat{B}^\gamma$ from \eqref{9ea1.71})
 \begin{align} \label{9ea2.801}
& \sum_{l=0}^\infty \sum_{m=0}^\infty \sum_{\substack{\beta\geq \alpha,\\ \vert \beta\vert =\vert \alpha\vert +l}}\sum_{\substack{ \gamma\geq \alpha,\\ \vert \gamma\vert =\vert \alpha\vert +m }}  1(T_{\beta}-T_{\alpha}<\tau_N)\nn\\
& 1(T_{\pi\gamma}-T_{\alpha}<\tau_N) 1(\sqrt{N}(\hat{B}^\beta- \hat{B}^\gamma)\in  [-2,2]^d),
\end{align}
where we   have used \eqref{9ea2.82} to bound $1(B^\beta-B^\gamma\in \cN_N)$. 
By conditioning on $\cH_\alpha$,  the conditional expectation of \eqref{9ea2.801} is at most
\begin{align} \label{9ea3.511}
&1(B^\alpha\neq \Delta)  \sum_{l=0}^\infty \sum_{m=0}^\infty  (\frac{2N+2\theta}{2N+\theta})^{l+m}  \Pi((2N+\theta)\tau_N, l) \Pi((2N+\theta)\tau_N, m-1) \frac{C}{(1+l+m)^{d/2}}.
\end{align}
The sum of $m$ gives at most ${C }/{(1+l)^{d/2-1}}$ by Lemma \ref{9l4.0}. Then the sum of $l$ gives
\begin{align*}
 \sum_{l=0}^\infty   (\frac{2N+2\theta}{2N+\theta})^{l}  \Pi((2N+\theta)\tau_N, l) \frac{C}{(1+l)^{d/2-1}}\leq CI(N),
\end{align*}
where the inequality is by Lemma \ref{9l8.7}.  We conclude from the above that
\begin{align} \label{9ea3.512}
\E\Big(& \sum_{\beta\geq\alpha} \sum_{\gamma\geq \alpha} 1(T_{\beta}-T_{\alpha}<\tau_N) 1(T_{\pi\gamma}-T_{\alpha}<\tau_N) \nn\\
 &\quad 1({B}^\beta-{B}^\gamma\in \cN_N)\Big\vert \cH_\alpha\Big)\leq CI(N)1(B^\alpha\neq \Delta).
\end{align}
Therefore  \eqref{9ea3.513} becomes
\begin{align*} 
I_0^{(b,1)}\leq \frac{C}{N^2} \E\Big(&\sum_{\alpha}1(T_{ \alpha} \leq t, B^\alpha\neq \Delta) \sum_{i=0}^{\vert \alpha\vert -1}1(T_{\alpha\vert i}>T_\alpha-\tau_N)  1(B^{\alpha\vert i}+W^{\alpha\vert i}\in K_0^N)\Big),
\end{align*}
which gives exactly the same expression as in \eqref{9ea3.444} for $I_0^{(a)}$ up to some constant. So we conclude $I_0^{(b,1)}\to 0$.\\

\no {\bf  Case (b,2).}  Recall $I_0^{(b,2)}$ is the contribution to \eqref{9ea5.666} from the sum of  $i\geq \vert \alpha\vert$, in which case $T_{\beta\vert i}\geq T_\alpha$ and $1(T_{\beta\vert i}>T_\beta-\tau_N)$ will be absorbed into $1(T_\alpha>T_\beta-\tau_N)$.   It follows that
 \begin{align}  \label{9ed1.661}
I_0^{(b,2)}\leq  \frac{1}{N\psi(N)} \E\Big(&\sum_{\alpha}\sum_{\beta\geq\alpha} \sum_{\gamma\geq \alpha}1(T_{\alpha} \leq t)  1(T_{\pi\beta}<T_{\alpha}+\tau_N) 1(T_{\pi\gamma}<T_{\alpha}+\tau_N) \\
&  1(B^\beta-B^\gamma\in \cN_N)\sum_{i=\vert \alpha\vert }^{\vert \beta\vert -1} 1(B^{\beta\vert i}+W^{\beta\vert i} \in K_0^N) \Big),\nn
\end{align}
where we have used  $T_{\pi\beta}<T_\beta$. Consider $\alpha_0=i_0$ for $1\leq i_0\leq NX_0^0(1)$. Then
\begin{align} \label{9ed1.66}
I_0^{(b,2)}\leq &\frac{C}{N\psi(N)} \sum_{i_0=1}^{NX_0^0(1)} \sum_{k=0}^\infty \sum_{\substack{\alpha\geq i_0,\\ \vert \alpha\vert =k}} \sum_{l=0}^\infty   \sum_{m=0}^\infty\sum_{\substack{\beta\geq \alpha,\\ \vert \beta\vert =k+l}}\sum_{\substack{ \gamma\geq \alpha,\\ \vert \gamma\vert =k+m}}      \nn\\
&  \sum_{i=k}^{k+l-1}     \E\Big(1(T_{\alpha} \leq t) 1(T_{\pi\beta}<T_{\alpha}+\tau_N) 1(T_{\pi\gamma}<T_{\alpha}+\tau_N)  \nn\\
& 1(B^\beta-B^\gamma\in \cN_N)  1(B^{\beta\vert i}+W^{\beta\vert i}\in K_0^N)  \Big),
\end{align}
  Note that on the event $\{B^\beta, B^\gamma \neq \Delta\}$,  the spatial events $\{B^\beta-B^\gamma\in \cN_N\}$ and $\{B^{\beta\vert i}+W^{\beta\vert i}\in K_0^N\}$ are independent of the branching events on $T_{\alpha}, T_{\pi\beta}, T_{\pi\gamma}$. Therefore the expectation above is at most
 \begin{align}  \label{9ed1.53}
&\E\Big(1(T_{\alpha} \leq t) 1(T_{\pi\beta}<T_{\alpha}+\tau_N)    1(T_{\pi\gamma}<T_{\alpha}+\tau_N)\Big)  \nn\\
&(\frac{N+\theta}{2N+\theta})^{k+l+m}    \E^{\beta, \gamma}\Big(1(\sqrt{N}(\hat{B}^\beta- \hat{B}^\gamma)\in  [-2,2]^d) 1(B^{\beta\vert i}+W^{\beta\vert i}\in K_0^N) \Big),
\end{align}
where $\E^{\beta, \gamma}$ is the expectation conditional on $\{B^\beta, B^\gamma\neq \Delta\}$. We also have used \eqref{9ea2.82} to bound $1(B^\beta-B^\gamma\in \cN_N)$. 
The first expectation  above is equal to
 \begin{align}  \label{9ed8.53}
& \P(\Gamma_{k+1}\leq (2N+\theta)t) \cdot \P(\Gamma_{l-1} \leq (2N+\theta)\tau_N)  \cdot \P(\Gamma_{m-1} \leq (2N+\theta)\tau_N).
\end{align}
Turning to the second expectation in \eqref{9ed1.53}, recall $\hat{B}^\beta$, $\hat{B}^\gamma$ from \eqref{9ea1.71}. By setting $$\hat{W}^{\gamma\vert j}:=W^{\gamma\vert j} 1_{\{e_{\beta\vert j}=\beta_{j+1}\}},$$
we get that the term $\sum_{j=k+1}^{i} \hat{W}^{\gamma\vert j}$ appearing in $\hat{B}^\beta$ is measurable with respect to $\cH_{\beta\vert (i+1)}$. Hence we may condition on $\cH_{\beta\vert (i+1)}$ to see that
 \begin{align}  \label{e4.33}
& \P^{\beta,\gamma}\Big(\sqrt{N}(\hat{B}^\beta- \hat{B}^\gamma)\in  [-2,2]^d \Big\vert \cH_{\beta\vert (i+1)}\Big)\nn\\
=  & \P\Big(\sqrt{N} \sum_{j=i+1}^{k+l-1}   \hat{W}^{\beta\vert j} -\sqrt{N} \sum_{j=k+1}^{k+m-1}    \hat{W}^{\gamma\vert j} \in [-2,2]^d-\sqrt{N} \sum_{j=k+1}^{i}   \hat{W}^{\beta\vert j} \Big)\nn\\
\leq& \frac{C}{(k+l-i+m+1)^{d/2}},
\end{align}
where the last inequality follows from Lemma \ref{9l4.2}.
Since $B^{\beta\vert i}, W^{\beta\vert i}\in \cH_{\beta\vert (i+1)}$, we may condition on $\cH_{\beta\vert (i+1)}$ and use the above to get that
 \begin{align}  \label{9ed1.51}
&  \E^{\beta, \gamma}\Big(1(B^\beta-B^\gamma\in \cN_N)  1(B^{\beta\vert i}+W^{\beta\vert i}\in K_0^N) \Big)\nn\\
&\leq   \frac{C}{(k+l-i+m+1)^{d/2}}  \P^{\beta, \gamma}(B^{\beta\vert i}+W^{\beta\vert i}\in K_0^N)\nn\\
&=   \frac{C}{(k+l-i+m+1)^{d/2}}  \P(x_{i_0}+V_{i}^N+W^{N}\in K_0^N).
\end{align}
Combining \eqref{9ed8.53} and \eqref{9ed1.51}, we have that \eqref{9ed1.53}  is at most  
 \begin{align*}  
&(\frac{N+\theta}{2N+\theta})^{k+l+m}     \frac{C}{(k+l-i+m+1)^{d/2}}  \P(x_{i_0}+V_{i}^N+W^{N}\in K_0^N)\\
& \P(\Gamma_{k+1}\leq (2N+\theta)t) \cdot \P(\Gamma_{l-1} \leq (2N+\theta)\tau_N)  \cdot \P(\Gamma_{m-1} \leq (2N+\theta)\tau_N) .
\end{align*}
Apply the above in \eqref{9ed1.66} to get
\begin{align} \label{e2.34}
I_0^{(b,2)}\leq &\frac{1}{N\psi(N)} \sum_{i_0=1}^{NX_0^0(1)} \sum_{k=0}^\infty    \sum_{l=0}^\infty   \sum_{m=0}^\infty     \sum_{i=k}^{k+l-1} (\frac{2N+2\theta}{2N+\theta})^{k+l+m}   \nn\\
& \frac{1}{(k+l-i+m+1)^{d/2}} \P(x_{i_0}+V_{i}^N+W^{N}\in K_0^N) \nn\\
& \P(\Gamma_{k+1}\leq (2N+\theta)t) \cdot \P(\Gamma_{l-1} \leq (2N+\theta)\tau_N)  \cdot \P(\Gamma_{m-1} \leq (2N+\theta)\tau_N) .
\end{align}
 By Lemma \ref{9l4.0},  the sum of $m$ gives
\begin{align*}
 \sum_{m=0}^\infty (\frac{2N+2\theta}{2N+\theta})^{m} \Pi((2N+\theta)\tau_N, m-1)  \frac{1}{(k+l-i+m+1)^{d/2}}  \leq \frac{C }{(1+l+k-i)^{d/2-1}}.
\end{align*}
Then the sum of $l, i$ gives
\begin{align*}
\sum_{l=0}^\infty  &   \sum_{i=k}^{k+l-1} (\frac{2N+2\theta}{2N+\theta})^{l}  \frac{1}{(1+l+k-i)^{d/2-1}}   \nn\\
& \P(x_{i_0}+V_{i}^N+W^{N}\in K_0^N) \P(\Gamma_{l-1} \leq (2N+\theta)\tau_N)\nn\\
=\sum_{l=0}^\infty &    \sum_{j=0}^{l-1} (\frac{2N+2\theta}{2N+\theta})^{l}  \frac{1}{(1+l-j)^{d/2-1}}   \nn\\
& \P(x_{i_0}+V_{j+k}^N+W^{N}\in K_0^N) \P(\Gamma_{l-1} \leq (2N+\theta)\tau_N).
\end{align*}
where the equality follows by letting $j=i-k$. 
Bounding $\P(\Gamma_{l-1}\leq (2N+\theta)\tau_N)$ by $\P(\Gamma_{j}\leq (2N+\theta)\tau_N)\P(\Gamma_{l-1}-\Gamma_{j}\leq (2N+\theta)\tau_N)$ and exchanging the sum of $l,j$, we get the above is at most
\begin{align*}
&\sum_{j=0}^\infty   (\frac{2N+2\theta}{2N+\theta})^{j}   \P(x_{i_0}+V_{j+k}^N+W^{N}\in K_0^N)\P(\Gamma_{j}\leq (2N+\theta)\tau_N) \nn\\
& \sum_{l=j+1}^{\infty} (\frac{2N+2\theta}{2N+\theta})^{l-j}  \frac{1}{(l-j+1)^{d/2-1}} \P(\Gamma_{l-j-1} \leq (2N+\theta)\tau_N).
\end{align*}
By Lemma \ref{9l8.7}, the sum of $l$ is at most $CI(N)$. Returning to \eqref{e2.34}, we are left with
\begin{align} \label{9ed1.20}
I_0^{(b,2)}\leq &\frac{1}{N^2} \sum_{i_0=1}^{NX_0^0(1)} \sum_{k=0}^\infty \sum_{j=0}^\infty   (\frac{2N+2\theta}{2N+\theta})^{j+k}  \P(\Gamma_{k+1}\leq (2N+\theta)t)\nn\\
&\P(\Gamma_{j}\leq (2N+\theta)\tau_N) \P(x_{i_0}+V_{j+k}^N+W^{N}\in K_0^N).
\end{align}
Notice that
\begin{align*} 
&\P(\Gamma_{k+1}\leq (2N+\theta)t)\P(\Gamma_{j}\leq (2N+\theta)\tau_N)\nn\\
=&\P(\Gamma_{k+1}\leq (2N+\theta)t, \Gamma_{k+j+1}-\Gamma_{k+1}\leq (2N+\theta)\tau_N)\leq \P(\Gamma_{k+j+1}\leq 2(2N+\theta)t).
\end{align*}
Use the above to bound \eqref{9ed1.20} and then let $n=j+k$ to see that
\begin{align*}  
I_0^{(b,2)}\leq &\frac{1}{N^2} \sum_{i_0=1}^{NX_0^0(1)} \sum_{n=0}^\infty (n+1)  (\frac{2N+2\theta}{2N+\theta})^{n}  \nn\\
& \P(\Gamma_{n+1}\leq 2(2N+\theta)t) \P(x_{i_0}+V_{n}^N+W^{N}\in K_0^N).
\end{align*}
By Lemma \ref{9l4.3}(ii), the sum of $n>2ANt$ gives at most $C2^{-2Nt}\leq C$. For $n\leq 2ANt$, we bound $n+1$ by $3ANt$ to see
\begin{align*}  
I_0^{(b,2)}\leq &\frac{CX_0^0(1)}{N}  +\frac{C}{N} \sum_{i_0=1}^{NX_0^0(1)} \sum_{n=0}^{2ANt}    \P(\Gamma_{n+1}\leq 2(2N+\theta)t)   \P(x_{i_0}+V_{n}^N+W^{N}\in K_0^N).
\end{align*}
By apply \eqref{9e3.490} with $t=2t$, we conclude $I_0^{(b,2)}\to 0$. The proof is complete.

\subsubsection{Convergence of $I_1$}
 There are two cases for $\delta$ in the summation of $I_1$: (a) $\delta_0 \neq \beta_0$; (b) $\delta_0= \beta_0$. Denote by $I_1^{(a)}$ (resp. $I_1^{(b)}$) the contribution to $I_1$ from case (a) (resp. case (b)). It suffices to prove that $I_1^{(a)} \to 0$ and $I_1^{(b)} \to 0$.\\

\no {\bf  Case (a).} Let $\beta_0=i_0\neq j_0=\delta_0$ for some $1\leq i_0\neq j_0\leq NX_0^0(1)$. There are at most $[NX_0^0(1)]^2$ such pairs,   so the contribution to $I_1$ from case (a) is bounded by
\begin{align}  \label{9e8.05}
I_1^{(a)}\leq &\frac{1}{N^2}[NX_0^0(1)]^2 \sum_{\beta\geq i_0}\sum_{\delta\geq j_0} \E\Big( 1(T_{\beta}\leq t, B^{\beta}\neq \Delta) \sum_{i=0}^{\vert \beta\vert -1} 1(T_{\beta\vert i}>T_{\beta}-\tau_N)  \nn\\
&1(T_{\pi\delta} <T_{\beta\vert i}) 1(B^{\beta\vert i}+W^{\beta\vert i}=B^\delta) \Big)\nn\\
\leq & X_0^0(1)^2 \sum_{l=0}^\infty \sum_{\beta\geq i_0, \vert \beta\vert =l} \sum_{n=0}^\infty\sum_{\delta\geq j_0, \vert \delta\vert =n} \sum_{i=0}^{l-1} \E\Big(1(T_{\beta}\leq t, B^{\beta}\neq \Delta)   \nn\\
& 1(T_{\beta\vert i}>T_{\beta}-\tau_N)  1(T_{\pi\delta} <t)  1(B^{\beta\vert i}+W^{\beta\vert i}=B^\delta) \Big).
\end{align}
where in the last inequality we have replaced $T_{\pi\delta} <T_{\beta\vert i}$ by $T_{\pi\delta} <T_\beta\leq t$. The last expectation above is at most
\begin{align} \label{9ea1.13}
 (\frac{N+\theta}{2N+\theta})^{l+n}  \E\Big(& 1(\Gamma_{l+1}\leq (2N+\theta)t)    1(\Gamma_{i+1}>\Gamma_{l+1}-(2N+\theta)\tau_N) \Big)  \nn\\
&\times \Pi((2N+\theta)t,n) \frac{1}{\psi(N)} \frac{C}{(i+n+1)^{d/2}},
\end{align}
where we have used Lemma \ref{9l4.2} to obtain
\begin{align*}
\P(B^{\beta\vert i}+W^{\beta\vert i}=B^\delta)=\frac{1}{\psi(N)} \P(B^{\beta\vert i}-B^\delta \in \cN_N) \leq \frac{1}{\psi(N)} \frac{C}{(i+n+1)^{d/2}}.
\end{align*} 
Apply \eqref{9ea1.13} in \eqref{9e8.05} to get
\begin{align}  \label{9ea1.14}
 I_1^{(a)}\leq &\frac{C}{\psi(N)}X_0^0(1)^2\sum_{l=0}^\infty \sum_{n=0}^\infty  \sum_{i=0}^{l-1}    (\frac{2N+2\theta}{2N+\theta})^{n+l}    \Pi((2N+\theta)t,n) \frac{1}{(i+n+1)^{d/2}}     \nn\\
&\E\Big( 1(\Gamma_{l+1}\leq (2N+\theta)t)\cdot 1(\Gamma_{l+1}<\Gamma_{i+1}+(2N+\theta)\tau_N)   \Big).
\end{align}
The sum of $n$ gives
\begin{align}  \label{9ea1.30} 
 & \sum_{n=0}^\infty    (\frac{2N+2\theta}{2N+\theta})^{n} \Pi((2N+\theta)t,n) \frac{1}{(i+n+1)^{d/2}}\leq Ce^{\theta t} \frac{1}{(i+1)^{d/2-1}},
\end{align}
where the   inequality is by  Lemma \ref{9l4.0}, thus giving
 \begin{align}  \label{9ea3.05} 
I_1^{(a)}\leq &\frac{C}{\psi(N)}X_0^0(1)^2\sum_{l=0}^\infty    (\frac{2N+2\theta}{2N+\theta})^{l}  \E\Big( 1(\Gamma_{l+1}\leq (2N+\theta)t)     \nn\\
&\quad \sum_{i=0}^{l-1}  1(\Gamma_{l+1}<\Gamma_{i+1}+(2N+\theta)\tau_N)    \frac{1}{(i+1)^{d/2-1}} \Big).
\end{align}
By applying Lemma \ref{9l4.3} (ii), we get the sum of $l>ANt$ is at most
 \begin{align*}   
 \sum_{l>ANt}    (\frac{2N+2\theta}{2N+\theta})^{l}  \E\Big( 1(\Gamma_{l+1}\leq (2N+\theta)t)   \times l\Big)\leq C_1 2^{-Nt}.
\end{align*}
The remaining sum of $l\leq ANt$ is bounded by
\begin{align*}  
&\sum_{l=0}^{ANt}   e^{A\theta t}   \E\Big(  \sum_{i=0}^{l-1} 1(\Gamma_{l+1}-\Gamma_{i+1}<(2N+\theta)\tau_N)  \frac{1}{(i+1)^{d/2-1}} \Big).
\end{align*} 
Use Fubini's theorem to see that the above is at most
\begin{align}  \label{9ea3.04}
&e^{A\theta t}    \sum_{i=0}^{ANt}  \frac{1}{(i+1)^{d/2-1}} \E\Big(   \sum_{l=i+1}^{ANt}     1(\Gamma_{l-i}<(2N+\theta)\tau_N)  \Big)\nn\\
&\leq e^{A\theta t}     I(ANt) \E\Big(   \sum_{l=1}^{ANt}     1(\Gamma_{l}<(2N+\theta)\tau_N)  \Big).
\end{align}
Again we apply Lemma \ref{9l4.3} (ii) to see that
\begin{align*}  
\E\Big(   \sum_{l=1}^{ANt}     1(\Gamma_{l}<(2N+\theta)\tau_N)  \Big)& \leq  \E\Big(   \sum_{l>AN\tau_N}     1(\Gamma_{l}<(2N+\theta)\tau_N)  \Big)+\E\Big(   \sum_{l\leq AN\tau_N}    1(\Gamma_{l}<(2N+\theta)\tau_N)  \Big)\\
&\leq C2^{-N\tau_N}+2 AN\tau_N \leq CN\tau_N.
\end{align*}  
Hence \eqref{9ea3.04} is bounded by $Ce^{A\theta t}   I(ANt) N\tau_N\leq CN\tau_N I(N)$. We conclude that \eqref{9ea3.05} becomes
\begin{align*}  
 I_1^{(a)}\leq\frac{C}{\psi(N)}X_0^0(1)^2  (C_1 2^{-Nt}+CN\tau_N I(N))\leq C[X_0^0(1)]^2\tau_N \to 0.
\end{align*}

\no {\bf  Case (b).} Let $\beta_0=i_0$ for some $1\leq i_0\leq NX_0^0(1)$. By translation invariance, the contribution to $I_1$ from the case (b) is given by
\begin{align*}  
I_1^{(b)}=\frac{1}{N^2}NX_0^0(1) \sum_{\beta\geq i_0}\sum_{\delta\geq i_0}\E\Big(& 1(T_{ \beta}\leq t, B^{\beta}\neq \Delta)   \sum_{i=0}^{\vert \beta\vert -1}1(T_{\beta\vert i}>T_\beta-\tau_N)  \nn\\
&1(T_{\pi\delta} <T_{\beta\vert i}, B^{\beta\vert i}+W^{\beta\vert i}=B^\delta)   \Big).
\end{align*}
Let $\vert \beta\vert =l$ for some $l\geq 1$ (the case for $|\beta|=0$ is trivial by $i\leq |\beta|-1$). Since $T_{\pi\delta} <T_{\beta\vert i}$ implies that $\beta\wedge \delta\leq  \beta\vert i$, we may let $\beta\wedge \delta=\beta\vert j$ for some $j\leq i$ and $\vert \delta\vert =j+n$ for some $n\geq 0$. Then 
\begin{align}  \label{9ea1.20}
I_1^{(b)}&\leq \frac{X_0^0(1)}{N} \sum_{l=1}^\infty \sum_{\beta\geq i_0, \vert \beta\vert =l}\sum_{i=0}^{l-1} \sum_{j=0}^{i} \sum_{n=0}^\infty \sum_{\substack{ \delta\geq \beta\vert j,\\ \vert \delta\vert =j+n}}\E\Big(  1(T_{ \beta}\leq t, B^{\beta}\neq \Delta) \nn\\
& 1(T_{\beta\vert i}>T_\beta-\tau_N)  1(T_{\pi\delta}-T_{\beta\vert j}<t)1(B^{\beta\vert i}+W^{\beta\vert i}=B^\delta)   \Big),
\end{align}
where we have replaced $\{T_{\pi\delta} <T_{\beta\vert i}\}$ by $\{T_{\pi\delta}-T_{\beta\vert j}< T_{\beta\vert i}<T_\beta\leq t\}$. By conditioning on $\{B^\beta,B^\delta \neq \Delta\}$, we get
\begin{align*}  
\P(B^{\beta\vert i}+W^{\beta\vert i}=B^\delta)&=\frac{1}{\psi(N)} \P((B^{\beta\vert i}-B^{\beta\vert j})-(B^\delta-B^{\beta\vert j}) \in \cN_N) \nn\\
&\leq \frac{1}{\psi(N)} \frac{C}{(i-j+n+1)^{d/2}},
\end{align*}
where the inequality follows by Lemma \ref{9l4.2}.
So the expectation in \eqref{9ea1.20} can be bounded by
\begin{align}  \label{9ea1.21}
&\frac{C}{\psi(N)} \frac{1}{(i-j+n+1)^{d/2}} (\frac{N+\theta}{2N+\theta})^{l+n}\E\Big( 1(\Gamma_{l+1}\leq (2N+\theta)t) \nn\\
&   1(\Gamma_{i+1}>\Gamma_{l+1}-(2N+\theta)\tau_N)\Big) \cdot \Pi((2N+\theta)t, n-1)  ,
\end{align}
where we have used that  $T_{\pi\delta}-T_{\beta\vert j}$ is independent of $\cH_\beta$ for $n\geq 2$ (trivial for $n=0,1$).  
Apply \eqref{9ea1.21} in \eqref{9ea1.20} to get
\begin{align}  \label{9ea1.31}
I_1^{(b)}&\leq \frac{CX_0^0(1)}{N\psi(N)} \sum_{l=1}^\infty  \sum_{i=0}^{l-1} \sum_{j=0}^{i} \sum_{n=0}^\infty (\frac{2N+2\theta}{2N+\theta})^{l+n}   \Pi((2N+\theta)t, n-1) \nn\\
&\frac{1}{(i-j+n+1)^{d/2}}\E\Big( 1(\Gamma_{l+1}\leq (2N+\theta)t)   1(\Gamma_{l+1}<\Gamma_{i+1}+(2N+\theta)\tau_N) ) \Big).
\end{align}
By  Lemma \ref{9l4.0}, the sum of $n$ gives at most
\begin{align*}  
 & \sum_{n=0}^\infty    (\frac{2N+2\theta}{2N+\theta})^{n} \Pi((2N+\theta)t, n-1) \frac{1}{(i-j+n+1)^{d/2}}\leq  \frac{C }{(i-j+1)^{d/2-1}}.
\end{align*}
Then the sum of $j$ equals
\begin{align*}  
 & \sum_{j=0}^i   C  \frac{1}{(i-j+1)^{d/2-1}}\leq C  I(i)\leq C  I(l).
 \end{align*}
 We conclude from the above that \eqref{9ea1.31} becomes
\begin{align}   \label{9ea2.84}
I_1^{(b)}\leq \frac{CX_0^0(1)}{N\psi(N)} &\sum_{l=1}^\infty  (\frac{2N+2\theta}{2N+\theta})^{l} I(l) \E\Big(  1(\Gamma_{l+1}\leq (2N+\theta)t)     \nn\\
&\cdot \sum_{i=0}^{l-1}  1(\Gamma_{l+1}-\Gamma_{i+1}<(2N+\theta)\tau_N)   \Big).
\end{align}

\begin{lemma}\label{9l8.4}
For $m=0,1$ or $2$, we have
\begin{align} \label{e2.84}
\sum_{l=0}^\infty & (\frac{2N+2\theta}{2N+\theta})^{l}  \cdot  I(l)^m \E\Big(  1(\Gamma_{l}\leq (2N+\theta)t)   \nn\\
&\cdot \sum_{i=0}^{l}  1(\Gamma_{l}-\Gamma_{i+1}<(2N+\theta)\tau_N)   \Big)\leq CN^2\tau_N I(N)^m.
\end{align}
\end{lemma}
\begin{proof}
The sum of $l>ANt$ is at most $C2^{-Nt}$ by Lemma \ref{9l4.3}. Next, the sum of $l\leq ANt$ is bounded by 
\begin{align*}  
&\sum_{l=0}^{ANt}      e^{A\theta t}  I(ANt)^m    \E\Big( \sum_{i=0}^{l}  1(\Gamma_{l}<\Gamma_{i+1}+(2N+\theta)\tau_N)     \Big)\nn\\
&\leq CI(N)^m \sum_{l=0}^{ANt}      \E\Big( \sum_{k=-1}^l  1(\Gamma_{k}<(2N+\theta)\tau_N)     \Big)\nn\\
&\leq C I(N)^m   ANt  \times \E\Big(\sum_{k=-1}^{ANt}  1(\Gamma_{k}<(2N+\theta)\tau_N)      \Big).
\end{align*}
Again by Lemma \ref{9l4.3}, the sum of $k>AN\tau_N$ is at most $C2^{-N\tau_N}$. The sum of $k\leq AN\tau_N$ is at most $CN\tau_N$. So we may bound the above by $C I(N)^m   ANt  \cdot (C2^{-N\tau_N}+CN\tau_N) \leq  CN^2\tau_N I(N)^m$.
\end{proof}
 One can easily check that \eqref{e2.84} with $m=1$ implies an upper bound for the sum in \eqref{9ea2.84}, thus giving
\begin{align*}  
I_1^{(b)}&\leq \frac{CX_0^0(1)}{N\psi(N)} CN^2\tau_N I(N)\leq  {CX_0^0(1)} \tau_N \to 0,
\end{align*}
as required.
 
\subsubsection{Convergence of $I_2$}
 
Recall $I(\beta)$ from \eqref{9ea3.08} to see that
\begin{align}  \label{9e1.32}
I_2=\frac{1}{N\psi(N)} \E\Big(&\sum_{\beta} 1(T_{ \beta} \leq t)\sum_{\gamma}1(T_{\pi\gamma}<T_\beta) 1(B^\beta-B^\gamma\in \cN_N) \nn\\
&1(T_{\gamma\wedge \beta}>T_\beta-\tau_N)  \sum_{i=0}^{ \vert \beta\vert-1} 1(T_{\beta\vert i}>T_\beta-\tau_N)  \nn\\
&\sum_{\delta} 1(T_{\pi\delta} <T_{\beta\vert i}, B^{\beta\vert i}+W^{\beta\vert i}=B^\delta) \Big).
\end{align}
Since $T_{\gamma\wedge \beta}>T_\beta-\tau_N$ implies that $\gamma_0=\beta_0$, we may let $\alpha=\beta\wedge \gamma$ and use $T_{\alpha}\leq T_\beta< T_{\alpha}+\tau_N$ to get  
 \begin{align} \label{9ec6.89}
I_2\leq \frac{1}{N\psi(N)} \E\Big(&\sum_{\alpha}\sum_{\beta\geq\alpha} \sum_{\gamma\geq \alpha}1(T_{\alpha} \leq t) \sum_{i=0}^{\vert \beta\vert -1}1(T_{\beta\vert i}>T_\beta-\tau_N)    \nn\\
& 1(T_\beta<T_{\alpha}+\tau_N) 1(T_{\pi\gamma}<T_{\alpha}+\tau_N) 1(B^\beta-B^\gamma\in \cN_N)\nn\\
&\sum_{\delta} 1(T_{\pi\delta} <T_{\beta\vert i}, B^{\beta\vert i}+W^{\beta\vert i}=B^\delta) \Big).
\end{align}
Denote by $I_3$ (resp. $I_4$) the sum of $i\leq \vert \alpha\vert -1$ (resp. $ \vert \alpha\vert \leq  i \leq \vert \beta\vert -1$) in the above expression. It suffices to prove that $I_3 \to 0$ and $I_4 \to 0$.\\

\no {\bf  Case for $I_3$.} For $I_3$, since $i\leq \vert \alpha\vert -1$, we get  $\beta\vert i=\alpha\vert i$ and so $T_{\beta\vert i}, B^{\beta\vert i}, W^{\beta\vert i}$ are all measurable with respect to $\cH_{\alpha}$. 
Use Fubini's theorem and $T_\beta\geq T_{\alpha}$ to get
 \begin{align} \label{9ec6.57}
I_3\leq \frac{1}{N\psi(N)} \E\Big(&\sum_{\alpha}1(T_{ \alpha} \leq t) \sum_{i=0}^{\vert \alpha\vert -1}1(T_{\alpha\vert i}>T_\alpha-\tau_N)    \nn\\
&\sum_{\delta} 1(T_{\pi\delta} <T_{\alpha\vert i}, B^{\alpha\vert i}+W^{\alpha\vert i}=B^\delta)\nn \\
& \sum_{\beta\geq\alpha} \sum_{\gamma\geq \alpha}1(T_{\beta}-T_{\alpha}<\tau_N) 1(T_{\pi\gamma}-T_{\alpha}<\tau_N) 1({B}^\beta-{B}^\gamma\in \cN_N)\Big).
\end{align}
Since $T_{\pi\delta} <T_{\alpha\vert i}\leq T_{\pi \alpha}$, we must have $\delta$ branch off the family tree of $\alpha, \beta, \gamma$ before $\alpha$. Hence we may condition on $\cH_\alpha \vee \cH_\delta$ and use \eqref{9ea3.512} to bound the above sum of $\beta,\gamma$ so that \eqref{9ec6.57} is at most
 \begin{align*}  
I_3\leq &\frac{C}{N^2}\E\Big( \sum_{\alpha}1(T_{ \alpha} \leq t, B^\alpha\neq \Delta) \sum_{i=0}^{\vert \alpha\vert -1}1(T_{\beta\vert i}>T_\alpha-\tau_N)    \nn\\
&\quad\quad\quad\sum_{\delta} 1(T_{\pi\delta} <T_{\alpha\vert i}, B^{\alpha\vert i}+W^{\alpha\vert i}=B^\delta) \Big).
\end{align*}
The above gives the same expression as in \eqref{9e8.03} for $I_1$ up to some constant. It follows that $I_3\to 0$.\\

\no {\bf  Case for $I_4$.}   Recall $I_4$ is the contribution to \eqref{9ec6.89} from the sum of  $i\geq \vert \alpha\vert$, in which case $T_{\beta\vert i}\geq T_\alpha$ and $1(T_{\beta\vert i}>T_\beta-\tau_N)$ will be absorbed into $1(T_\alpha>T_\beta-\tau_N)$.   It follows that
 \begin{align}  \label{9ea2.81}
I_4\leq  \frac{1}{N\psi(N)} \E\Big(&\sum_{\alpha}\sum_{\beta\geq\alpha} \sum_{\gamma\geq \alpha}1(T_{\alpha} \leq t) \sum_{i=\vert \alpha\vert }^{\vert \beta\vert -1}  1(T_{\pi\beta}<T_{\alpha}+\tau_N)  \nn\\
& 1(T_{\pi\gamma}<T_{\alpha}+\tau_N) 1(B^\beta-B^\gamma\in \cN_N)\nn\\
&\sum_{\delta} 1(T_{\pi\delta} <T_{\beta\vert i}, B^{\beta\vert i}+W^{\beta\vert i}=B^\delta) \Big),
\end{align}
where we have bounded $1(T_{\beta}<T_{\alpha}+\tau_N) $ by $1(T_{\pi\beta}<T_{\alpha}+\tau_N)$.

Separate the sum of $\delta$ into two parts: (a) $\delta_0 \neq \alpha_0$; (b) $\delta_0= \alpha_0$. Denote by $I_4^{(a)}$ (resp. $I_4^{(b)}$) the contribution from case (a) (resp. case (b)). It suffices to prove $I_4^{(a)} \to 0$ and $I_4^{(b)} \to 0$.\\

\no {\bf  Case (a).} Let $\alpha_0=i_0\neq j_0=\delta_0$ for some $1\leq i_0\neq j_0\leq NX_0^0(1)$. There are at most $[NX_0^0(1)]^2$ such pairs. It follows that  
\begin{align} \label{9ea1.66}
I_4^{(a)}\leq &\frac{C}{N\psi(N)}[NX_0^0(1)]^2\sum_{k=0}^\infty \sum_{\substack{\alpha\geq i_0,\\ \vert \alpha\vert =k}} \sum_{l=0}^\infty   \sum_{m=0}^\infty\sum_{\substack{\beta\geq \alpha,\\ \vert \beta\vert =k+l}}\sum_{\substack{ \gamma\geq \alpha,\\ \vert \gamma\vert =k+m}}  \sum_{n=0}^\infty \sum_{\substack{\delta\geq j_0,\\ \vert \delta\vert =n} }   \nn\\
&  \sum_{i=k}^{k+l-1}     \E\Big(1(T_{ \alpha} \leq t) 1(T_{\pi\beta}<T_{\alpha}+\tau_N) 1(T_{\pi\gamma}<T_{\alpha}+\tau_N)  \nn\\
& 1(B^\beta-B^\gamma\in \cN_N)  1(T_{\pi\delta} <2t)1(B^{\beta\vert i}+W^{\beta\vert i}=B^\delta)  \Big),
\end{align}
  where we  also replace  $T_{\pi\delta}<T_{\beta\vert i}$ by $T_{\pi\delta}<T_{\beta}<T_\alpha+\tau_N\leq  2t$. Note that on the event $\{B^\beta, B^\gamma, B^{\delta}\neq \Delta\}$, the spatial events $\{B^\beta-B^\gamma\in \cN_N\}$ and $\{B^{\beta\vert i}+W^{\beta\vert i}=B^\delta\}$ are independent of the branching events on $T_{\alpha}, T_\beta$, etc. Therefore the expectation above is at most
 \begin{align}  \label{9ea1.53}
&\E\Big(1(T_{\alpha}\leq t)   1(T_{\pi\gamma}<T_{\alpha}+\tau_N) 1(T_{\pi\beta}<T_{\alpha}+\tau_N) 1(T_{\pi\delta} <2t)\Big)  \nn\\
&(\frac{N+\theta}{2N+\theta})^{k+l+m+n}  \sum_{i=k}^{k+l-1} \E^{\beta, \gamma,{\delta}}\Big(1(B^\beta-B^\gamma\in \cN_N)  1(B^{\beta\vert i}+W^{\beta\vert i}=B^\delta) \Big),
\end{align}
where $\E^{\beta, \gamma,{\delta}}$ is the expectation conditional on $\{B^\beta, B^\gamma, B^{\delta}\neq \Delta\}$. 
The first expectation  above is equal to
 \begin{align}  \label{9ec8.53} 
& \Pi((2N+\theta)t, k+1) \cdot \Pi((2N+\theta)\tau_N, m-1)  \cdot \Pi((2N+\theta)\tau_N, l-1) \cdot \Pi(2(2N+\theta)t, n).
\end{align}
Next, use \eqref{9ea2.82} to bound $1(B^\beta-B^\gamma\in \cN_N)$ by $1(\sqrt{N}(\hat{B}^\beta- \hat{B}^\gamma)\in  [-2,2]^d)$. Similar to the derivation of \eqref{e4.33}, we may condition on $\cH_{\beta\vert (i+1)}$ to see that
 \begin{align}  \label{9ea1.84}
& \P^{\beta, \gamma,{\delta}}\Big(\sqrt{N}(\hat{B}^\beta- \hat{B}^\gamma)\in  [-2,2]^d \Big\vert \cH_{\beta\vert (i+1)}\Big) \leq  \frac{C}{(k+l-i+m+1)^{d/2}}.
\end{align}
Since $B^{\beta\vert i}, W^{\beta\vert i}\in \cH_{\beta\vert (i+1)}$, we may condition on $\cH_{\beta\vert (i+1)}\vee \cH_\delta$ and use the above to get
 \begin{align}  \label{9ea1.51}
&  \E^{\beta, \gamma,{\delta}}\Big(1(\sqrt{N}(\hat{B}^\beta- \hat{B}^\gamma)\in  [-2,2]^d)   1(B^{\beta\vert i}+W^{\beta\vert i}=B^\delta) \Big)\nn\\
&\leq   \frac{C}{(k+l-i+m+1)^{d/2}}  \P^{\beta, \gamma,{\delta}}(B^{\beta\vert i}+W^{\beta\vert i}=B^\delta).
\end{align}
Further bound the last probability above by
 \begin{align}    \label{9ea1.52}
 \P^{\beta, \gamma,{\delta}}(B^{\beta\vert i}+W^{\beta\vert i}=B^\delta)&=\frac{1}{\psi(N)} \P^{\beta, \gamma,{\delta}}(B^{\beta\vert i}-B^\delta\in \cN_N)\nn\\
 &\leq \frac{1}{\psi(N)} \frac{C}{(1+n+i)^{d/2}}\leq \frac{C}{\psi(N)} \frac{1}{(1+n+k)^{d/2}}.
\end{align}
Combining \eqref{9ec8.53}, \eqref{9ea1.51} and \eqref{9ea1.52}, we get  \eqref{9ea1.53}  is at most  
 \begin{align*}  
&(\frac{N+\theta}{2N+\theta})^{k+l+m+n}     \frac{C}{(k+l-i+m+1)^{d/2}}\frac{1}{\psi(N)} \frac{1}{(1+n+k)^{d/2}}\\
&\Pi((2N+\theta)t,k+1) \cdot \Pi((2N+\theta)\tau_N, m-1) \cdot \Pi((2N+\theta)\tau_N, l-1) \cdot \Pi(2(2N+\theta)t, n).\nn
\end{align*}
Apply the above in \eqref{9ea1.66} to obtain
\begin{align*} 
I_4^{(a)}\leq &\frac{C[X_0^0(1)]^2}{\psi(N)\psi_0(N)}\sum_{k=0}^\infty \Pi((2N+\theta)t,k+1)   \sum_{l=0}^\infty   \sum_{m=0}^\infty     \sum_{i=k}^{k+l-1} \sum_{n=0}^\infty (\frac{2N+2\theta}{2N+\theta})^{k+l+m+n}   \nn\\
& \frac{1}{(k+l-i+m+1)^{d/2}} \frac{1}{(1+n+k)^{d/2}} \nn\\
& \cdot \Pi((2N+\theta)\tau_N, m-1) \cdot \Pi((2N+\theta)\tau_N, l-1) \cdot \Pi(2(2N+\theta)t, n).
\end{align*}
By Lemma \ref{9l4.0}, the sum of $n$ gives
\begin{align*}
 \sum_{n=0}^\infty (\frac{2N+2\theta}{2N+\theta})^{n} \Pi(2(2N+\theta)t, n)  \frac{1}{(1+n+k)^{d/2}} \leq C \frac{1}{(1+k)^{d/2-1}}.
\end{align*}
Similarly, the sum of $m$ gives
\begin{align*}
 \sum_{m=0}^\infty (\frac{2N+2\theta}{2N+\theta})^{m} \Pi((2N+\theta)\tau_N, m-1)  \frac{1}{(k+l-i+m+1)^{d/2}}  \leq \frac{C }{(1+k+l-i)^{d/2-1}}.
\end{align*}
Next, the sum of $i$ equals
\begin{align*}
\sum_{i=k}^{k+l-1} \frac{1}{(1+k+l-i)^{d/2-1}}=\sum_{i=1}^{l} \frac{1}{(1+i)^{d/2-1}}\leq I(l).
\end{align*}
To this end, we get
\begin{align}  \label{9ec7.58}
I_4^{(a)}\leq &\frac{C[X_0^0(1)]^2}{\psi_0(N) \psi(N)}\sum_{k=0}^\infty (\frac{2N+2\theta}{2N+\theta})^{k}\Pi((2N+\theta)t, k+1)  \frac{1}{(1+k)^{d/2-1}} \nn\\
&\quad \sum_{l=0}^\infty   (\frac{2N+2\theta}{2N+\theta})^{l } \Pi((2N+\theta)\tau_N, l) I(l).
\end{align}
By Lemma \ref{9l8.7}, the sum of $k$ is at most $CI(N)$. For the sum of $l$, we write $\Pi((2N+\theta)\tau_N, l)=\P(\Gamma_{l}\leq (2N+\theta)\tau_N)$ to see
\begin{align*} 
&\sum_{l=0}^\infty   (\frac{2N+2\theta}{2N+\theta})^{l } \P(\Gamma_{l }\leq (2N+\theta)\tau_N) I(l)\nn\\
&\leq C2^{-N\tau_N}+\sum_{l=0}^{AN\tau_N}   e^{A\theta \tau_N}   I(l)\leq CN\tau_N I(N),
\end{align*}
where the first inequality is by Lemma \ref{9l4.3} (ii). Now we get \eqref{9ec7.58} is at most
\begin{align*} 
I_4^{(a)}\leq &\frac{C[X_0^0(1)]^2}{\psi_0(N) \psi(N)} CN\tau_N I(N)^2 \leq C [X_0^0(1)]^2\tau_N\to 0.
\end{align*}

\no {\bf  Case (b).} Let $\alpha_0=\delta_0=i_0$ for some $1\leq i_0\leq NX_0^0(1)$. The contribution to $I_4$ in \eqref{9ea2.81} from case (b) is bounded by 
\begin{align*} 
I_4^{(b)}\leq &\frac{C}{N\psi(N)}[NX_0^0(1)]\sum_{k=0}^\infty \sum_{\substack{\alpha\geq i_0,\\ \vert \alpha\vert =k}} \sum_{l=1}^\infty   \sum_{m=0}^\infty\sum_{\substack{\beta\geq \alpha,\\ \vert \beta\vert =k+l}}\sum_{\substack{ \gamma\geq \alpha,\\ \vert \gamma\vert =k+m}}   \sum_{i=k}^{k+l-1}  \sum_{\delta\geq i_0}    \nn\\
&\E\Big(1(T_{\alpha}\leq t)   1(T_{\pi\gamma}<T_{\alpha}+\tau_N) 1(T_{\pi\beta}<T_{\alpha}+\tau_N) \nn\\
&1(B^\beta-B^\gamma\in \cN_N)   1(T_{\pi\delta} <T_{\beta\vert i}) 1(B^{\beta\vert i}+W^{\beta\vert i}=B^\delta)  \Big).
\end{align*}
There are two cases for the generation when $\delta$ branches off the family tree of $\beta,\gamma$: 
\begin{align*} 
\text{(1)} &\quad  \tau(\{\beta,\gamma\}; \delta)=\vert \beta \wedge \delta\vert ;\nn\\
\text{(2)}& \quad \tau(\{\beta,\gamma\}; \delta)=\vert \gamma \wedge \delta\vert >\vert \beta \wedge \delta\vert .
\end{align*}
 Let $I_4^{(b,1)}$ (resp. $I_4^{(b,2)}$) denote the contribution to $I_4^{(b)}$ from case (1) (resp. case (2)), for which we refer the reader to Figure \ref{fig4}. \\
  \begin{figure}[ht]
  \begin{center}
    \includegraphics[width=1 \textwidth]{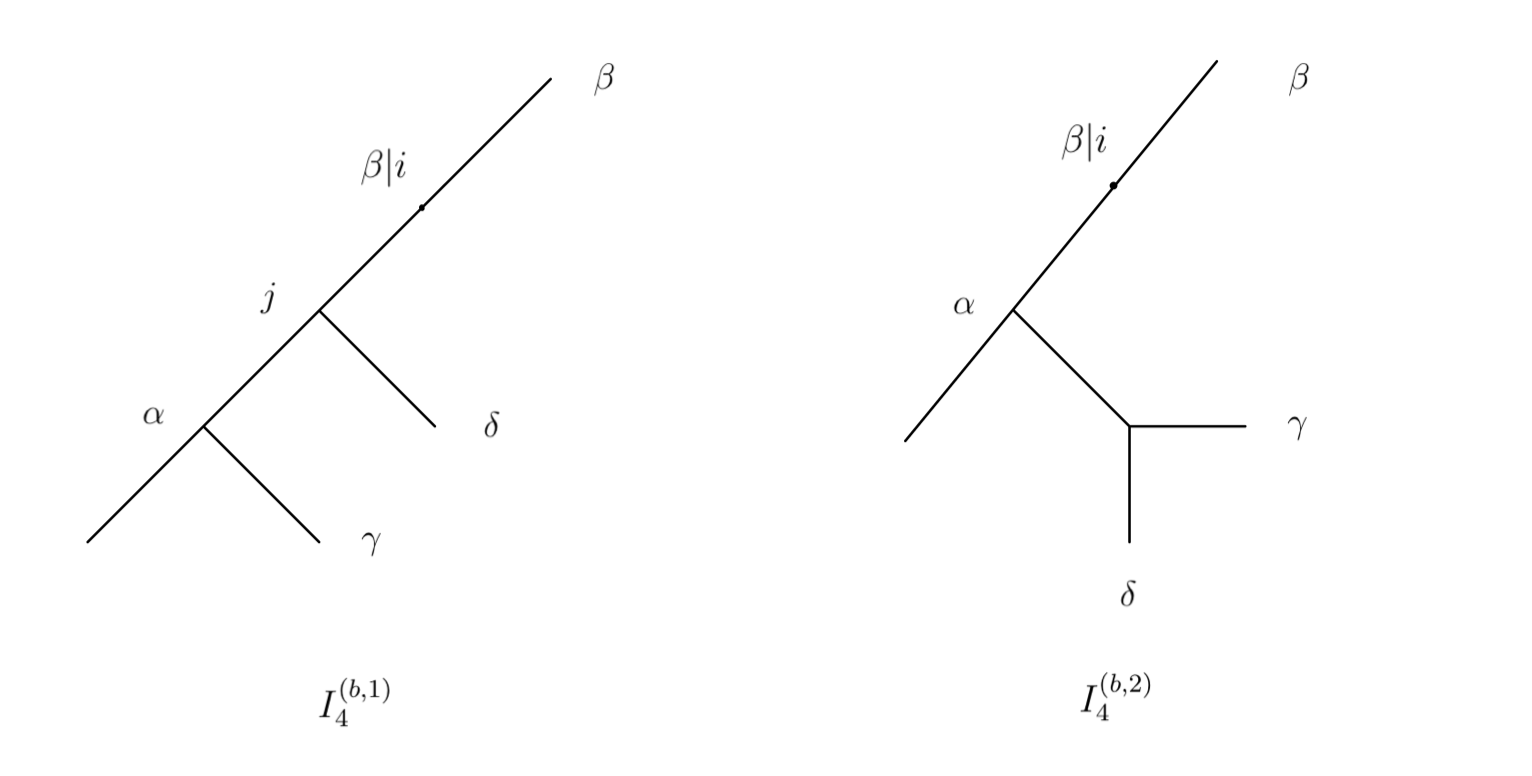}
    \caption[Branching Particle System]{\label{fig4}   Two cases for $I_4^{(b)}$.
      }
  \end{center}
\end{figure}

 \no {\bf  Case (b,1).} Since $T_{\pi\delta} <T_{\beta\vert i}$, we may set $\delta\wedge \beta=\beta\vert j$ for some $j\leq i$. Let $\vert \delta\vert =j+n$ for some $n\geq 0$. It follows that 
 \begin{align}  \label{9ea1.82}
I_4^{(b,1)}\leq &\frac{CX_0^0(1)}{\psi(N)}\sum_{k=0}^\infty \sum_{\substack{\alpha\geq i_0,\\ \vert \alpha\vert =k}} \sum_{l=1}^\infty   \sum_{m=0}^\infty\sum_{\substack{\beta\geq \alpha,\\ \vert \beta\vert =k+l}}\sum_{\substack{ \gamma\geq \alpha,\\ \vert \gamma\vert =k+m}} \sum_{i=k}^{k+l-1}   \sum_{j=0}^i \sum_{n=0}^\infty \sum_{\substack{\delta\geq \beta\vert j,\\ \vert \delta\vert =j+n}}    \nn\\
\E\Big(&1(T_{\alpha}\leq t)   1(T_{\pi\gamma}<T_{\alpha}+\tau_N) 1(T_{\pi\beta}<T_{\alpha}+\tau_N)1(T_{\pi\delta}-T_{\beta\vert j}<2t)  \nn\\
& 1(B^\beta-B^\gamma\in \cN_N) 1(B^{\beta\vert i}+W^{\beta\vert i}=B^\delta)  \Big).
\end{align}
where we have replaced $T_{\pi\delta}<T_{\beta\vert i}$ by $T_{\pi\delta}-T_{\beta\vert j}<T_{\beta}<T_\alpha+\tau_N\leq  2t$. Similar to \eqref{9ea1.53}, the expectation above is at most
 \begin{align}  \label{9ea1.80}
&\E\Big(1(T_{\alpha}\leq t)   1(T_{\pi\gamma}<T_{\alpha}+\tau_N) 1(T_{\pi\beta}<2t) 1(T_{\pi\delta}-T_{\beta\vert j} <2t)\Big)  \nn\\
&(\frac{N+\theta}{2N+\theta})^{k+n+l+m}   \E^{\beta,\gamma, \delta}\Big(1(\sqrt{N}(\hat{B}^\beta-\hat{B}^\gamma)\in [-2,2]^d)  1(B^{\beta\vert i}+W^{\beta\vert i}=B^\delta) \Big),
\end{align}
where we have used \eqref{9ea2.82} to bound $1(B^\beta-B^\gamma\in \cN_N)$.  For the first expectation, we condition on $\cH_{\beta}$ to get
\begin{align*}  
&\E\Big(1(T_{\pi\gamma}-T_{\alpha}<\tau_N)1(T_{\pi\delta}-T_{\beta\vert j} <2t)  \Big\vert \cH_\beta\Big)\nn\\
&=\Pi((2N+\theta)\tau_N, m-1) \cdot \Pi(2(2N+\theta)t), n-1).
\end{align*}
Therefore the first expectation in \eqref{9ea1.80} is equal to
\begin{align}   \label{9ea3.81} 
&\E\Big(1(\Gamma_{k+1}\leq (2N+\theta)t)    1(\Gamma_{k+l}<\Gamma_{k+1}+(2N+\theta)\tau_N)  \Big)\nn\\
& \times \Pi((2N+\theta)\tau_N, m-1) \cdot \Pi(2(2N+\theta)t, n-1).
\end{align}
Turning to the second expectation in \eqref{9ea1.80}, similar to the derivation of \eqref{9ea1.51}, we condition on $\cH_{\beta\vert (i+1)}\vee \cH_\delta$  to get
 \begin{align}   \label{9ea3.82} 
&  \E^{\beta, \gamma,{\delta}}\Big(1(\sqrt{N}(\hat{B}^\beta-\hat{B}^\gamma)\in [-2,2]^d) 1(B^{\beta\vert i}+W^{\beta\vert i}=B^\delta) \Big)\nn\\
&\leq    \frac{C}{(k+l-i+m+1)^{d/2}} \P^{\beta, \gamma,{\delta}}(B^{\beta\vert i}+W^{\beta\vert i}=B^\delta),
\end{align}
and then bound the last probability above by
 \begin{align}     \label{9ea3.83} 
 &\P(B^{\beta\vert i}+W^{\beta\vert i}=B^\delta)=\frac{1}{\psi(N)} \P(B^{\beta\vert i}-B^\delta\in \cN_N)\nn\\
 &= \frac{1}{\psi(N)} \P((B^{\beta\vert i}-B^{\beta\vert j})-(B^\delta-B^{\beta\vert j})\in \cN_N)\leq \frac{1}{\psi(N)}\frac{C}{(1+n+i-j)^{d/2}}.
\end{align}
Apply \eqref{9ea3.81}, \eqref{9ea3.82} and \eqref{9ea3.83} to bound \eqref{9ea1.80}  by 
  \begin{align*} 
&(\frac{N+\theta}{2N+\theta})^{k+n+l+m}   \frac{C}{(k+l-i+m+1)^{d/2}}\frac{1}{\psi(N)} \frac{1}{(1+n+i-j)^{d/2}}\nn\\
&\E\Big(1(\Gamma_{k+1}\leq (2N+\theta)t)    1(\Gamma_{k+l}<\Gamma_{k+1}+(2N+\theta)\tau_N) \Big)\nn\\
&\cdot \Pi((2N+\theta)\tau_N, m-1) \cdot \Pi(2(2N+\theta)t, n-1),
\end{align*}
thus giving
 \begin{align*}  
I_4^{(b,1)}\leq &\frac{C}{\psi(N)^2}X_0^0(1)\sum_{k=0}^\infty  \sum_{l=1}^\infty   \sum_{m=0}^\infty    \sum_{i=k}^{k+l-1}   \sum_{j=0}^i \sum_{n=0}^\infty (\frac{2N+2\theta}{2N+\theta})^{k+n+l+m}    \nn\\
& \cdot \frac{1}{(k+l-i+m+1)^{d/2}} \frac{1}{(1+n+i-j)^{d/2}}\nn\\
&\cdot \Pi((2N+\theta)\tau_N, m-1) \cdot \Pi(2(2N+\theta)t, n-1)\nn\\
&\cdot \E\Big(1(\Gamma_{k+l}\leq 2(2N+\theta)t)    1(\Gamma_{k+l}<\Gamma_{k+1}+(2N+\theta)\tau_N) \Big),
\end{align*}
where we also replace $1\{\Gamma_{k+1}\leq (2N+\theta)t\} $ by $1\{\Gamma_{k+l}<(2N+\theta)t+(2N+\theta)\tau_N\leq 2(2N+\theta)t\} $.
By Lemma \ref{9l4.0}, the sum of $n$ gives at most ${C }/{(1+i-j)^{d/2-1}}$,
and the sum of $m$ gives at most $ {C }/{(1+k+l-i)^{d/2-1}}$.
Next, the sum of $j$ is bounded by 
\begin{align*}  
 \sum_{j=0}^i  \frac{C}{(1+i-j)^{d/2-1}}\leq CI(i)\leq CI(l+k).
\end{align*}
The sum of $i$ gives
\begin{align*}  
 \sum_{i=k}^{k+l-1}  \frac{C}{(1+k+l-i)^{d/2-1}}\leq CI(l)\leq CI(l+k).
\end{align*}
We are left with
\begin{align*}  
I_4^{(b,1)}\leq &\frac{C}{\psi(N)^2}X_0^0(1)\sum_{k=0}^\infty  \sum_{l=1}^\infty   I(k+l)^2      (\frac{2N+2\theta}{2N+\theta})^{k+l} \nn\\
&\E\Big(1(\Gamma_{k+l}\leq 2(2N+\theta)t)    1(\Gamma_{k+l}<\Gamma_{k+1}+(2N+\theta)\tau_N) \Big).
\end{align*}
 Rewrite the sum of $k,l$ by letting $n=k+l$ to see
\begin{align*}  
I_4^{(b,1)}\leq &\frac{C}{\psi(N)^2}X_0^0(1)\sum_{n=1}^\infty   (\frac{2N+2\theta}{2N+\theta})^{n}  I(n)^2      \E\Big(1(\Gamma_{n}\leq 2(2N+\theta)t)  \nn\\
& \sum_{k=0}^{n}    1(\Gamma_{n}-\Gamma_{k+1}<(2N+\theta)\tau_N) \Big).
\end{align*}
Now apply lemma \ref{9l8.4} with $m=2$ to conclude that
\begin{align*}  
I_4^{(b,1)}\leq &\frac{C}{\psi(N)^2}X_0^0(1)\cdot CN^2\tau_N I(N)^2\leq CX_0^0(1) \tau_N\to 0.
\end{align*}

 \no {\bf  Case (b,2).} Turning to $I_4^{(b,2)}$, we set $\delta\wedge \gamma=\gamma\vert j$ for some $k+1\leq j\leq k+m$, and let $\vert \delta\vert =j+n$ for some $n\geq 0$. See Figure \ref{fig4}. It follows that 
 \begin{align}  \label{9ea2.02}
I_4^{(b,2)}\leq &\frac{C}{\psi(N)}X_0^0(1)\sum_{k=0}^\infty \sum_{\substack{\alpha\geq i_0,\\ \vert \alpha\vert =k}} \sum_{l=1}^\infty   \sum_{m=0}^\infty\sum_{\substack{\beta>\alpha,\\ \vert \beta\vert =k+l}}\sum_{\substack{ \gamma\geq\alpha,\\ \vert \gamma\vert =k+m }}  \sum_{i=k}^{k+l-1}   \sum_{j=k+1}^{k+m} \sum_{n=0}^\infty \sum_{\substack{\delta\geq \gamma\vert j,\\ \vert \delta\vert =j+n}}    \nn\\
\E\Big(&1(T_{\alpha}\leq t)  1(T_{\pi\beta}<T_{\alpha}+\tau_N)  1(T_{\pi\gamma}<T_{\alpha}+\tau_N) 1(T_{\pi\delta} <T_{\alpha}+\tau_N)  \nn\\
& 1(B^\beta-B^\gamma\in \cN_N) 1(B^{\beta\vert i}+W^{\beta\vert i}=B^\delta)  \Big),
\end{align}
where we have replaced $T_{\pi\delta}<T_{\beta\vert i}\leq T_\beta<T_\alpha+\tau_N$.
As in \eqref{9ea1.80},  the expectation above is at most 
 \begin{align}  \label{9ea1.90}
& (\frac{N+\theta}{2N+\theta})^{k+n+l+m} \E\Big(1(T_{\alpha}\leq t)   1(T_{\pi\beta}<T_{\alpha}+\tau_N) \nn\\
& 1(T_{\pi\gamma}<T_{\gamma\vert j}+\tau_N)1(T_{\pi\delta} <T_{\gamma\vert j}+\tau_N) 1(T_{\gamma\vert (j-1)}<T_\alpha+\tau_N)\Big)  \nn\\
&  \E^{\beta,\gamma, \delta}\Big(1(\sqrt{N}(\hat{B}^\beta-\hat{B}^\gamma)\in [-2,2]^d)   1(B^{\beta\vert i}+W^{\beta\vert i}=B^\delta) \Big),
\end{align}
where we also use $T_{\alpha} \leq T_{\gamma\vert j}$ and $T_{\gamma\vert (j-1)}\leq T_{\pi\gamma}<T_\alpha+\tau_N$.
For the first expectation above, we condition on $\cH_\beta \vee \cH_{\gamma\vert j}$ to get 
\begin{align*}  
&\E\Big(1(T_{\pi\gamma}-T_{\gamma\vert j}<\tau_N)1(T_{\pi\delta}-T_{\gamma\vert j} <\tau_N)  \Big\vert \cH_\beta \vee \cH_{\gamma\vert j}\Big)\nn\\
&=\Pi((2N+\theta)\tau_N, k+m-j-1) \cdot \Pi((2N+\theta)\tau_N, n-1).
\end{align*}
Next, we condition on $\cH_\beta$ to see
\begin{align*}  
& \E\Big(1(T_{\gamma\vert (j-1)}<T_\alpha+\tau_N) \Big\vert \cH_\beta \Big)=\Pi((2N+\theta)\tau_N, j-k-1).
\end{align*}
We conclude the first expectation in \eqref{9ea1.90} is equal to
 \begin{align}   \label{9ea2.01}
 \E\Big(&1(\Gamma_{k+1}\leq (2N+\theta)t)   1(\Gamma_{k+l}-\Gamma_{k+1}<(2N+\theta)\tau_N)\Big)\\
&\cdot \Pi((2N+\theta)\tau_N, k+m-j-1) \cdot\Pi((2N+\theta)\tau_N, n-1)\cdot \Pi((2N+\theta)\tau_N, j-k-1).\nn
\end{align}
Turning to the second expectation in \eqref{9ea1.90}, since $\cH_\delta$ includes the information from $\cH_{\gamma\vert j}$, similar to \eqref{e4.33} we may get that
\begin{align*} 
& \E^{\beta,\gamma, \delta}\Big(1(\sqrt{N}(\hat{B}^\beta- \hat{B}^\gamma)\in  [-2,2]^d) \Big\vert \cH_{\beta\vert (i+1)}\vee \cH_\delta \Big)\nn\\
=  & \P\Big(\sqrt{N} \sum_{r=i+1}^{k+l-1}   \hat{W}^{\beta\vert r} -\sqrt{N} \sum_{r=j}^{k+m-1}    \hat{W}^{\gamma\vert r} \in [-2,2]^d-\sqrt{N} \sum_{r=k+1}^{i}   \hat{W}^{\beta\vert r}-\sqrt{N} \sum_{r=k+1}^{j-1}   \hat{W}^{\gamma\vert r} \Big)\nn\\
\leq& \frac{C}{(1+(k+l-i)+(k+m-j))^{d/2}}, 
\end{align*}
where the last inequality follows from Lemma \ref{9l4.2}, thus giving
 \begin{align*}  
&  \E^{\beta,\gamma, \delta}\Big(1(\sqrt{N}(\hat{B}^\beta-\hat{B}^\gamma)\in [-2,2]^d)   1(B^{\beta\vert i}+W^{\beta\vert i}=B^\delta) \Big)\\
&\leq  \frac{C}{(1+(k+l-i)+(k+m-j))^{d/2}} \P(B^{\beta\vert i}+W^{\beta\vert i}=B^\delta).
\end{align*}
The last probability above can be bounded by  
 \begin{align*}     
 &\P(B^{\beta\vert i}+W^{\beta\vert i}=B^\delta)=\frac{1}{\psi(N)} \P(B^{\beta\vert i}-B^\delta\in \cN_N)\\
 &= \frac{1}{\psi(N)}   \P\Big((B^{\beta\vert i}-B^\alpha)-(B^\delta-B^\alpha)\in \cN_N\Big)\nn\\
 &\leq \frac{1}{\psi(N)}\frac{C}{(1+(i-k)+(n+j-k))^{d/2}}\leq  \frac{1}{\psi(N)}\frac{C}{(1+n+j-k)^{d/2}},\nn
\end{align*}
where the last inequality uses $i\geq k$. Conclude from the above that the second expectation in \eqref{9ea1.90} is bounded by
 \begin{align}   \label{9ea1.91}
& \frac{C}{\psi(N)} \frac{1}{(1+(k+l-i)+(k+m-j))^{d/2}} \frac{1}{(1+n+j-k)^{d/2}}.
\end{align}
Combine \eqref{9ea1.90},  \eqref{9ea2.01} and \eqref{9ea1.91}  to see  \eqref{9ea2.02} becomes
 \begin{align}  \label{9ec6.81}
I_4^{(b,2)}\leq &\frac{C}{\psi(N)^2}X_0^0(1)\sum_{k=0}^\infty  \sum_{l=1}^\infty   \sum_{m=0}^\infty    \sum_{i=k}^{k+l-1}   \sum_{j=k+1}^{k+m} \sum_{n=0}^\infty (\frac{2N+2\theta}{2N+\theta})^{k+n+l+m}    \nn\\
&  \frac{1}{(1+(k+l-i)+(k+m-j))^{d/2}}  \frac{1}{(1+n+j-k)^{d/2}}\cdot \Pi((2N+\theta)\tau_N, n-1)\nn\\
& \cdot \Pi((2N+\theta)\tau_N, j-k-1)\cdot  \Pi((2N+\theta)\tau_N, k+m-j-1) \nn\\
& \E\Big(1(\Gamma_{k+1}\leq (2N+\theta)t)   1(\Gamma_{k+l}-\Gamma_{k+1}<(2N+\theta)\tau_N)\Big).
\end{align}
Rewrite the above by letting $i=i-k$ and $j=j-k$ to get 
 \begin{align}  \label{9ea2.03}
I_4^{(b,2)}\leq &\frac{C}{\psi(N)^2}X_0^0(1)\sum_{k=0}^\infty  \sum_{l=1}^\infty   \sum_{m=0}^\infty    \sum_{i=0}^{l-1}   \sum_{j=1}^{m} \sum_{n=0}^\infty (\frac{2N+2\theta}{2N+\theta})^{k+n+l+m}    \nn\\
&  \frac{1}{(1+(l-i)+(m-j))^{d/2}}  \frac{1}{(1+n+j)^{d/2}}\cdot \Pi((2N+\theta)\tau_N, n-1)\nn\\
& \cdot \Pi((2N+\theta)\tau_N, j-1)\cdot  \Pi((2N+\theta)\tau_N, m-j-1) \nn\\
& \E\Big(1(\Gamma_{k+l}\leq 2(2N+\theta)t)   1(\Gamma_{k+l}-\Gamma_{k+1}<(2N+\theta)\tau_N)\Big),
\end{align}
where we also replace $1\{\Gamma_{k+1}\leq (2N+\theta)t\} $ by $1\{\Gamma_{k+l}<(2N+\theta)t+(2N+\theta)\tau_N\leq 2(2N+\theta)t\}$.
The sum of $n$ gives
\begin{align*}  
 \sum_{n=0}^\infty &(\frac{2N+2\theta}{2N+\theta})^{n} \Pi((2N+\theta)\tau_N, n-1)  \frac{1}{(1+n+j)^{d/2}}\leq     \frac{C}{(1+j)^{d/2-1}},
\end{align*}
where the inequality follows by Lemma \ref{9l4.0}. 
Then the sum of $m,j$ gives
\begin{align*}  
& \sum_{m=0}^\infty \sum_{j=1}^{m}   (\frac{2N+2\theta}{2N+\theta})^{m} \Pi((2N+\theta)\tau_N, j-1)   \frac{C }{(1+j)^{d/2-1}}\nn\\
&\cdot  \Pi((2N+\theta)\tau_N, m-j-1)  \frac{1}{(1+(l-i)+(m-j))^{d/2}}.
\end{align*}
By Fubini's theorem, the above sum is equal to 
\begin{align*}  
&C \sum_{j=1}^\infty (\frac{2N+2\theta}{2N+\theta})^{j} \Pi((2N+\theta)\tau_N, j-1)  \frac{1}{(1+j)^{d/2-1}}   \nn\\
& \times \sum_{m=j}^{\infty}  (\frac{2N+2\theta}{2N+\theta})^{m-j}\Pi((2N+\theta)\tau_N, m-j-1)  \frac{1}{(1+(l-i)+(m-j))^{d/2}}\nn\\
=&C \sum_{j=1}^\infty (\frac{2N+2\theta}{2N+\theta})^{j}\Pi((2N+\theta)\tau_N, j-1)  \frac{1}{(1+j)^{d/2-1}}   \nn\\
& \times \sum_{m=0}^{\infty}  (\frac{2N+2\theta}{2N+\theta})^{m} \Pi((2N+\theta)\tau_N, m-1)  \frac{1}{(1+(l-i)+m)^{d/2}}\nn\\
\leq& CI(N) \cdot\frac{C}{(1+l-i)^{d/2-1}},
\end{align*}
where we have used Lemma \ref{9l8.7} to bound the sum of $j$, and Lemma \ref{9l4.0} to bound the sum of $m$.

Now conclude from the above that \eqref{9ec6.81} becomes
 \begin{align*}  
I_4^{(b,2)}\leq &\frac{CI(N)}{\psi(N)^2}X_0^0(1)\sum_{k=0}^\infty  \sum_{l=1}^\infty     (\frac{2N+2\theta}{2N+\theta})^{k+l}    \sum_{i=0}^{l-1}   \frac{1}{(1+l-i)^{d/2-1}}  \nn\\
& \E\Big(1(\Gamma_{k+l}\leq 2(2N+\theta)t)   1(\Gamma_{k+l}-\Gamma_{k+1}<(2N+\theta)\tau_N)\Big).
\end{align*}
The sum of $i$ gives at most $I(l)\leq I(l+k)$ and we are left with 
\begin{align*}  
I_4^{(b,2)}&\leq \frac{CI(N)}{\psi(N)^2}X_0^0(1)\sum_{k=0}^\infty  \sum_{l=1}^\infty   I(k+l)        (\frac{2N+2\theta}{2N+\theta})^{k+l} \nn\\
&  \E\Big(1(\Gamma_{k+l}\leq 2(2N+\theta)t)    1(\Gamma_{k+l}<\Gamma_{k+1}+(2N+\theta)\tau_N)     \Big).
\end{align*}
Rewrite the above sum of $k,l$ by letting $n=k+l$ to see
\begin{align*}  
I_4^{(b,2)}\leq &\frac{CI(N)}{\psi(N)^2}X_0^0(1) \sum_{n=1}^\infty   (\frac{2N+2\theta}{2N+\theta})^{n}  I(n)      \E\Big(1(\Gamma_{n}\leq 2(2N+\theta)t)  \nn\\
& \sum_{k=0}^{n}    1(\Gamma_{n}-\Gamma_{k+1}<(2N+\theta)\tau_N) \Big).
\end{align*}
Now apply the lemma \ref{9l8.4} with $m=1$ to conclude
\begin{align*}  
I_4^{(b,2)}\leq &\frac{CI(N)}{\psi(N)^2}X_0^0(1)  CN^2\tau_N I(N)\leq CX_0^0(1) \tau_N\to 0.
\end{align*}
The proof is complete.

\subsection{Proof of Lemma \ref{9l10.1}(ii)}\label{s9.2}
Since $B^\gamma=B^\gamma_{T_\gamma^-}\neq \Delta$, we must have $\zeta_\gamma^0=T_{\gamma}$ by its definition. If $m=0$, we cannot have both $\zeta_\gamma^0\in (T_\beta-\tau_N, T_\gamma)$ and $\zeta_\gamma^0=T_{\gamma}$ hold. So \eqref{9ea1.16} is trivial for $m=0$. It suffices to prove for $m\geq 1$.
 In order that $\zeta_\gamma^m \in (T_{\beta}-\tau_N, T_{\gamma})$, there has to be some $i<\vert \gamma\vert $ so that 
\begin{align*} 
T_{\gamma\vert i}>T_{\beta}-\tau_N, \quad B^{\gamma\vert i}+W^{\gamma\vert i} \in \overline{\cR}^{X^{m-1}}_{T_{\gamma\vert i}^-}, \quad e_{\gamma\vert i}=\gamma_{i+1},
\end{align*}
meaning at time $T_{\gamma\vert i}$, the particle $\gamma\vert i$ gives birth to $\gamma\vert (i+1)$ which is displaced from $B^{\gamma\vert i}$ by a distance of $W^{\gamma\vert i}$. However, that location had already been visited by some particle $\delta$ in $X^{m-1}$ before time $T_{\gamma\vert i}$, or that location lies in $K_0^N$. Similar to \eqref{9ec9.53}, we may bound $1\{\zeta_\gamma^m \in (T_\beta-\tau_N, T_\gamma)\}$ by
\begin{align*}
&\sum_{i=0}^{\vert \gamma\vert -1} 1(T_{\gamma\vert i}>T_\beta-\tau_N) \sum_{\delta} 1(T_{\pi\delta} <T_{\gamma\vert i}, B^{\gamma\vert i}+W^{\gamma\vert i}=B^\delta)\nn\\
&\quad +\sum_{i=0}^{\vert \gamma\vert -1} 1(T_{\gamma\vert i}>T_{\beta}-\tau_N)  1(B^{\gamma\vert i}+W^{\gamma\vert i}\in K_0^N).
\end{align*}
Therefore the left-hand side term of \eqref{9ea1.16} is at most
\begin{align*} 
\frac{1}{N\psi(N)} &\E\Big(\sum_{\beta} 1(T_{ \beta} \leq t)  \sum_{\gamma} 1(T_{\pi\gamma}<T_\beta) 1(B^\beta-B^\gamma\in \cN_N) 1(T_{\beta\wedge \gamma}>T_\beta-\tau_N)\nn\\
& \sum_{i=0}^{\vert \gamma\vert -1} 1(T_{\gamma\vert i}>T_\beta-\tau_N) 1(B^{\gamma\vert i}+W^{\gamma\vert i}\in K_0^N) \Big)\nn\\
+\frac{1}{N\psi(N)} &\E\Big(\sum_{\beta} 1(T_{ \beta} \leq t)  \sum_{\gamma} 1(T_{\pi\gamma}<T_\beta) 1(B^\beta-B^\gamma\in \cN_N) 1(T_{\beta\wedge \gamma}>T_\beta-\tau_N)\nn\\
& \sum_{i=0}^{\vert \gamma\vert -1} 1(T_{\gamma\vert i}>T_\beta-\tau_N)  \sum_{\delta} 1(T_{\pi\delta} <T_{\gamma\vert i}, B^{\gamma\vert i}+W^{\gamma\vert i}=B^\delta)  \Big):=J_0+J_1.
\end{align*}

 \no {\bf  Case for $J_0$.}  First for $J_0$, if $i\leq \vert \beta\wedge \gamma\vert-1 $, we have $\gamma\vert i=\beta\vert i$, so the sum for $i\leq \vert \beta\wedge \gamma\vert-1$ becomes
\begin{align*} 
\frac{1}{N\psi(N)} \E\Big(&\sum_{\beta} 1(T_{ \beta} \leq t)\sum_{\gamma}1(T_{\pi\gamma}<T_\beta) 1(B^\beta-B^\gamma\in \cN_N) \nn\\
&1(T_{\gamma\wedge \beta}>T_\beta-\tau_N)  \sum_{i=0}^{ \vert \beta\wedge \gamma\vert -1} 1(T_{\beta\vert i}>T_\beta-\tau_N) 1(B^{\beta\vert i}+W^{\beta\vert i}\in K_0^N) \Big).
\end{align*}
The above can be bounded by $I_0^{(b)}$ from \eqref{9ea5.667} and so will converge to $0$. 

Turning to $i\geq \vert \beta\wedge \gamma\vert $, we have $T_{\gamma\vert i}\geq T_{\beta\wedge \gamma}$. Hence  $1(T_{\gamma\vert i}>T_\beta-\tau_N)$ is absorbed into $1(T_{\beta\wedge \gamma}>T_\beta-\tau_N)$.  Let $\alpha=\beta\wedge \gamma$  to see the sum for $i\geq \vert \beta\wedge \gamma\vert$ in $J_0$ is at most 
 \begin{align*} 
 \frac{1}{N\psi(N)} \E\Big(&\sum_{\alpha}\sum_{\beta\geq\alpha} \sum_{\gamma\geq \alpha}1(T_{\alpha} \leq t) \sum_{i=\vert \alpha\vert }^{\vert \gamma\vert -1}    1(T_{\pi\gamma}<T_{\alpha}+\tau_N)  \nn\\
& 1(T_{\pi\beta}<T_{\alpha}+\tau_N) 1(B^\beta-B^\gamma\in \cN_N)  1(B^{\gamma\vert i}+W^{\gamma\vert i}\in K_0^N) \Big),
\end{align*}
where we also use $T_{\alpha}\leq T_\beta\leq t$, $T_\beta< T_{\alpha}+\tau_N$ and $T_{\pi\beta}\leq T_\beta$. Exchange $\beta,\gamma$ to see the above equals
 \begin{align*}  
 \frac{1}{N\psi(N)} \E\Big(&\sum_{\alpha}\sum_{\beta\geq\alpha} \sum_{\gamma\geq \alpha}1(T_{\alpha} \leq t) \sum_{i=\vert \alpha\vert }^{\vert \beta\vert -1}    1(T_{\pi\beta}<T_{\alpha}+\tau_N)  \nn\\
& 1(T_{\pi\gamma}<T_{\alpha}+\tau_N) 1(B^\beta-B^\gamma\in \cN_N) 1(B^{\beta\vert i}+W^{\beta\vert i}\in K_0^N) \Big),
\end{align*}
reproducing the same expression as in \eqref{9ed1.661}. Hence the above converges to $0$ and we conclude $J_0\to 0$.

 \no {\bf  Case for $J_1$.}  The case for $J_1$ is similar. If $i\leq \vert \beta\wedge \gamma\vert-1$, we get sum for $i\leq \vert \beta\wedge \gamma\vert-1$ is equal to
\begin{align*} 
\frac{1}{N\psi(N)} \E\Big(&\sum_{\beta} 1(T_{ \beta} \leq t)\sum_{\gamma}1(T_{\pi\gamma}<T_\beta) 1(B^\beta-B^\gamma\in \cN_N) \nn\\
&1(T_{\gamma\wedge \beta}>T_\beta-\tau_N)  \sum_{i=0}^{ \vert \beta\wedge \gamma\vert -1} 1(T_{\beta\vert i}>T_\beta-\tau_N)  \nn\\
&\sum_{\delta} 1(T_{\pi\delta} <T_{\beta\vert i}, B^{\beta\vert i}+W^{\beta\vert i}=B^\delta) \Big).
\end{align*}
We may bound the above by $I_2$ from\eqref{9e1.32}. Hence the above converges to $0$.  

For $i\geq \vert \beta\wedge \gamma\vert $, again $1(T_{\gamma\vert i}>T_\beta-\tau_N)$ is absorbed into $1(T_{\beta\wedge \gamma}>T_\beta-\tau_N)$ and we may let $\alpha=\beta\wedge \gamma$  to see the sum for $i\geq \vert \beta\wedge \gamma\vert$ in $J_1$ is at most 
 \begin{align*} 
 \frac{1}{N\psi(N)} \E\Big(&\sum_{\alpha}\sum_{\beta\geq\alpha} \sum_{\gamma\geq \alpha}1(T_{\alpha} \leq t) \sum_{i=\vert \alpha\vert }^{\vert \gamma\vert -1}    1(T_{\pi\gamma}<T_{\alpha}+\tau_N) 1(T_{\pi\beta}<T_{\alpha}+\tau_N)   \nn\\
&1(B^\beta-B^\gamma\in \cN_N) \sum_{\delta} 1(T_{\pi\delta} <T_{\gamma\vert i}, B^{\gamma\vert i}+W^{\gamma\vert i}=B^\delta) \Big).
\end{align*}
Exchange $\beta,\gamma$ to see the above equals 
 \begin{align*}  
 \frac{1}{N\psi(N)} \E\Big(&\sum_{\alpha}\sum_{\beta\geq\alpha} \sum_{\gamma\geq \alpha}1(T_{\alpha} \leq t) \sum_{i=\vert \alpha\vert }^{\vert \beta\vert -1}    1(T_{\pi\beta}<T_{\alpha}+\tau_N)  1(T_{\pi\gamma}<T_{\alpha}+\tau_N)  \nn\\
&1(B^\beta-B^\gamma\in \cN_N) \sum_{\delta} 1(T_{\pi\delta} <T_{\beta\vert i}, B^{\beta\vert i}+W^{\beta\vert i}=B^\delta) \Big).
\end{align*}
This is the same expression as in \eqref{9ea2.81}. Hence the above converges to $0$ and we conclude $J_1\to 0$. The proof is complete.

\bibliographystyle{plain}
\def\cprime{$'$}

\appendix

\section{Proofs of Lemmas \ref{9l2.1}-\ref{9l2.4}}\label{9a1}

\begin{proof}[Proof of Lemma \ref{9l2.1}]
Recall from \eqref{9e2.1} that
\begin{align}\label{9ec4.78}
M_t^n(\phi)&=\frac{1}{N} \sum_{\beta} 1(T_\beta\leq t, \zeta_\beta^n> T_{\pi \beta})  \phi(B^\beta) g_\beta\nn\\
&=\frac{1}{N} \sum_{\beta} 1(T_\beta\leq t, \zeta_\beta^n\geq T_{\beta})  \phi(B^\beta) g_\beta,
\end{align}
where $g_\beta=\delta_\beta-\frac{\theta}{2N+\theta}$. Then it follows immediately from Lemma 3.3 and 3.4(a) of \cite{DP99} that $M_t^n(\phi)$ is an $\cF_t$-martingale. One may also check
\begin{align*}
[M^n(\phi)]_t=\frac{1}{N^2} \sum_{\beta} 1(T_\beta\leq t, \zeta_\beta^n\geq T_{\beta})  \phi(B^\beta)^2 g_\beta^2.
\end{align*}
Since $\delta_\beta$ is independent of $\cF_{T_\beta^-}$ by definition, we have 
\begin{align*}
\E(g_\beta^2\vert \cF_{T_\beta^-})=\E(g_\beta^2)=1-\eps_N^2.
\end{align*}
By Lemma \ref{9l6.1}, we get
\begin{align*}
Z_t:=\frac{1}{N^2} \sum_{\beta} 1(T_\beta\leq t, \zeta_\beta^n\geq T_{\beta})  \phi(B^\beta)^2 [g_\beta^2-(1-\eps_N^2)]
\end{align*}
is a martingale. Apply Lemma \ref{9l3.2} to further get
\begin{align*}
N_t:=&(1-\eps_N^2)\frac{1}{N^2} \sum_{\beta} 1(T_\beta\leq t, \zeta_\beta^n\geq T_{\beta})  \phi(B^\beta)^2 \\
&-(1-\eps_N^2) \frac{1}{N^2} \int_0^t (2N+\theta) \sum_{\beta} 1(T_{\pi \beta}<r< T_\beta, \zeta_\beta^n>r)  \phi(B^\beta)^2
\end{align*}
is a martingale. Recalling $X_r^n(\phi^2)$ from \eqref{9e0.03}, we may conclude from the above that  
\begin{align*}
[M^n(\phi)]_t=Z_t+N_t+(1-\eps_N^2) \Big(2+\frac{\theta}{N}\Big) \int_0^t X_r^n(\phi^2) dr.
\end{align*}
Since both $Z_t$ and $N_t$ are martingales, we get
\begin{align*}
\la M^n(\phi)\ra_t=(1-\eps_N^2) \Big(2+\frac{\theta}{N}\Big) \int_0^t X_r^n(\phi^2) dr.
\end{align*}
Turning to $\la M^2(\phi)- M^1(\phi)\ra_t$, we obtain from \eqref{9ec4.78} that
\begin{align*}
M_t^2(\phi)-M_t^2(\phi)=\frac{1}{N} \sum_{\beta} 1(T_\beta\leq t) 1(\zeta_\beta^1< T_{\beta}\leq \zeta_\beta^2)  \phi(B^\beta) g_\beta.
\end{align*}
The arguments for deriving $\la M^2(\phi)- M^1(\phi)\ra_t$ follow in a similar way.

\end{proof}

\begin{proof}[Proof of Lemma \ref{9l2.2}]
Recall from \eqref{9e2.1} that
\begin{align} \label{9ec6.07}
D_t^{n,1}(\phi)=\frac{\theta}{2N+\theta} \frac{1}{N} \sum_{\beta} 1(T_\beta\leq t, \zeta_\beta^n\geq T_{\beta})  \phi(B^\beta).
\end{align}
By Lemma \ref{9l3.2}, we have
\begin{align} \label{9ec5.07}
M_t:=&D_t^{n,1}(\phi)-   \frac{\theta}{N} \int_0^t \sum_{\beta} 1(T_{\pi \beta}<r<T_\beta, \zeta_\beta^n>r)  \phi(B^\beta) dr\nn\\
=&D_t^{n,1}(\phi)-\theta\int_0^t X_r^n(\phi) dr
\end{align}
is a martingale. Notice that
\begin{align*} 
[M]_t=\Big(\frac{\theta}{2N+\theta}\Big)^2 \frac{1}{N^2} \sum_{\beta} 1(T_\beta\leq t, \zeta_\beta^n\geq T_{\beta})  \phi(B^\beta)^2.
\end{align*}
Using Lemma \ref{9l3.2} again, we get
\begin{align*} 
&[M]_t- \frac{\theta^2}{(2N+\theta)N^2} \int_0^t \sum_{\beta} 1(T_{\pi \beta}<r<T_\beta, \zeta_\beta^n>r)  \phi(B^\beta)^2 dr\nn\\
&=[M]_t-\frac{\theta^2}{(2N+\theta)N}  \int_0^t X_r^n(\phi^2) dr
\end{align*}
is also a martingale, thus giving $\la M\ra_t=\frac{\theta^2}{(2N+\theta)N}  \int_0^t X_r^n(\phi^2) dr$. By the $L^2$ maximal inequality, we get
 \begin{align*} 
\E\Big(\sup_{s\leq t} M_s^2 \Big)\leq \frac{\theta^2}{(2N+\theta)N}  \int_0^t \E(X_r^n(\phi^2)) dr\nn\\
\leq  \frac{\theta^2}{(2N+\theta)N}  \|\phi\|_\infty \int_0^t e^{\theta r} X_0^0(1) dr\to 0,
\end{align*}
where the second inequality uses Lemma \ref{9l2.0}. By applying Cauchy-Schwartz in the above, we conclude the proof is complete by \eqref{9ec5.07}.
\end{proof}

\begin{proof}[Proof of Lemma \ref{9l2.3}]
Recall from \eqref{9e2.1} that
\begin{align*} 
D_t^{n,2}(\phi)=&\frac{N+\theta}{2N+\theta} \frac{1}{N} \sum_{\beta} 1(T_\beta\leq t, \zeta_\beta^n\geq T_{\beta})  \nabla_\beta\phi,
\end{align*}
where $\nabla_\beta\phi=\phi(B^\beta+W^\beta) -\phi(B^\beta)$. Using Taylor's theorem with expansion at $B^\beta$, $\E(W_\beta^i)=0$ and $\E(W_\beta^i W_\beta^j)=0$ for $i\neq j$, one may obtain
\begin{align*} 
&\E(\nabla_\beta\phi \vert \cF_{T_\beta^-})=\frac{1}{2} \sum_{i=1}^d \phi_{ii}(B^\beta) \E[(W_\beta^i)^2]+R_N^\beta(\omega),\\
&R_N^\beta(\omega)=\frac{1}{2}\E\Big(\sum_{1\leq i, j\leq d} (\phi_{ij}(v(\omega))-\phi_{ij}(B^\beta)) W_\beta^i W_\beta^j \Big\vert \cF_{T_\beta^-}\Big),
\end{align*}
where $\phi_{ij}$ are the partial derivatives of $\phi$ and $v(\omega)$ is a point on the line segment from $B^\beta$ to $B^\beta+W^\beta$. Since $\phi \in C_b^3$ and $|W^\beta|\leq CN^{-1/2}$, we get $|\phi_{ij}(v(\omega))-\phi_{ij}(B^\beta)|\leq CN^{-1/2}$ and
\begin{align}\label{9ec7.45}
R_N^\beta(\omega)\leq CN^{-1/2} (CN^{-1/2})^2\leq CN^{-3/2}.
\end{align}
Notice $\sqrt{N} W_\beta^i$ converges weakly to a random variable uniformly distributed on $[-1,1]$ whose second moment is $1/3$. We conclude from the above that
\begin{align*} 
&\E(\nabla_\beta\phi \vert \cF_{T_\beta^-})=\Big(\frac{1}{6}+\delta_N\Big) N^{-1}\Delta \phi(B^\beta)  +R_N^\beta(\omega),
\end{align*}
where $\delta_N\to 0$ as $N\to\infty$. Define 
\begin{align*}
G_\beta=1(T_\beta\leq \zeta_\beta^n)\Big(\nabla_\beta\phi-\E(\nabla_\beta\phi \vert \cF_{T_\beta^-})\Big).
\end{align*}
Then $\E(G_\beta|\cF_{T_\beta^-})=0$ and $|G_\beta|\leq CN^{-1/2} 1(T_\beta\leq \zeta_\beta^0)$ by $\phi\in C_b^3$. Apply Lemma \ref{9l6.1} to see
\begin{align*}
M_t=\frac{N+\theta}{2N+\theta} \frac{1}{N} \sum_{\beta} 1(T_\beta\leq t) G_\beta \text{ is a martingale,}
\end{align*}
and
\begin{align}\label{9ec7.05}
\E\Big(\sup_{s\leq t} M_s^2\Big) \leq C\frac{1}{N^2} \E\Big(\sum_\beta 1(T_\beta \leq t) CN^{-1} 1(T_\beta\leq \zeta_\beta^0)\Big).
\end{align}
By Lemma 3.4 (b) of \cite{DP99}, we have 
\begin{align}\label{9ec7.42}
\E\Big(\sum_\beta 1(T_\beta \leq t) 1(T_\beta\leq \zeta_\beta^0)\Big)\leq CN^2 X_0^0(1). 
\end{align}
Apply the above in \eqref{9ec7.05} and then use Cauchy-Schwartz to get
\begin{align}\label{9ec7.50}
\E\Big(\sup_{s\leq t} |M_s|\Big) \leq CN^{-1/2} \sqrt{X_0^0(1)}\to 0.
\end{align}
Now rewrite $D_t^{n,2}(\phi)$ as
\begin{align}\label{9ec7.49}
&D_t^{n,2}(\phi)=M_t+E_t(\phi)+D_t^{n,3}(\phi),
 \end{align}
where
\begin{align*}
&D_t^{n,3}(\phi)=\frac{1}{2N+\theta} \frac{1}{N} \Big(\frac{1}{6}+\delta_N\Big) \frac{N+\theta}{N}  \sum_{\beta} 1(T_\beta\leq t, \zeta_\beta^n\geq T_{\beta}) \Delta \phi(B^\beta), \nn\\
&E_t(\phi)=\frac{N+\theta}{2N+\theta} \frac{1}{N} \sum_{\beta} 1(T_\beta\leq t, \zeta_\beta^n\geq T_{\beta})  R_N^\beta(\omega).\nn
 \end{align*}
 For $E_t(\phi)$, use \eqref{9ec7.45} and \eqref{9ec7.42} to get
\begin{align}\label{9ec7.56}
\E\Big(\sup_{s\leq t} |E_t(\phi)|\Big) \leq CN^{-1} CN^2 X_0^0(1) CN^{-3/2}\leq CN^{-1/2} X_0^0(1) \to 0.
\end{align}
For $D_t^{n,3}(\phi)$, by comparing to $D_t^{n,1}(\phi)$ as in \eqref{9ec6.07},  we may apply Lemma \ref{9l2.2} with $\phi=\Delta \phi$ to see  
\begin{align}\label{9ec7.43}
\lim_{N\to \infty} \E\Big(\sup_{s\leq t} \Big\vert D_s^{n,3}(\phi)-\frac{N+\theta}{N} \Big(\frac{1}{6}+\delta_N\Big) \int_0^s X_r^n(\Delta \phi) dr\Big\vert \Big)=0.
\end{align}
The proof is complete in view of \eqref{9ec7.50}, \eqref{9ec7.49},  \eqref{9ec7.56} and \eqref{9ec7.43}.
\end{proof}

\begin{proof}[Proof of Lemma \ref{9l2.4}]
Recall from \eqref{9e2.1} that
\begin{align*} 
E_t^{n}(\phi)=&\frac{1}{N} \sum_{\beta} a_\beta^n(t) h_\beta \Big(\phi(B^\beta+W^\beta)1(B^\beta+W^\beta\notin \cR^{X^{n-1}}_{T_\beta^-}) -\phi(B^\beta) \Big),
\end{align*}
where $h_\beta=1(\delta_\beta=1)-\frac{N+\theta}{2N+\theta}$. Since $\E(h_\beta\vert \cF_{T_\beta^-})=0$ and $h_\beta$ is independent of $W^\beta$, we conclude from Lemma \ref{9l6.1} that $E_t^{n}(\phi)$ is a martingale. By the $L^2$ maximal inequality, we get
\begin{align*} 
\E\Big(\sup_{s\leq t} E_s^{n}(\phi)\Big)&\leq  C\E([E^{n}(\phi)]_t)\nn\\
&\leq \frac{C}{N^2} \E\Big(\sum_{\beta} a_\beta^0(t)  \Big(\phi(B^\beta+W^\beta)-\phi(B^\beta)\Big)^2\Big)\nn\\
&+ \frac{C}{N^2} \E\Big(\sum_{\beta} a_\beta^0(t) \phi(B^\beta)^2 1(B^\beta+W^\beta\in \cR^{X^{n-1}}_{T_\beta^-})\Big). 
\end{align*}
Using Lemma \ref{9l2.5}, we may bound the second term by
\begin{align*} 
&\frac{C}{N} \|\phi\|_\infty^2\E\Big(\frac{1}{N}\sum_{\beta} a_\beta^0(t)  1(B^\beta+W^\beta\in \cR^{X^{0}}_{T_\beta^-})\Big)\\
&\leq  \frac{C}{N} \|\phi\|_\infty^2 C(X_0^0(1)+X_0^0(1)^2) \to 0.
\end{align*}
Turning to the first term, we use $\phi \in C_b^3$ and recall $\eta_N$ from \eqref{9ec5.37} to get 
\begin{align*}
\Big(\phi(B^\beta+W^\beta)-\phi(B^\beta)\Big)^2 \leq \eta_N^2\leq CN^{-1}.
\end{align*}
It follows that the first term is at most
\begin{align*} 
 \frac{C}{N^3} \E\Big(\sum_{\beta} a_\beta^0(t)  \Big)\leq CN^{-1} X_0^0(1) \to 0,
\end{align*}
where the inequality is by \eqref{9ec7.42}.
\end{proof}

\end{document}